\documentclass[final,leqno,onefignum,onetabnum]{siamltex1213}
\usepackage{amsfonts}
\usepackage{enumerate}

\newtheorem{remark}{{\em Remark}}[section]

\newtheorem{example}{{\em Example}}[section]

\def\mr{\mathbb{R}}
\def\mrn{\mathbb{R}^n}
\def\mrm{\mathbb{R}^m}
\def\mn{\mathbb{N}}

\def\mmf{\mathbb{F}}
\def\mf{\mathcal{F}}

\def\mb{\mathcal{B}}

\def\mh{\mathbb{H}}
\def\hh{\mathcal{H}}
\def\mS{\mathrm{S}}
\def\mT{\mathrm{T}}
\def\mss{\mathbb{S}}

\def\mt{\mathbb{T}}

\def\md{\mathbb{D}}
\def\ml{\mathbb{L}}

\def\me{\mathbb{E}}
\def\mmu{\mathcal{U}}

\def\t{\tau}
\def\mx{\mathcal{X}}
\def\dd{\mathcal{D}}
\def\ephs{\varepsilon}

\newcommand{\inner}[2]{\left\langle#1,#2\right\rangle}

\title{Pointwise second-order necessary conditions for stochastic optimal controls, Part II: The general case\thanks{This work is partially supported by the NSF
of China under grants 11231007, 11401404 and 11471231, and the Chang Jiang
Scholars Program from the Chinese Education
Ministry.}}

\author{Haisen Zhang\thanks{School of Mathematics and Statistics, Southwest University, Chongqing 400715,  China; and School of Mathematics, Sichuan University, Chengdu 610064, Sichuan Province, China. {\small\it E-mail:} {\small\tt haisenzhang@yeah.net}.}\and
Xu Zhang\thanks{School of Mathematics, Sichuan University, Chengdu 610064, Sichuan Province, China. {\small\it E-mail:} {\small\tt
zhang$\_$xu@scu.edu.cn}.}}

\begin{document}
\maketitle
\slugger{sicon}{xxxx}{xx}{x}{x--x}

\begin{abstract}
This paper is the second part of our series of work to establish pointwise second-order
necessary conditions for stochastic optimal controls. In this part, we consider the general cases, i.e., the control region is allowed to be nonconvex, and the control variable enters into both the drift and the diffusion terms of the control systems. By introducing four variational equations and four adjoint equations,
we obtain the desired necessary conditions for stochastic singular
optimal controls in the sense of Pontryagin-type maximum principle.
\end{abstract}



\pagestyle{myheadings}
\thispagestyle{plain}
\markboth{H.~Zhang and X.~Zhang}{Pointwise second-order necessary conditions}

\section{Introduction}
Let $T>0$ and $(\Omega,\mf, \mmf,$ $P)$  be a complete filtered
probability space (satisfying the usual conditions),
on which a $1$-dimensional standard Wiener
process $W(\cdot)$ is defined such that $\mmf=\{\mf_{t} \}_{0\le t\le T}$ is the natural filtration generated by $W(\cdot)$ (augmented by all of the $P$-null sets).

We consider the following controlled stochastic differential equation
\begin{equation}\label{controlsys}
\left\{
\begin{array}{l}
dx(t)=b(t,x(t),u(t))dt+\sigma(t,x(t),u(t))dW(t),\ \ \ t\in[0,T],\\
x(0)=x_0,
\end{array}\right.
\end{equation}
with a cost functional
\begin{equation}\label{costfunction}
J(u(\cdot))=\me\Bigg[\int_{0}^{T}f(t,x(t),u(t))dt+h(x(T))\Bigg].
\end{equation}
Here $u(\cdot)$ is the control variable valued in a set $U\subset \mrm$ (for some $m\in \mn$), $x(\cdot)$ is the state variable with values in $\mr^n$, and $b,\sigma:\Omega\times[0,T]\times \mrn\times U\to \mrn$  (for some $n\in \mn$), $f:\Omega\times[0,T]\times \mrn\times U\to \mr$ and $h:\Omega \times\mrn\to \mr$ are given functions (satisfying some conditions to be given later). As usual, for maps $\varphi=b,\ \sigma,\ f$, denote by $\varphi_{x} (\omega,t,x,u)$, $\varphi_{xx} (\omega,t,x,u)$, $\varphi_{xxx} (\omega,t,x,u)$ and $\varphi_{xxxx} (\omega,t,x,u)$ its first, second, third and forth order partial derivatives with respect to the variable $x$ at $(\omega,t,x,u)$, respectively. And, when the context is clear, we omit the $\omega(\in\Omega)$ argument in the defined functions.

Denote by $\mb(\mx)$ the Borel $\sigma$-field of a metric space $\mx$, and by $\mmu_{ad}$ the set of $\mf \otimes\mb([0,T])$-measurable and $\mmf$-adapted stochastic processes valued in $U$. Any $u(\cdot)\in \mmu_{ad}$ is called an admissible control.
The stochastic optimal control problem considered in this paper is to find a control
$\bar{u}(\cdot)\in\mmu_{ad}$  such that
\begin{equation}\label{minimum J}
J(\bar{u}(\cdot))=\inf_{u(\cdot)\in \mmu_{ad}}J(u(\cdot)).
\end{equation}
Any  $\bar{u}(\cdot)\in \mmu_{ad}$ satisfying (\ref{minimum J}) is called an optimal control. The corresponding state  $\bar{x}(\cdot)(=\bar{x}(\cdot;x_0,\bar{u}(\cdot)))$ to (\ref{controlsys}) is called an optimal state, and $(\bar{x}(\cdot),\bar{u}(\cdot))$ is called an optimal pair.

One of the central problems in stochastic control theory is to derive necessary conditions for the optimal pair $(\bar{x}(\cdot),\bar{u}(\cdot))$. Before analyzing this issue in detail, we recall first some elementary facts from the classical calculus. Let us consider a minimizer $x_0(\in G)$ of a smooth function $f(\cdot)$ defined on a set $G\subset \mrn$, i.e., $x_0$ satisfies
 \begin{equation}\label{E1-01}
 f(x_0)=\inf_{x\in G}f(x).
 \end{equation}
If a nonzero vector $\ell\in \mrn$ is admissible (i.e., there is a $\delta>0$ so that $x_0+s\ell\in G$ for any $s\in [0,\delta]$), then one has the following  first-order
necessary condition:
\begin{equation}\label{E101}
  0  \leq   \lim_{s\to 0^+} {f(x_0+s\ell)-f(x_0)\over
s}=\langle f_x(x_0),\ell\rangle.
\end{equation}
When $
 \langle f_x(x_0),\ell\rangle=0$
holds, i.e., (\ref{E101}) degenerates, then one can obtain further a
second-order necessary condition as follows:
\begin{equation}\label{sE102}
  0 \leq  2\lim_{s\to 0^+} {f(x_0+s\ell)-f(x_0)\over
s^2}=\langle f_{xx}(x_0)\ell,\ell\rangle.
 \end{equation}
In the particular case that $G$ is convex, by (\ref{E101}), one has
\begin{equation}\label{sE101}
  0  \leq  \langle f_x(x_0),x-x_0\rangle,\quad \forall\;x\in G.
\end{equation}
When $ f_x(x_0)=0$, then it follows from (\ref{sE102}) that
\begin{equation}\label{E102}
  0 \leq \langle f_{xx}(x_0)(x-x_0),x-x_0\rangle,\quad \forall\;x\in G.
\end{equation}
Clearly, compared to the first-order necessary condition (\ref{E101})/(\ref{sE101}), the second-order necessary condition (\ref{sE102})/(\ref{E102}) can be used to single out the possible minimizer $x_0$ from a smaller subset of $G$. From the above analysis on the minimization problem (\ref{E1-01}), it is easy to see the following:
\begin{itemize}
\item[1)]  Usually, one has to impose more regularity on the data (say $C^2$ for $f(\cdot)$) for the second-order necessary condition than that for the first-order (for which $C^1$ for $f(\cdot)$ is enough);

\item[2)] The derivation of the second-order necessary condition is probably more complicated than that of the first-order situation;

\item[3)] Usually, in order to establish the second-order necessary condition, one needs to assume that the first-order condition degenerates in some sense.

\end{itemize}
Very similar phenomenons happen when one establishes the optimality conditions for optimal control problems, though generally it turns out to be much more difficult than that for the above minimization problem.

For the moment, let us return to the deterministic optimal control problem, i.e., the functions $\sigma(\cdot)\equiv 0$, $b(\cdot)$, $f(\cdot)$, $h(\cdot)$, $x(\cdot)$ and $u(\cdot)$ in (\ref{controlsys})--(\ref{costfunction}) are independent of the sample point $\omega$.
Let $\psi(\cdot)$ be the solution to the following ordinary differential equation,
\begin{equation}\label{firstadjoint for ode}
\left\{
\begin{array}{l}
\dot{\psi}(t)=-b_{x}(t,\bar{x}(t),\bar{u}(t))^{\top}\psi(t)+f_{x}(t,\bar{x}(t),\bar{u}(t)),\quad t\in[0,T],\\
\psi(T)=-h_{x}(\bar{x}(T)).
\end{array}\right.
\end{equation}
Define the Hamiltonian
$$H(t,x,u,\psi):=\inner{\psi}{b(t,x,u)}-f(t,x,u), \qquad\forall\;(t,x,u,\psi)\in [0,T]\times\mrn\times U\times\mrn.
 $$
Then the following Pontryagin maximum principle (\cite{Pontryagin62}) holds
\begin{equation}\label{pontryagin max for ode}
H(t,\bar{x}(t),\bar{u}(t),\psi(t))=\max_{v\in U}H(t,\bar{x}(t),v,\psi(t)), \;\ a.e.\;\ t\in[0,T].
\end{equation}
The maximum condition (\ref{pontryagin max for ode}) is a first-order necessary condition for optimal controls. Suppose that, for a.e. $t\in[0,T]$ the maximization problem (\ref{pontryagin max for ode}) admits a unique solution and the optimal control $\bar{u}(\cdot)$ can be represented as a function $\Upsilon(\cdot,\cdot,\cdot)$ of $t$, $\bar{x}(t)$ and $\psi(t) $, i.e., $\bar{u}(t)=\Upsilon(t, \bar{x}(t),\psi(t))$ satisfies
\begin{equation}\label{repersent optimal ode}
\qquad H(t,\bar{x}(t),\Upsilon(t, \bar{x}(t),\psi(t)),\psi(t))=\max_{v\in U}H(t,\bar{x}(t),v,\psi(t)), \quad a.e.\ t\in[0,T].
\end{equation}
Then, substituting $\Upsilon$ into the control system (\ref{controlsys}) (with $\sigma\equiv 0$) and the adjoint equation (\ref{firstadjoint for ode}), we obtain the following two-point boundary-value problem:
\begin{equation}\label{twopiont boundaryvalue for ode}
\left\{
\begin{array}{l}
\dot{\bar{x}}(t)=H_{\psi}(t,\bar{x}(t),\Upsilon(t, \bar{x}(t),\psi(t)),\psi(t)),\quad t\in[0,T],\\[2mm]
\dot{\psi}(t)=-H_{x}(t,\bar{x}(t),\Upsilon(t, \bar{x}(t),\psi(t)),\psi(t)),\quad t\in[0,T],\\[2mm]
\bar x(0)=x_0,\ \psi(T)=-h_{x}(\bar{x}(T)).
\end{array}\right.
\end{equation}
If both the original optimal control problem and the  two-point boundary-value problem (\ref{twopiont boundaryvalue for ode}) admit unique solutions, then $\bar{u}(\cdot)=\Upsilon(\cdot, \bar{x}(\cdot),\psi(\cdot))$ is the solution to the original optimal control problem (\ref{minimum J}) where $(\bar{x}(\cdot),\psi(\cdot))$ is the solution to (\ref{twopiont boundaryvalue for ode}). It is easy to see that, the uniqueness of the solution to the maximization problem (\ref{pontryagin max for ode}) (in the first-order necessary condition) plays an important role to reformulate the  original optimal control problem into the two-point boundary-value problem (\ref{twopiont boundaryvalue for ode}). When this maximization problem admits multi-solutions, the first-order necessary condition is not enough to determine the optimal controls. Indeed, in this cases, the solution map for the maximization problem becomes a set-valued map. When substituting this set-valued map into the control system and the adjoint equation,  one obtains a differential inclusion problem, which is usually very hard to solve. When this happens, as in the classical calculus, it is quite useful to analyze further the second-order (or even higher-order) necessary conditions for optimal controls. In the case of deterministic control problems (even in finite dimensions), there are many works devoted to this topic (See \cite{fh, Frankowska13, Gabasov72, Goh66, h, Lou10, pz, w} and the rich references therein), especially one can find several interesting monographs (\cite{BellJa75, CA78, Gabasov73, Knobloc}) in this direction.

Naturally, one expects to establish the stochastic maximum principle for the optimal control problem (\ref{minimum J}). We refer to \cite{Bensoussan81, Bismut78, Haussmann76, Kushner72} and references cited therein for some early works in this respect. Since in this case the It\^{o} integral appears in the control system (\ref{controlsys}), things became much complicated. First, quite different from the equation (\ref{firstadjoint for ode}), the adjoint equations in the stochastic cases (called backward stochastic differential equations, BSDEs for short) have two unknowns. Second, when the control region is nonconvex, the needle variation, which is essential a perturbation of the optimal control on a measurable set with small measure, is used to derive the optimality conditions. When the diffusion term $\sigma$ contains the control variable $u$, the state increment is an infinitesimal of order $1/2$ with respect to $\varepsilon$ ($\varepsilon\to 0^+$) (when the optimal control is perturbed with respect to the time variable on a measurable set with Lebesgue measure $\varepsilon$). Therefore, to obtain the first-order necessary condition for optimal controls for the general case, the cost functional needs to be expanded up to the second order, and two variational equations and two adjoint equations need to be introduced  (See \cite{Peng90}). More precisely, define the Hamiltonian $\hh$ by
\begin{equation}\label{Hamiltonian}
\begin{array}{ll}
\qquad\hh(\omega,t,x,u, y_{1},z_{1})=\inner{y_{1}}{b(\omega,t,x,u)}
+\inner{z_{1}}{\sigma(\omega,t,x,u)}-f(\omega,t,x,u),\\[1mm]
\qquad\qquad\qquad\qquad\qquad\quad\forall\;
(\omega,t,x,u,y_{1},z_{1})\in \Omega\times[0,T]\times\mrn\times U\times\mrn\times\mrn.
\end{array}
\end{equation}
Let $(p_{1}(\cdot),q_{1}(\cdot))$ and $(p_{2}(\cdot),q_{2}(\cdot))$ be respectively solutions to the following first- and second-order adjoint equations,
\begin{equation}\label{firstajointequ}
\qquad\ \left\{
\begin{array}{l}
dp_{1}(t)=-\Big[b_{x}(t)^{\top}p_{1}(t)+\sigma_{x}(t)^{\top}q_{1}(t)
-f_{x}(t)\Big]dt+q_{1}(t)dW(t),
 \  t\in[0,T], \\[+0.5em]
p_{1}(T)=-h_{x}(\bar{x}(T))
\end{array}\right.
\end{equation}
and
\begin{equation}\label{secondajointequ}
\qquad\left\{
\begin{array}{l}
dp_{2}(t)=-\Big[b_{x}(t)^{\top}p_{2}(t)+p_{2}(t)b_{x}(t) +\sigma_{x}(t)^{\top}p_{2}(t)\sigma_{x}(t) +\sigma_{x}(t)^{\top}q_{2}(t)\\[+0.5em]
\qquad\qquad \qquad\qquad\qquad
+q_{2}(t)\sigma_{x}(t)+\hh_{xx}(t)\Big]dt+q_{2}(t)dW(t),\  t\in[0,T], \\[+0.5em]
p_{2}(T)=-h_{xx}(\bar{x}(T)),
\end{array}\right.
\end{equation}
where $b_{x}(t)\!=\!b_{x}(t,\bar{x}(t),\bar{u}(t))$, $\sigma_{x}(t)\!=\!\sigma_{x}(t,\bar{x}(t),\bar{u}(t))$, $f_{x}(t)\!=\!f_{x}(t,\bar{x}(t),\bar{u}(t))$,
$\hh_{xx}(t)\!=\!\hh_{xx}(t, \bar{x}(t),\bar{u}(t), p_{1}(t),q_{1}(t))$. The following first-order necessary condition for the optimal pair $(\bar{x}(\cdot),\bar{u}(\cdot))$  is established in \cite{Peng90}:
\begin{equation}\label{pengs firstorder condition}
\mh(t,\bar{x}(t),v)\le 0,\qquad \forall\ v\in U, \ a.e.\ (\omega,t)\in \Omega\times [0,T].
\end{equation}
where
$$
\begin{array}{ll}
&\mh(\omega,t,x,u)\\[2mm]
&=\hh(\omega,t,x,u, p_{1}(t),q_{1}(t))
-\hh(\omega,t,x,\bar{u}(t), p_{1}(t),q_{1}(t))\nonumber\\[2mm]
&\quad+\frac{1}{2}\inner{p_{2}(t)\big(\sigma(\omega,t,x,u)-\sigma(\omega,t,x,\bar{u}(t))\big)}
{\sigma(\omega,t,x,u)-\sigma(\omega,t,x,\bar{u}(t))},\\[2mm]
&\quad\qquad\qquad \qquad\qquad \qquad\qquad \qquad\qquad
(\omega,t,x,u)\in \Omega\times[0,T]\times\mrn\times U.
\end{array}
 $$
Similar to the above, if the optimal control $\bar{u}(\cdot)$ can be represented as a function $\Psi(\cdot,\cdot,\cdot,\cdot,\cdot,\cdot)$ of $(\omega,t,\bar{x},p_{1},q_{1},$ $p_{2}) $ using the condition (\ref{pengs firstorder condition}) (i.e., $\bar{u}(\omega,t)=\Psi(\omega,t,\bar{x}(t),$ $p_{1}(t),q_{1}(t),p_{2}(t))$) (Note that $q_{2}$ does not appear explicitly in the definition of $\mh$), then the optimal control problem can be closely related to the following fully coupled forward backward stochastic differential  equation (FBSDE, in short):
\begin{equation}\label{coupled FBSDE}
\qquad\;\left\{
\begin{array}{l}
d\bar{x}(t)=\hat{b}(t)dt
+\hat{\sigma}(t)dW(t),\ t\in[0,T],\\[2mm]
dp_{1}(t)=-\Big[\hat{b}_{x}(t)^{\top}p_{1}(t)+\hat{\sigma}_{x}(t)^{\top}q_{1}(t)
-\hat{f}_{x}(t)\Big]dt+q_{1}(t)dW(t),\ t\in[0,T],\\[2mm]
dp_{2}(t)=-\Big[\hat{b}_{x}(t)^{\top}p_{2}(t)+p_{2}(t)\hat{b}_{x}(t) +\hat{\sigma}_{x}(t)^{\top}p_{2}(t)\hat{\sigma}_{x}(t) +\hat{\sigma}_{x}(t)^{\top}q_{2}(t)\\[2mm]
\qquad\qquad
+q_{2}(t)\hat{\sigma}_{x}(t)+\hat{\hh}_{xx}(t)\Big]dt+q_{2}(t)dW(t),\ t\in[0,T],\\[2mm]
x(0)=x_0,\ p_{1}(T)=-h_{x}(\bar{x}(T)),\ p_{2}(T)=-h_{xx}(\bar{x}(T)).
\end{array}\right.
\end{equation}
where $\hat{b}(t)=b(t,\bar{x}(t),\Psi(t,\bar{x}(t),p_{1}(t),q_{1}(t),p_{2}(t)))$,
$\hat{\sigma}(t)=\sigma(t,\bar{x}(t),\Psi(t,\bar{x}(t),$ $p_{1}(t),$ $q_{1}(t),p_{2}(t)))$,
$\hat{\hh}_{xx}(t)\!=\!\hh_{xx}(t,\bar{x}(t),
\!\Psi(t,\bar{x}(t),p_{1}(t),q_{1}(t),p_{2}(t))$,
similar for $\hat{b}_{x}(t)$, $\hat{\sigma}_{x}(t)$, and $\hat{f}_{x}(t)$.
For some more discussions about FBSDEs,  we refer to \cite{MaYong99}.

However, exactly as the deterministic case, the first-order necessary condition is not always effectively to find the stochastic optimal controls. In the preceding discussion, the uniqueness of the solution to (\ref{pengs firstorder condition}) plays an important role to reduce the original optimal control problem to the FBSDE (\ref{coupled FBSDE}). When the problem (\ref{pengs firstorder condition}) admits multi-solutions, one needs to establish suitable second-order necessary condition for optimal controls as an effective supplement to the first-order condition. As we mentioned before, there exist many works addressing to the corresponding deterministic problems. However, in the stochastic setting, there are only two articles \cite{Bonnans12} and \cite{Tang10} available before our work \cite{zhangH14a}. When the diffusion terms do not contain the control variable, Tang \cite{Tang10} derived a pointwise second-order maximum principle for stochastic optimal controls, for which the control regions are allowed to be nonconvex. When the diffusion terms contain the control variable, Bonnans and Silva \cite{Bonnans12} established some integral-type (rather than pointwise) second-order necessary conditions for stochastic optimal controls with convex control constrains. In \cite{zhangH14a}, we found that, quite different from the
deterministic setting, there exist some essential difficulties in deriving the {\em pointwise}
second-order necessary condition from an integral-type one whenever the diffusion terms
contain the control variable, even for the special case of convex control
constraints, and obtained a positive result for this case under some assumptions
in terms of the Malliavin calculus.

The main purpose of this paper is to establish some pointwise second-order necessary conditions for stochastic optimal controls in the general cases, i.e., the control regions are allowed to be nonconvex and both the drift and diffusion terms contain the control variable. Stimulated by \cite{Peng90}, it is easy to see that, in order to obtain the second-order optimality condition for the general case, one needs to expand the cost functional up to the forth order, and introduce four variational equations and four adjoint equations. This is the main difference between the present paper and the previous related works (i.e., \cite{Bonnans12, Tang10, zhangH14a}). On the other hand, the solutions of the variational equations appear in the second-order terms (in the sense of the perturbation measure) of the variational formulation with respect to the optimal controls, and it seems to us that, they cannot be eliminated by introducing new adjoint equations. When the diffusion terms of the control systems contain the control variables, similar to the convex control constraint cases, the Lebesgue differentiation theorem cannot be used directly to derive the pointwise second-order necessary condition from the variational formulation (See \cite[Subsection 3.2]{zhangH14a} for a detailed explanation). This is another difference between this paper and \cite{Tang10} addressing to the case of the diffusion term independent of the control variable. In this paper, first we establish a variational formulation of (\ref{minimum J}) with respect to the optimal controls. Then, using this variational formulation and the martingale representation theorem, we derive a second-order necessary condition for stochastic optimal controls. Further, under some conditions, we refine this result and obtain a pointwise second-order necessary condition. Note that the analysis in this paper is much complicated than that in \cite{zhangH14a} though some of the ideas and techniques are the same in these two papers.

The rest of this paper is organized as follows. In Section 2, we collect some
notation and concepts. In Section 3, we introduce the related variational equations and adjoint equations. In Section 4, we state
the main results of this paper and present some remarks and examples. Section 5 is devoted to proving our main results. Finally, the proofs of two technical results are given in Appendixes A and B, respectively.

Partial results in this paper have been announced in \cite{ZhangSCM15} without proofs.

\section{Preliminaries}

Let $m,\ n,\ d,\ h,\ l\in \mn$. Denote by $\inner{\cdot}{\cdot}$ and $|\cdot|$ respectively the inner product and norm in $\mrn$ or $\mrm$, which can be identified from the contexts.
For any $\alpha,\beta\in [1,+\infty)$, denote by $L_{\mf_{T}}^{\beta}(\Omega; \mrn)$ the space of $\mf_{T}$-measurable random variables $\xi$ such that $\me~|\xi|^{\beta}<+\infty$, by $L^{\beta}(\Omega\times[0,T]; \mrn)$ the space of $\mf\otimes\mb([0,T])$-measurable processes $\varphi$ such that $\|\varphi\|_{\beta}:=\big[\me\int_{0}^{T}|\varphi(t)|^{\beta}dt
\big]^{\frac{1}{\beta}} <+\infty$,
by $L_{\mmf}^{\beta}(\Omega; L^{\alpha}(0,T; \mrn))$ the space of $\mf\otimes\mb([0,T])$-measurable, $\mmf$-adapted processes $\varphi$ such that $\|\varphi\|_{\alpha,\beta}:=\big[\me~\big(\int_{0}^{T}|\varphi(t)|^{\alpha}dt\big)
^{\frac{\beta}{\alpha}}\big]^{\frac{1}{\beta}} <+\infty$,
by $L_{\mmf}^{\beta}(\Omega; C([0,T]; \mrn))$ the space of $\mf\otimes\mb([0,T])$-measurable, $\mmf$-adapted  continuous processes $\varphi$  such that $\|\varphi\|_{\infty,\beta}:=
\big[\me~\big(\sup_{t\in[0,T]}|\varphi(t)|^{\beta}\big)\big]^{\frac{1}{\beta}} <+\infty$, by $L^{\infty}(\Omega\times[0,T]; \mrn)$ the space of $\mf\otimes\mb([0,T])$-measurable processes $\varphi$ such that $\|\varphi\|_{\infty}:=\mbox{ess sup}_{(\omega,t)\in \Omega\times[0,T]}|\varphi(\omega,t)| <+\infty $,
and by
$L^{\beta}(0,T;  L_{\mmf}^{\beta}(\Omega\times[0,T]; \mrn))$ the $\mf\otimes \mb([0,T]\times[0,T])$ measurable maps $\varphi$ such that for any $t\in[0,T]$, $\varphi(\cdot,t)$ is $\mmf$-adapted and $\|\varphi\|_{\beta}
:=\big[\me\int_{0}^{T}\int_{0}^{T}|\varphi(s,t)|^{\beta}dsdt
\big]^{\frac{1}{\beta}}<+\infty$.

Let $\md^{1,2}(\mrn)\subset L_{\mf_{T}}^{2}(\Omega; \mrn)$ be the space of Malliavin differentiable random variables, and for any $\xi\in \md^{1,2}(\mrn)$ denote by $\dd_{\cdot}\xi$ its Malliavin derivative. Denote by $\ml^{1,2}(\mrn)$ the subspace of $L^{2}(\Omega\times [0,T];\mrn))$ whose elements satisfy the following conditions.
\begin{enumerate}[{\rm (i)}]
  \item For almost every $t\in[0,T]$, $\varphi(t,\cdot)\in \md^{1,2}(\mrn)$,
  \item $(\omega, t, s)\to D_{s}\varphi(t, \omega)$ admits an $\mf\otimes\mb([0,T]\times [0,T])$-measurable version, and
  \item $|||\varphi|||_{1,2}:=\big[\me\int_{0}^{T}|\varphi(t)|^2dt
      +\me\int_{0}^{T}\int_{0}^{T}|D_{s}\varphi(t)|^2dsdt\big]^{\frac{1}{2}}<+\infty,$
\end{enumerate}
where $\dd_{\cdot}\varphi(t,\cdot)$ is the Malliavin derivative of the random variable $\varphi(t,\cdot)$.
Denote by $\ml_{2}^{1,2}(\mrn)$ the subspace of $L^{2}(\Omega\times [0,T];\mrn))$  whose elements are Malliavin differentiable almost everywhere and their Malliavin derivatives have suitable continuity. More precisely, write
\begin{eqnarray*}
\ml_{2^+}^{1,2}(\mrn)&&:=\Big\{\varphi(\cdot)\in\ml^{1,2}(\mrn)\;\Big|\; \exists\ \nabla^+\varphi(\cdot)\in L^2(\Omega\times[0,T];\mrn)\ \mbox{such that}\\
& & f_{\ephs}(s):=\sup_{s<t<(s+\varepsilon)\wedge T}
\me~\big|\dd_{s}\varphi(t)-\nabla^+\varphi(s)\big|^2<+\infty,\ a.e.\ s\in [0,T],\\
& & f_{\ephs}(\cdot)\ \mbox{is measurable on } [0,T]\mbox{ for any }\ephs>0,\ \mbox{and}\ \lim_{\varepsilon\to 0^+}\int_{0}^{T}f_{\ephs}(s)ds=0\Big\},\\
\ml_{2^-}^{1,2}(\mrn)&&:=\Big\{\varphi(\cdot)\in\ml^{1,2}(\mrn)\;\Big|\; \exists\ \nabla^-\varphi(\cdot)\in L^2(\Omega\times[0,T];\mrn)\ \mbox{such that}\\
& & g_{\ephs}(s):=\sup_{(s-\varepsilon)\vee 0<t<s}
\me~\big|\dd_{s}\varphi(t)-\nabla^-\varphi(s)\big|^2<+\infty,\ a.e.\ s\in [0,T],\\
& & g_{\ephs}(\cdot)\ \mbox{is measurable on } [0,T] \mbox{ for any }\ephs>0,\ \mbox{and}\ \lim_{\varepsilon\to 0^+}\int_{0}^{T}g_{\ephs}(s)ds=0\Big\}.
\end{eqnarray*}
Denote
$\ml_{2}^{1,2}(\mrn)=\ml_{2^+}^{1,2}(\mrn)\cap\ml_{2^-}^{1,2}(\mrn)$.
For any $\varphi(\cdot)\in \ml_{2}^{1,2}(\mrn)$,  denote
$\nabla\varphi(\cdot)=\nabla^{+}\varphi(\cdot)+\nabla^{-}\varphi(\cdot)$.
When $\varphi$ is $\mmf$-adapted, $\dd_{s}\varphi(t)=0$ for any $t<s$. In this case, $\nabla^{-}\varphi(\cdot)=0$, and $\nabla\varphi(\cdot)=\nabla^{+}\varphi(\cdot)$. Denote by $\ml_{2,\mmf}^{1,2}(\mrn)$ the set of all $\mmf$-adapted processes in $\ml_{2}^{1,2}(\mrn)$. We refer to \cite{Nualart06} for more materials on this topic.

Denote by $L(\mathop\Pi\limits_{d}\mrn; \mrm)$ the $d$-linear maps from $\underbrace{\mrn\times\ldots\times\mrn}_{d}$ to $\mrm$.
Let $\{e_{1},\ldots,$ $ e_{n}\}$ be the standard basis of $\mrn$, $\{\mathfrak{e}_{1},\ldots, \mathfrak{e}_{m}\}$ be the standard basis of $\mrm$.
Any $\Lambda$ in $L(\mathop\Pi\limits_{d}\mrn; \mrm)$ is uniquely determined by the numbers $$\lambda^{ij_{1},\ldots, j_{d}}:=\inner{\Lambda(e_{j_{1}},\ldots, e_{j_{d}})}{\mathfrak{e}_{i}}, \quad i=1,\ldots, m,\ j_{k}=1,\ldots, n,\ k=1,\ldots, d.$$
We define the norm of $\Lambda$ by
$$|\Lambda|:=\Bigg[\sum_{i,j_{k},k=1,\ldots,d}\Big(\lambda^{ij_{1},\ldots, j_{d}}\Big)^{2}\Bigg]^{\frac{1}{2}}.$$
Let $\Gamma\in L(\mathop\Pi\limits_{h}\mrn; \mrn)$ and $\Theta\in L(\mathop\Pi\limits_{l}\mrn; \mrn)$. We denote by $\Lambda\mathop\circ\limits_{i}\Gamma$ the composition of $\Lambda$ with $\Gamma$ at the $i$th position ($i=1,\cdots,d$), i.e.,
\begin{eqnarray*}
\Lambda\mathop\circ\limits_{i}\Gamma(x_{1},\ldots,x_{d+h-1})
&=&\Lambda(x_{1},\ldots,\mathop{\Gamma}_{i}(x_{i},\ldots,x_{i+h-1}),\ldots,x_{d+h-1}),\\
& &\qquad\qquad\qquad\qquad\qquad
\quad x_{k}\in \mrn,\ k=1,\ldots, d\!+\!h\!-\!1,
\end{eqnarray*}
and we denote by $\Lambda\mathop\circ\limits_{i,j}(\Gamma,\Theta)$ the composition of $\Lambda$ with $\Gamma$ and $\Theta$ at the $i$th and the $j$th positions ($i\not=j,\; i,j=1,\cdots,d$), i.e.,
\begin{eqnarray*}
& &\Lambda\mathop\circ\limits_{i,j}(\Gamma,\Theta)(x_{1},\ldots,x_{d+h+l-2})\\
&=&\Lambda(x_{1},\ldots,\mathop{\Gamma}_{i}(x_{i},\ldots,x_{i+h-1}),\ldots,
\mathop{\Theta}_{j}(x_{j+h-1},\ldots,x_{j+h+l-2}), \ldots,x_{d+h+l-2}),\\
& &\qquad\qquad\qquad\qquad\qquad\qquad\qquad\qquad\quad
\quad x_{k}\in \mrn,\ k=1,\ldots, d\!+\!h\!+\!l\!-\!2.
\end{eqnarray*}

In a similar way, if $y,z\in \mrn$, we denote
\begin{eqnarray*}
& &\Lambda\mathop\bullet\limits_{i}y(x_{1},\ldots,x_{d-1})
=\Lambda(x_{1},\ldots,\mathop{y}_{i},\ldots,x_{d-1}),\ \ x_{k}\in \mrn,\ k\!=\!1,\ldots, d\!-\!1,\\
& &\Lambda\mathop\bullet\limits_{i,j}(y,z)(x_{1},\ldots,x_{d-2})
=\Lambda(x_{1},\ldots,\mathop{y}_{i},\ldots,
\mathop{z}_{j}, \ldots,x_{d-2}), \ x_{k}\in \mrn,\ k\!=\!1,\ldots, d\!-\!2.\\
\end{eqnarray*}

Denote by $L_{\mmf}^{2}(\Omega;L^{1}(0,T;L(\mathop\Pi\limits_{d}\mrn; \mr))$ the space of $\mf\otimes\mb([0,T])/\mb(L(\mathop\Pi\limits_{d}\mrn; \mr))$-measurable processes $\varphi$ such that $\|\varphi\|_{1,2}:=\big[\me~\big(\int_{0}^{T}|\varphi(t)|dt\big)
^{2}\big]^{\frac{1}{2}} <+\infty$ and by
$L_{\mmf}^{2}(\Omega;L^{2}(0,T;L(\mathop\Pi\limits_{d}\mrn; \mr))$ the space of $\mf\otimes\mb([0,T])/\mb(L(\mathop\Pi\limits_{d}\mrn; \mr))$-measurable processes $\varphi$ such that $\|\varphi\|_{2}:=\big[\me\int_{0}^{T}|\varphi(t)|^2dt
\big]^{\frac{1}{2}} <+\infty$.
We give below an It\^{o} formula for  multi-linear function-valued stochastic processes, which is an easy extension of the classical It\^{o} formula (Hence we omit its proof).
\begin{lemma}\label{o formu}
Let $P(\cdot)$ be an $L(\mathop\Pi\limits_{d}\mrn; \mr)$-valued process of the form
$$P(t)=P_{0}+\int_{0}^{t}A(s)ds+\int_{0}^{t}B(s)dW(s),\qquad t\in [0,T]$$
where $A(\cdot)\in L_{\mmf}^{2}(\Omega;L^{1}(0,T;L(\mathop\Pi\limits_{d}\mrn; \mr))$, $B(\cdot)\in L_{\mmf}^{2}(\Omega;L^{2}(0,T;L(\mathop\Pi\limits_{d}\mrn; \mr))$, and let $x(\cdot)$ be an $\mrn$-valued process such that
$$x(t)=x_{0}+\int_{0}^{t}f(s)ds+\int_{0}^{t}g(s)dW(s),\qquad t\in [0,T],$$
where $f(\cdot)\in L_{\mmf}^{2}(\Omega;L^{1}(0,T;\mrn)$, $g(\cdot)\in L_{\mmf}^{2}(\Omega;L^{2}(0,T;\mrn)$.
Then the following It\^{o} formula holds.
\begin{eqnarray}\label{Ito formulation}
& &P(T)\big(x(T),\ldots,x(T)\big)-P_{0}\big(x_{0},\ldots,x_{0}\big)\nonumber\\
&=&\int_{0}^{T}\Big[\sum_{i=1}^{d}P(t)\big(x(t),\ldots,\mathop {f(t)}\limits_{i}, \ldots, x(t)\big)+A(t)\big(x(t),\ldots,x(t)\big)\nonumber\\
& &\qquad\quad
+\sum_{i=1}^{d}B(t)\big(x(t),\ldots,\mathop {g(t)}\limits_{i}, \ldots, x(t)\big)\nonumber\\
& &\qquad\quad
+\sum_{i=1}^{d}\sum_{j=i+1}^{d}P(t)\big(x(t),\ldots,\mathop {g(t)}\limits_{i},\ldots,\mathop {g(t)}\limits_{j}, \ldots, x(t)\big)
\Big]dt\nonumber\\
& &+\int_{0}^{T}\Big[B(t)\big(x(t),\ldots,x(t)\big)+\sum_{i=1}^{d}P(t)\big(x(t),\ldots,\mathop {g(t)}\limits_{i}, \ldots, x(t)\big)
\Big]dW(t).
\end{eqnarray}
\end{lemma}

\section{Variational formulations}
In this section, we establish a second-order (with respect to the perturbation measure) Taylor expansion of the cost function at the optimal control $\bar{u}(\cdot)$. Firstly, we recall some known estimates for stochastic differential equations.
\begin{lemma}\label{estimatelinearsde} {\rm (\cite[Proposition 2.1]{Yong07})}
Suppose that there exists a constant $L > 0$ such that for $\varphi = b,\ \sigma$ and any $x,\ \tilde{x}\in\mrn$,  $u\in U$,
\begin{equation}\label{standard assumption}
\left\{
\begin{array}{l}
|\varphi(t,x, u)-\varphi(t,\tilde{x}, u)|\le L|x-\tilde{x}|,
\  a.s.\ a.e.\ t\in [0,T],\\
|\varphi(t,0, u)|\le L,\ a.s.\ a.e.\ t\in [0,T].\\
\end{array}\right.
\end{equation}
Then for any $\beta\ge 1$, $u(\cdot)\in \mmu_{ad}$ and initial datum $x_{0}\in \mrn$, the state equation (\ref{controlsys}) admits a unique solution
$x(\cdot) \in L_{\mmf}^{\beta}(\Omega; C([0; T]; \mrn))$, and for some constant $C=C(\beta,L,T)>0$ the following estimate holds:
\begin{eqnarray}\label{estimateofx}
& &\me\Big[\sup_{s\in[0,t]}|x(s)|^{\beta}\Big]\nonumber\\
&\le& C\me~\Bigg[|x_{0}|^{\beta}
+\Big(\int_{0}^{t}|b(s,0,u(s))|ds\Big)^{\beta}
+\Big(\int_{0}^{t}|\sigma(s,0,u(s))|^{2}ds\Big)^{\frac{\beta}{2}}\Bigg].
\end{eqnarray}
Further, if $\bar{x}(\cdot)$ is the unique solution corresponding to
$(\bar{x}_{0}, \bar{u}(\cdot))\in \mrn\times \mmu_{ad}$, then
\begin{eqnarray}\label{continuousofxwithx0andu}
& &\me\Big[\sup_{s\in[0,t]}|x(s)-\bar{x}(s)|^{\beta}\Big]\nonumber\\
&\le&  C\me~\Bigg[|x_{0}-\bar{x}_{0}|^{\beta}+
\Big(\int_{0}^{t}\big|b(s,\bar{x}(s),u(s))
-b(s,\bar{x}(s),\bar{u}(s))\big|ds\Big)^{\beta}\nonumber\\
& &\qquad\qquad\qquad\quad+\Big(\int_{0}^{t}\big|\sigma(s,\bar{x}(s),u(s))
-\sigma(s,\bar{x}(s),\bar{u}(s))\big|^{2}ds
\Big)^{\frac{\beta}{2}}\Bigg].
\end{eqnarray}
\end{lemma}

In what follows, we assume that
\begin{enumerate}
    \item [(C1)] The control region $U \subset \mrm$ is nonempty and bounded.
    \item [(C2)] Functions $b$, $\sigma$, $f$, and $h$ satisfy
       \begin{enumerate}[{\rm (i)}]
       \item For any $(x, u)\in \mrn\times U$, the stochastic processes
             $b(\cdot, x, u):\ [0,T]\times\Omega\to \mrn$, $\sigma(\cdot, x, u):\ \Omega\times [0,T]\to \mrn$ and $f(\cdot, x, u):\ \Omega\times [0,T]\to \mr$
             are $\mf\otimes\mb([0,T])$-measurable and $\mmf$-adapted. $h(x,\cdot):\Omega\to \mr$ is $\mf_{T}$-measurable.
       \item For almost all $(\omega, t)\in \Omega\times[0, T]$ and any $u\in U$, the map
             $x\mapsto (b(\omega,t, x, u), \sigma(\omega,t, x, u), f(\omega,t, x, u))$
             is continuously differentiable up to the forth order, and
             there exist a constant $L > 0$ and a modulus of continuity $\tilde{\omega}: [0,\infty) \to [0, \infty)$ such that for a.e. $ (\omega,t)\in \Omega\times [0,T]$, and all $x,\ \tilde{x}\in\mrn$, $u,\ \tilde{u}\in U$, $\varphi = b,\ \sigma,\ f$,
             $$\left\{
             \begin{array}{l}
             |\varphi(t,0,0)|\le L,\
             |\varphi(t,x, u)-\varphi(t,\tilde{x}, \tilde{u})|\le L|x-\tilde{x}|+ \tilde{\omega}(|u-\tilde{u}|),\\
             |\varphi_{x}(t,x, u)-\varphi_{x}(t,\tilde{x}, \tilde{u})|
             \le L|x-\tilde{x}|+ \tilde{\omega}(|u-\tilde{u}|),\\
             |\varphi_{xx}(t,x, u)-\varphi_{xx}(t,\tilde{x}, \tilde{u})|
             \le L|x-\tilde{x}|+ \tilde{\omega}(|u-\tilde{u}|),\\
             |\varphi_{xxx}(t,x, u)-\varphi_{xxx}(t,\tilde{x}, \tilde{u})|
             \le L|x-\tilde{x}|+ \tilde{\omega}(|u-\tilde{u}|),\\
             |\varphi_{xxxx}(t,x, u)-\varphi_{xxxx}(t,\tilde{x}, \tilde{u})|\le L|x-\tilde{x}|+\tilde{\omega}(|u-\tilde{u}|).\\
             \end{array}\right.
             $$
       \item $h(\cdot)$ is continuously differentiable up to the forth order (a.s.), and there exists a constant $L > 0$ such that for any $x\in \mrn$,
           \begin{eqnarray*}
           & &|h(x)|\le L(1+|x|^{4}), \quad\ \ \ |h_{x}(x)|\le L(1+|x|^3),\quad \\
           & &|h_{xx}(x)|\le L(1+|x|^{2}), \quad |h_{xxx}(x)|\le L(1+|x|),\quad |h_{xxxx}(x)|\le L,\ \ a.s.
           \end{eqnarray*}
           \end{enumerate}
\end{enumerate}
Obviously, for $\varphi = b,\ \sigma,\ f$, when (C1)--(C2) are satisfied, for a.e. $(\omega,t)\in \Omega\times[0,T]$,
$$|\varphi_{x}(t,x, u)|+ |\varphi_{xx}(t,x, u)|+|\varphi_{xxx}(t,x, u)|+
|\varphi_{xxxx}(t,x, u)|\le L,$$
for all $(x,u)\in \mrn\times U$, and the controlled stochastic differential equation (\ref{controlsys}) admits a unique solution for any $u(\cdot)\in \mmu_{ad}$ and the cost functional is well-defined.

Let $(\bar{x}(\cdot),\bar{u}(\cdot))$ be an optimal pair, $u(\cdot)\in \mathcal{U}_{ad}$ be an admissible control, $E_{\varepsilon} \subset [0,T]$ be a measurable set with measure $|E_{\varepsilon}|=\varepsilon$ for a given $\varepsilon\in (0,T)$. Define
$$
u^{\varepsilon}(t)=\left\{
\begin{array}{l}
u(t), \qquad\qquad t\in E_\varepsilon,\\
\bar{u}(t), \qquad \qquad t\in [0,T] \setminus E_{\varepsilon}.\\
\end{array}\right.
$$
Let $x^{\varepsilon}(\cdot)$ be the state with respect to  the control
$u^{\varepsilon}(\cdot)$ and let $\delta x(\cdot)=x^{\varepsilon}(\cdot)-\bar{x}(\cdot)$. For $\varphi=b,\ \sigma,\ f$, write
$\varphi_{x}(t)=\varphi_{x}(t,\bar{x}(t),\bar{u}(t))$,
$\varphi_{xx}(t)=\varphi_{xx}(t,\bar{x}(t),\bar{u}(t))$,
$\varphi_{xxx}(t)=\varphi_{xxx}(t,\bar{x}(t),\bar{u}(t))$,
$\varphi_{xxxx}(t)=\varphi_{xxxx}(t,\bar{x}(t),\bar{u}(t))$,
and put
$$
\begin{array}{ll}
\delta \varphi(t)=\varphi(t,\bar{x}(t),u(t))-\varphi(t,\bar{x}(t),\bar{u}(t)),\\
\delta \varphi_{x}(t)
=\varphi_{x}(t,\bar{x}(t),u(t))-\varphi_{x}(t,\bar{x}(t),\bar{u}(t)),\\
\delta \varphi_{xx}(t)
=\varphi_{xx}(t,\bar{x}(t),u(t))-\varphi_{xx}(t,\bar{x}(t),\bar{u}(t)),\\
\delta\varphi_{xxx}(t)
=\varphi_{xxx}(t,\bar{x}(t),u(t))-\varphi_{xxx}(t,\bar{x}(t),\bar{u}(t)).
\end{array}
$$
Now, we introduce the following four variational equations:
\begin{equation}\label{firstvariequ}
\left\{
\begin{array}{l}
dy_{1}^{\varepsilon}(t)= b_{x}(t) y_{1}^{\varepsilon}(t)dt
+\Big[\sigma_{x}(t) y_{1}^{\varepsilon}(t)+ \delta\sigma(t)\chi_{E_{\varepsilon}}(t)\Big]dW(t),\quad t\in [0,T],  \\
y_{1}^{\varepsilon}(0)=0;
\end{array}\right.
\end{equation}
\begin{equation}\label{secondvariequ}
\qquad\left\{
\begin{array}{l}
dy_{2}^{\varepsilon}(t)= \Big[b_{x}(t)
y_{2}^{\varepsilon}(t)+\frac{1}{2}b_{xx}(t)
\big(y_{1}^{\varepsilon}(t),y_{1}^{\varepsilon}(t)\big)
+\delta b(t)\chi_{E_{\varepsilon}}(t)\Big]dt\\[+0.8em]
\qquad\qquad
+\Big[\sigma_{x}(t) y_{2}^{\varepsilon}(t)
+\frac{1}{2}\sigma_{xx}(t)\big(y_{1}^{\varepsilon}(t),y_{1}^{\varepsilon}(t)\big)\\[+0.8em]
\qquad\qquad\qquad\qquad\qquad\qquad\qquad\ \ \ \
+\delta\sigma_{x}(t)y_{1}^{\varepsilon}(t)\chi_{E_{\varepsilon}}(t)\Big]dW(t),\ t\in [0,T], \\[+0.8em]
y_{2}(0)=0;
\end{array}\right.
\end{equation}
\begin{equation}\label{thirdvariequ}
\quad\quad \left\{
\begin{array}{l}
dy_{3}^{\varepsilon}(t)= \Big\{b_{x}(t)y_{3}^{\varepsilon}(t)
+\frac{1}{2}\big[
b_{xx}(t)\big(y_{1}^{\varepsilon}(t)+y_{2}^{\varepsilon}(t), y_{2}^{\varepsilon}(t)\big)\\[+0.8em]
\qquad\qquad
+b_{xx}(t)\big(y_{2}^{\varepsilon}(t),y_{1}^{\varepsilon}(t)\big) \big]
+\frac{1}{6}b_{xxx}(t)
\big(y_{1}^{\varepsilon}(t),y_{1}^{\varepsilon}(t),y_{1}^{\varepsilon}(t)\big)\\[+0.8em]
\qquad \qquad
+\delta b_{x}(t) y_{1}^{\varepsilon}(t) \chi_{E_{\varepsilon}}(t)\Big\}dt\\[+0.8em]
\qquad \qquad
+\Big\{\sigma_{x}(t) y_{3}^{\varepsilon}(t)
+ \frac{1}{2}\big[\sigma_{xx}(t)\big(y_{1}^{\varepsilon}(t)+y_{2}^{\varepsilon}(t), y_{2}^{\varepsilon}(t)\big)\\[+0.8em]
\qquad \qquad
+ \sigma_{xx}(t)\big(y_{2}^{\varepsilon}(t),y_{1}^{\varepsilon}(t)\big)\big]
+\frac{1}{6}\sigma_{xxx}(t)
\big(y_{1}^{\varepsilon}(t),y_{1}^{\varepsilon}(t),y_{1}^{\varepsilon}(t)\big)\\[+0.8em]
\qquad \qquad
+\delta \sigma_{x}(t) y_{2}^{\varepsilon}(t) \chi_{E_{\varepsilon}}(t)
+\frac{1}{2}\delta\sigma_{xx}(t)\big(y_{1}^{\varepsilon}(t),y_{1}^{\varepsilon}(t)\big)
\chi_{E_{\varepsilon}}(t)\Big\}dW(t),\ \ \ \ \\[+0.8em]
\qquad \qquad t\in [0,T],\\[+0.8em]
y_{3}^{\varepsilon}(0)=0;
\end{array}\right.
\end{equation}
\begin{equation}\label{forthvariequ}
\qquad\left\{
\begin{array}{l}
dy_{4}^{\varepsilon}(t)= \Big\{b_{x}(t) y_{4}^{\varepsilon}(t)
+\frac{1}{2}\big[b_{xx}(t)\big(y_{1}^{\varepsilon}(t)+ y_{2}^{\varepsilon}(t) +y_{3}^{\varepsilon}(t), y_{3}^{\varepsilon}(t)\big)\\[+0.8em]
\qquad\qquad
+b_{xx}(t)\big(y_{3}^{\varepsilon}(t),y_{1}^{\varepsilon}(t)+ y_{2}^{\varepsilon}(t)\big)
\big]
+\frac{1}{6}\big[b_{xxx}(t)\big(y_{1}^{\varepsilon}(t),y_{1}^{\varepsilon}(t),
y_{2}^{\varepsilon}(t)\big)\\[+0.8em]
\qquad\qquad
+b_{xxx}(t)\big(y_{1}^{\varepsilon}(t),y_{2}^{\varepsilon}(t),
y_{1}^{\varepsilon}(t)\big)
+b_{xxx}(t)\big(y_{2}^{\varepsilon}(t),y_{1}^{\varepsilon}(t),
y_{1}^{\varepsilon}(t)\big)\\[+0.8em]
\qquad\qquad
+b_{xxx}(t)\big(y_{1}^{\varepsilon}(t),y_{2}^{\varepsilon}(t),
y_{2}^{\varepsilon}(t)\big)
+b_{xxx}(t)\big(y_{2}^{\varepsilon}(t),y_{2}^{\varepsilon}(t),
y_{1}^{\varepsilon}(t)\big)\\[+0.8em]
\qquad\qquad
+b_{xxx}(t)\big(y_{2}^{\varepsilon}(t),y_{1}^{\varepsilon}(t),
y_{2}^{\varepsilon}(t)\big)
+b_{xxx}(t)\big(y_{2}^{\varepsilon}(t),y_{2}^{\varepsilon}(t),
y_{2}^{\varepsilon}(t)\big)\big]\\[+0.8em]
\qquad\qquad
+\frac{1}{24}b_{xxxx}(t)\big( y_{1}^{\varepsilon}(t), y_{1}^{\varepsilon}(t), y_{1}^{\varepsilon}(t), y_{1}^{\varepsilon}(t)\big)\\[+0.8em]
\qquad\qquad
+\delta b_{x}(t) y_{2}^{\varepsilon}(t) \chi_{E_{\varepsilon}}(t)
+\frac{1}{2}\delta b_{xx}(t) \big(y_{1}^{\varepsilon}(t),y_{1}^{\varepsilon}(t)\big)
\chi_{E_{\varepsilon}}(t)\Big\}dt\\[+0.8em]
\qquad\qquad
+\Big\{\sigma_{x}(t) y_{4}^{\varepsilon}(t)
+\frac{1}{2}\big[\sigma_{xx}(t)\big(y_{1}^{\varepsilon}(t)+ y_{2}^{\varepsilon}(t) +y_{3}^{\varepsilon}(t), y_{3}^{\varepsilon}(t)\big)\\[+0.8em]
\qquad\qquad
+\sigma_{xx}(t)\big(y_{3}^{\varepsilon}(t),y_{1}^{\varepsilon}(t)+ y_{2}^{\varepsilon}(t)\big)
\big]
+\frac{1}{6}\big[\sigma_{xxx}(t)\big(y_{1}^{\varepsilon}(t),y_{1}^{\varepsilon}(t),
y_{2}^{\varepsilon}(t)\big)\\[+0.8em]
\qquad\qquad
+\sigma_{xxx}(t)\big(y_{1}^{\varepsilon}(t),y_{2}^{\varepsilon}(t),
y_{1}^{\varepsilon}(t)\big)
+\sigma_{xxx}(t)\big(y_{2}^{\varepsilon}(t),y_{1}^{\varepsilon}(t),
y_{1}^{\varepsilon}(t)\big)\\[+0.8em]
\qquad\qquad
+\sigma_{xxx}(t)\big(y_{1}^{\varepsilon}(t),y_{2}^{\varepsilon}(t),
y_{2}^{\varepsilon}(t)\big)
+\sigma_{xxx}(t)\big(y_{2}^{\varepsilon}(t),y_{2}^{\varepsilon}(t),
y_{1}^{\varepsilon}(t)\big)\\[+0.8em]
\qquad\qquad
+\sigma_{xxx}(t)\big(y_{2}^{\varepsilon}(t),y_{1}^{\varepsilon}(t),
y_{2}^{\varepsilon}(t)\big)
+\sigma_{xxx}(t)\big(y_{2}^{\varepsilon}(t),y_{2}^{\varepsilon}(t),
y_{2}^{\varepsilon}(t)\big)\big]\\[+0.8em]
\qquad \qquad
+\frac{1}{24}\sigma_{xxxx}(t)\big(y_{1}^{\varepsilon}(t),y_{1}^{\varepsilon}(t),
y_{1}^{\varepsilon}(t),y_{1}^{\varepsilon}(t)\big)
+\delta \sigma_{x}(t) y_{3}^{\varepsilon}(t) \chi_{E_{\varepsilon}}(t) \\[+0.8em]
\qquad \qquad +\frac{1}{2}\big[\delta\sigma_{xx}(t)\big(y_{1}^{\varepsilon}(t)+y_{2}^{\varepsilon}(t), y_{2}^{\varepsilon}(t)\big)
+\delta\sigma_{xx}(t)\big(y_{2}^{\varepsilon}(t),y_{1}^{\varepsilon}(t)
\big)\big]\chi_{E_{\varepsilon}}(t)\\[+0.8em]
\qquad\qquad
+\frac{1}{6}\delta\sigma_{xxx}(t)\big(y_{1}^{\varepsilon}(t),
y_{1}^{\varepsilon}(t),y_{1}^{\varepsilon}(t)\big)
\chi_{E_{\varepsilon}}(t)\Big\}dW(t), \quad t\in [0,T],\\[+0.8em]
y_{4}^{\varepsilon}(0)=0.
\end{array}\right.
\end{equation}

Denote
\begin{center}
\setlength{\tabcolsep}{0.5pt}
\begin{tabular*}{13cm}{@{\extracolsep{\fill}}lllr}
& $\xi(\cdot):=y_{1}^{\varepsilon}(\cdot)+y_{2}^{\varepsilon}(\cdot)
+y_{3}^{\varepsilon}(\cdot)+y_{4}^{\varepsilon}(\cdot)$,
& $\eta(\cdot):=y_{1}^{\varepsilon}(\cdot)+y_{2}^{\varepsilon}(\cdot)
+y_{3}^{\varepsilon}(\cdot)$,\\[+0.5em]
& $\gamma(\cdot):=y_{1}^{\varepsilon}(\cdot)+y_{2}^{\varepsilon}(\cdot)$,
& $r_{1}(\cdot):=\delta x(\cdot)-y_{1}^{\varepsilon}(\cdot)$,\\[+0.5em]
& $r_{2}(\cdot):=\delta x(\cdot)-\gamma(\cdot)$,
& $r_{3}(\cdot):=\delta x(\cdot)-\eta(\cdot)$,\\[+0.5em]
& $r_{4}(\cdot):=\delta x(\cdot)-\xi(\cdot)$.  &$ $
\end{tabular*}
\end{center}

From
(\ref{firstvariequ})--(\ref{forthvariequ}) and Lemma \ref{estimatelinearsde}, we obtain the following result.

\begin{lemma}\label{estimateofvariequ}
Let (C1) and (C2) hold. Then, for any $\beta\ge 1$, $\ephs\in (0,T)$, $\ephs\to 0^+$, the following estimates hold:
\begin{center}
\setlength{\tabcolsep}{0.5pt}
\begin{tabular*}{12cm}{@{\extracolsep{\fill}}lllr}
& $\|y_{1}^{\varepsilon}\|_{\infty,\beta}^{\beta}\le C\varepsilon^{\frac{\beta}{2}}$,
& $\|r_{1}\|_{\infty,\beta}^{\beta}\le C \varepsilon^{\beta}$,\\
& $\|y_{2}^{\varepsilon}\|_{\infty,\beta}^{\beta}\le C\varepsilon^{\beta}$,
& $\|r_{2}\|_{\infty,\beta}^{\beta}\le C\varepsilon^{\frac{3\beta}{2}}$,\\
& $\|y_{3}^{\varepsilon}\|_{\infty,\beta}^{\beta}\le C\varepsilon^{\frac{3\beta}{2}}$,
& $\|r_{3}\|_{\infty,\beta}^{\beta}\le C\varepsilon^{2\beta}$,\\
& $\|y_{4}^{\varepsilon}\|_{\infty,\beta}^{\beta}\le C \varepsilon^{2\beta}$,
& $\|r_{4}\|_{\infty,\beta}^{\beta}\le C\varepsilon^{\frac{5\beta}{2}}$,\\
& $\|\delta x\|_{\infty,\beta}^{\beta}\le \varepsilon^{\frac{\beta}{2}}$.
& \\
\end{tabular*}
\end{center}
\end{lemma}
\begin{proof}
See Appendix A.
\end{proof}

Further, we obtain the following Taylor expansion for the cost functional with respect to the control perturbation.
\begin{lemma}\label{taylorexplemma}
Let (C1) and (C2) hold. Then,
\begin{eqnarray}\label{taylorexp}
& & J(u^{\varepsilon})-J(\bar{u})\nonumber\\
&=& \me\int_{0}^{T}\Big[f_x(t)\xi(t)
+\frac{1}{2}f_{xx}(t)\big(\eta(t),\eta(t)\big)
+\frac{1}{6}f_{xxx}(t)\big(\gamma(t),\gamma(t),\gamma(t)\big)\nonumber\\
& &\qquad
+\frac{1}{24}f_{xxxx}(t)\big(y_{1}^{\varepsilon}(t),y_{1}^{\varepsilon}(t),
y_{1}^{\varepsilon}(t),y_{1}^{\varepsilon}(t)\big)+\delta f(t)\chi_{E_{\varepsilon}}(t)\nonumber\\
& &\qquad
+\delta f_x(t)\gamma(t)\chi_{E_{\varepsilon}}(t)
+\frac{1}{2}\delta f_{xx}(t)\big(y_{1}^{\varepsilon}(t),y_{1}^{\varepsilon}(t)\big)
\chi_{E_{\varepsilon}}(t)\Big]dt\nonumber\\
& &\qquad + \me \Big[h_{x}(\bar{x}(T))\xi(T)+\frac{1}{2}h_{xx}(\bar{x}(T))\big(\eta(T),\eta(T)\big)
\nonumber\\
& &\qquad
+\frac{1}{6}h_{xxx}(\bar{x}(T))\big(\gamma(T),\gamma(T),\gamma(T)\big)\nonumber\\
& &\qquad +\frac{1}{24}h_{xxxx}(\bar{x}(T))\big(y_{1}^{\varepsilon}(T),y_{1}^{\varepsilon}(T),
y_{1}^{\varepsilon}(T),y_{1}^{\varepsilon}(T)\big)\Big]
+ o(\varepsilon^2)\qquad (\ephs\to 0^+).
\end{eqnarray}
\end{lemma}
\begin{proof}
Similar to the proof of Lemma \ref{estimateofvariequ}, we only consider the $1$-dimensional case.
By Taylor's formulation,
\begin{eqnarray*}
& & f(t,x^{\varepsilon}(t),u^{\varepsilon}(t))-f(t,\bar{x}(t),\bar{u}(t))\\
&=& f(t,x^{\varepsilon}(t),\bar{u}(t))-f(t,\bar{x}(t),\bar{u}(t))
+f(t,\bar{x}(t),u^{\varepsilon}(t))-f(t,\bar{x}(t),\bar{u}(t))\\
& &+f(t,x^{\varepsilon}(t),u^{\varepsilon}(t))-f(t,\bar{x}(t),u^{\varepsilon}(t))
-f(t,x^{\varepsilon}(t),\bar{u}(t))+f(t,\bar{x}(t),\bar{u}(t))\\
&=& f_x(t)\delta x(t)+\frac{1}{2}f_{xx}(t)\delta x(t)^2
+\frac{1}{6}f_{xxx}(t)\delta x(t)^3\\
& & +\frac{1}{6}\int_{0}^{1}\theta^3 f_{xxxx}(t,\theta\bar{x}(t)+
(1-\theta)x^{\varepsilon}(t),\bar{u}(t))\delta x(t)^4d\theta
+\delta f(t)\chi_{E_{\varepsilon}}(t)\\
& &
+f_{x}(t,\bar{x}(t),u^{\varepsilon}(t))\delta x(t)
+\frac{1}{2}f_{xx}(t,\bar{x}(t),u^{\varepsilon}(t))\delta x(t)^2\\
& &+ \frac{1}{2}\int_{0}^{1}\theta^2 f_{xxx}(t,\theta\bar{x}(t)+
(1-\theta)x^{\varepsilon}(t),u^{\varepsilon}(t))\delta x(t)^3 d\theta
-f_{x}(t)\delta x(t)\\
& &
-\frac{1}{2}f_{xx}(t)\delta x(t)^2
-\frac{1}{2}\int_{0}^{1}\theta^2 f_{xxx}(t,\theta\bar{x}(t)+
(1-\theta)x^{\varepsilon}(t),\bar{u}(t))\delta x(t)^3 d\theta\\
&=& f_x(t)\delta x(t)+\frac{1}{2}f_{xx}(t)\delta x(t)^2
+\frac{1}{6}f_{xxx}(t)\delta x(t)^3\\
& & +\frac{1}{6}\int_{0}^{1}\theta^3 f_{xxxx}(t,\theta\bar{x}(t)+
(1-\theta)x^{\varepsilon}(t),\bar{u}(t))\delta x(t)^4d\theta\\
& &+\delta f(t)\chi_{E_{\varepsilon}}(t)
+\delta f_x(t)\delta x(t)\chi_{E_{\varepsilon}}(t)
+\frac{1}{2}\delta f_{xx}(t)\delta x(t)^2\chi_{E_{\varepsilon}}(t)\\
& &+ \frac{1}{2}\int_{0}^{1}\theta^2 \Big( f_{xxx}(t,\theta\bar{x}(t)+
(1-\theta)x^{\varepsilon}(t),u^{\varepsilon}(t))\\
& &\qquad\qquad\qquad -f_{xxx}(t,\theta\bar{x}(t)
+(1-\theta)x^{\varepsilon}(t),\bar{u}(t))\Big)\delta x(t)^3 d\theta,
\end{eqnarray*}
and
\begin{eqnarray*}
& & h(x^{\varepsilon}(T))-h(\bar{x}(T))\\
&=& h_{x}(\bar{x}(T))\delta x(T)+\frac{1}{2}h_{xx}(\bar{x}(T))\delta x(T)^2
+\frac{1}{6}h_{xxx}(\bar{x}(T))\delta x(T)^3\\
& & +\frac{1}{6}\int_{0}^{1}\theta^3 h_{xxxx}(\theta\bar{x}(T)
+(1-\theta)x^{\varepsilon}(T))\delta x(T)^4d\theta.
\end{eqnarray*}

By Lemma \ref{estimateofvariequ},
\begin{eqnarray*}
& & J(u^{\varepsilon})-J(\bar{u})\\
&=& \me\int_{0}^{T}\Big[f_x(t)\xi(t)
+\frac{1}{2}f_{xx}(t)\eta(t)^2+\frac{1}{6}f_{xxx}(t)\gamma(t)^3
+\frac{1}{24}f_{xxxx}(t)y_{1}^{\varepsilon}(t)^4\\
& &+\delta f(t)\chi_{E_{\varepsilon}}(t)
+\delta f_x(t)\gamma(t)\chi_{E_{\varepsilon}}(t)+\frac{1}{2}\delta f_{xx}(t)y_{1}^{\varepsilon}(t)^2\chi_{E_{\varepsilon}}(t)\Big]dt\\
& &+\me ~\Big[h_{x}(\bar{x}(T))\xi(T)+\frac{1}{2}h_{xx}(\bar{x}(T))\eta(T)^2 \\
& &\qquad\qquad\qquad\qquad +\frac{1}{6}h_{xxx}(\bar{x}(T))\gamma(T)^3
+\frac{1}{24}h_{xxxx}(\bar{x}(T))y_{1}^{\varepsilon}(T)^4\Big]\\
& &+\me\int_{0}^{T}\Big[f_x(t)r_{4}(t)+\frac{1}{2}f_{xx}(t)\big(\delta x(t)^2-\eta(t)^{2}\big)+\frac{1}{6}f_{xxx}(t)\big(\delta x(t)^3-\gamma(t)^{3}\big)\\
& & +\frac{1}{6}\int_{0}^{1}\theta^3 f_{xxxx}(t,\theta\bar{x}(t)+
(1-\theta)x^{\varepsilon}(t),\bar{u}(t))\delta x(t)^4d\theta
-\frac{1}{24}f_{xxxx}(t)y_{1}^{\varepsilon}(t)^4\\
& &+\delta f_x(t)r_{2}(t)\chi_{E_{\varepsilon}}(t)
+\frac{1}{2}\delta f_{xx}(t)\big(\delta x(t)^2
-y_{1}^{\varepsilon}(t)^2\big)\chi_{E_{\varepsilon}}(t)\\
& &+\frac{1}{2}\int_{0}^{1}\theta^{2} \Big( f_{xxx}(t,\theta\bar{x}(t)+
(1-\theta)x^{\varepsilon}(t),u^{\varepsilon}(t))\\
& &\qquad\qquad\qquad\qquad
-f_{xxx}(t,\theta\bar{x}(t)+(1-\theta)
x^{\varepsilon}(t),\bar{u}(t))\Big)\delta x(t)^3 d\theta\Big]dt\\
& &+\me~\Big[h_{x}(\bar{x}(T))r_{4}(T)
+\frac{1}{2}h_{xx}(\bar{x}(T))\big(\delta x(T)^2-\eta(T)^{2}\big)\\
& &
 +\frac{1}{6}h_{xxx}(\bar{x}(T))\big(\delta x(T)^3-\gamma(T)^{3}\big)\\
& &+\frac{1}{6}\int_{0}^{1}\theta^3 h_{xxxx}(\theta\bar{x}(T)
+(1-\theta)x^{\varepsilon}(T))\delta x(T)^4d\theta
-\frac{1}{24}h_{xxxx}(\bar{x}(T))y_{1}^{\varepsilon}(T)^4\Big]\\
&=& \me\int_{0}^{T}\Big[f_x(t)\xi(t)
+\frac{1}{2}f_{xx}(t)\eta(t)^2+\frac{1}{6}f_{xxx}(t)\gamma(t)^3
+\frac{1}{24}f_{xxxx}(t)y_{1}^{\varepsilon}(t)^4\\
& &+\delta f(t)\chi_{E_{\varepsilon}}(t)
+\delta f_x(t)\gamma(t)\chi_{E_{\varepsilon}}(t)+\frac{1}{2}\delta f_{xx}(t)y_{1}^{\varepsilon}(t)^2\chi_{E_{\varepsilon}}(t)\Big]dt\\
& &+\me~\Big[h_{x}(\bar{x}(T))\xi(T)+\frac{1}{2}h_{xx}(\bar{x}(T))\eta(T)^2\\
& &\quad\quad +\frac{1}{6}h_{xxx}(\bar{x}(T))\gamma(T)^3
+\frac{1}{24}h_{xxxx}(\bar{x}(T))y_{1}^{\varepsilon}(T)^4\Big]+ o(\varepsilon^2) \quad (\ephs\to 0^+). \\
\end{eqnarray*}
This completes the proof of Lemma \ref{taylorexplemma}. \end{proof}

To establish the variational formulation for the optimal control $\bar{u}(\cdot)$, in addition to the adjoint equations (\ref{firstajointequ})--(\ref{secondajointequ}), the following two adjoint equations are also needed:
\begin{equation}\label{thirdajointequ}
\qquad\left\{
\begin{array}{l}
dp_{3}(t)=-\Bigg[\sum\limits_{k=1}^{3}p_{3}(t)\mathop\circ\limits_{k}b_{x}(t)
+\sum\limits_{k=1}^{2}\sum\limits_{l=k+1}^{3}
p_{3}(t)\mathop\circ\limits_{k,l}\big(\sigma_{x}(t),\sigma_{x}(t)\big)\\[+1.5em]
\qquad\qquad
+\sum\limits_{k=1}^{3} q_{3}(t)\mathop\circ\limits_{k} \sigma_{x}(t)
+\frac{3}{2}\sum\limits_{k=1}^{2}\Big(p_{2}(t)\mathop\circ\limits_{k} b_{xx}(t)
+q_{2}(t)\mathop\circ\limits_{k}\sigma_{xx}(t)\Big)\\[+1.5em]
\qquad\qquad
+\frac{3}{2}\Big(p_{2}(t)\mathop\circ\limits_{1,2}\big(\sigma_{x}(t),\sigma_{xx}(t)\big)
+p_{2}(t)\mathop\circ\limits_{1,2}\big(\sigma_{xx}(t),\sigma_{x}(t)\big)\Big)\\[+1.5em]
\qquad\qquad\qquad\qquad\qquad\qquad
+\hh_{xxx}(t)\Bigg]dt+q_{3}(t)dW(t),\quad t\in [0,T],  \\[+1.3em]
p_{3}(T)=-h_{xxx}(\bar{x}(T));
\end{array}\right.
\end{equation}
and
\begin{equation}\label{forthajointequ}
\ \ \left\{
\begin{array}{l}
dp_{4}(t)=-\Bigg[\sum\limits_{k=1}^{4} p_{4}(t) \mathop\circ\limits_{k}b_{x}(t)
+\sum\limits_{k=1}^{3}\sum\limits_{l=k+1}^{4} p_{4}(t)\mathop\circ\limits_{k,l}\big(\sigma_{x}(t),\sigma_{x}(t)\big)\\[+1em]
\qquad\qquad
+\sum\limits_{k=1}^{4}q_{4}(t)\mathop\circ\limits_{k}\sigma_{x}(t)
+2\sum\limits_{k=1}^{3}\Big( p_{3}(t)\mathop\circ\limits_{k}b_{xx}(t)
+q_{3}(t)\mathop\circ\limits_{k} \sigma_{xx}(t)\Big)\\[+1em]
\qquad\qquad
+2 \sum\limits_{k=1}^{3}\sum\limits_{ l=1,\ l\neq k}^{3}p_{3}(t)\mathop\circ\limits_{k,l}\big(\sigma_{x}(t),\sigma_{xx}(t)\big)
+2\sum\limits_{k=1}^{2}p_{2}(t)\mathop\circ\limits_{k}b_{xxx}(t)\\[+1em]
\qquad\qquad
+2\Big(p_{2}(t)\mathop\circ\limits_{1,2}\big(\sigma_{x}(t),\sigma_{xxx}(t)\big)
+p_{2}(t)\mathop\circ\limits_{1,2}\big(\sigma_{xxx}(t),\sigma_{x}(t)\big)\Big)\\[+1em]
\qquad\qquad
+3p_{2}(t)\mathop\circ\limits_{1,2}\big(\sigma_{xx}(t),\sigma_{xx}(t)\big)
+2\sum\limits_{k=1}^{2}q_{2}(t) \mathop\circ\limits_{k}\sigma_{xxx}(t)\\[+1em]
\qquad\qquad\qquad\qquad\qquad
+\hh_{xxxx}(t)\Bigg]dt+q_{4}(t)dW(t),\quad t\in [0,T], \\[+1em]
p_{4}(T)=-h_{xxxx}(\bar{x}(T)),
\end{array}\right.
\end{equation}
where the Hamiltonian $\hh$ is defined by (\ref{Hamiltonian}), and
\begin{eqnarray*}
& &\hh_{xxx}(t)=\hh_{xxx}(\omega,t, \bar{x}(t),\bar{u}(t), p_{1}(t),q_{1}(t)),\\
& &\hh_{xxxx}(t)=\hh_{xxxx}(\omega,t, \bar{x}(t),\bar{u}(t), p_{1}(t),q_{1}(t)).
\end{eqnarray*}

By the existence and regularity results for BSDEs (see \cite{Peng97}),  for any $\beta\ge 1$, the adjoint  equations (\ref{thirdajointequ})--(\ref{forthajointequ}) admit unique solutions, respectively, and
\begin{eqnarray*}
& &(p_{3}(\cdot),q_{3}(\cdot))\in L_{\mmf}^{\beta}(\Omega; C([0,T]; L(\mathop\Pi\limits_{3}\mrn;\mr)))\times
L_{\mmf}^{\beta}(\Omega; L^{2}(0,T; L(\mathop\Pi\limits_{3}\mrn;\mr))),\\
& &(p_{4}(\cdot),q_{4}(\cdot))\in L_{\mmf}^{\beta}(\Omega; C([0,T]; L(\mathop\Pi\limits_{4}\mrn;\mr)))\times
L_{\mmf}^{\beta}(\Omega; L^{2}(0,T; L(\mathop\Pi\limits_{4}\mrn;\mr))).
\end{eqnarray*}

Using the Taylor expansion of the cost functional established in Lemma \ref{taylorexplemma} and the duality relationship between the variational equations (\ref{firstvariequ})--(\ref{forthvariequ}) and the adjoint equations (\ref{firstajointequ})--(\ref{secondajointequ}) and (\ref{thirdajointequ})--(\ref{forthajointequ}), we obtain a variational formulation for the cost functional. In order to short the expression of this formulation, we introduce some more notations.

Let the Hamiltonian be defined by (\ref{Hamiltonian}). Write
\begin{eqnarray*}
& & \mS(\omega, t,x,u, y_{2},z_{2})
=\frac{1}{2}\sum_{k=1}^{2}\Big[y_{2}\mathop\bullet\limits_{k} b(\omega,t,x,u)
+z_{2}\mathop\bullet\limits_{k} \sigma(\omega,t,x,u)
\Big],\\
& & \qquad\qquad
(\omega,t,x,u,y_{2},z_{2})\in \Omega\times[0,T]\times\mrn\times U\times L(\mathop\Pi\limits_{2}\mrn;\mr)\times L(\mathop\Pi\limits_{2}\mrn;\mr);\\
& & \mT(\omega,t,x,u, y_{3},z_{3})
=\frac{1}{3}\sum_{k=1}^{3}\Big[y_{3}\mathop\bullet\limits_{k} b(\omega,t,x,u)
+z_{3}\mathop\bullet\limits_{k} \sigma(\omega,t,x,u)
\Big],\\
& & \qquad\qquad
(\omega,t,x,u,y_{3},z_{3})\in \Omega\times[0,T]\times\mrn\times U\times L(\mathop\Pi\limits_{3}\mrn;\mr)\times L(\mathop\Pi\limits_{3}\mrn;\mr),
\end{eqnarray*}
and  denote
\begin{eqnarray*}
& & \mss(\omega,t,x,u)\\
&=&\mh_{x}(\omega,t,x,u)+\mS(\omega,t,x,u, p_{2}(t),q_{2}(t))
-\mS(\omega,t,x,\bar{u}(t), p_{2}(t),q_{2}(t))\\
& &
+\frac{1}{2}
\big[p_{2}(t)\mathop\circ\limits_{1}\sigma_{x}(\omega,t,x,\bar{u}(t))\big]
\mathop\bullet\limits_{2}\big(\sigma(\omega,t,x,u)-\sigma(\omega,t,x,\bar{u}(t))\big)\\
& &
+\frac{1}{2}
\big[p_{2}(t)\mathop\circ\limits_{2}\sigma_{x}(\omega,t,x,\bar{u}(t))\big]
\mathop\bullet\limits_{1}\big(\sigma(\omega,t,x,u)-\sigma(\omega,t,x,\bar{u}(t))\big)\\
& &
+\frac{1}{6}\sum\limits_{k=1}^{2}\sum\limits_{l=k+1}^{3}
p_{3}(t)\mathop\bullet\limits_{k,l}
\big(\sigma(\omega,t,x,u)-\sigma(\omega,t,x,\bar{u}(t)),
\sigma(\omega,t,x,u)-\sigma(\omega,t,x,\bar{u}(t))\big),\\[+0.8em]
& & \mt(\omega,t,x,u)\\
&=&\mss_{x}(\omega,t,x,u)+\mS_{x}(\omega,t,x,u, p_{2}(t),q_{2}(t))
-\mS_{x}(\omega,t,x,\bar{u}(t), p_{2}(t),q_{2}(t))\\
& &
+\mT(\omega,t,x,u, p_{3}(t),q_{3}(t))
 -\mT(\omega,t,x,\bar{u}(t), p_{3}(t),q_{3}(t))\\
& &
+\frac{1}{2}\sum\limits_{k=1}^{2}\sum\limits_{l=1,\ l\neq k}^{2}
p_{2}(t)\mathop\circ\limits_{k,l}\big(\sigma_{x}(\omega,t,x,\bar{u}(t)),\sigma_{x}(t,x,u)
-\sigma_{x}(\omega,t,x,\bar{u}(t))\big)\\
& &
+\frac{1}{6}\sum\limits_{k=1}^{3}\sum\limits_{l=1,\ l\neq k}^{3}\big[
p_{3}(t)\mathop\circ\limits_{k}
\big(\sigma_{x}(\omega,t,x,u)-\sigma_{x}(\omega,t,x,\bar{u}(t))\big)\big]\\
& &\qquad\qquad\qquad\qquad\qquad\qquad\qquad\qquad\qquad
\mathop\bullet\limits_{l}
\big(\sigma(\omega,t,x,u)-\sigma(\omega,t,x,\bar{u}(t))\big)
\\
& &
+\frac{1}{3}\sum\limits_{k=1}^{3}\sum\limits_{l=1,\ l\neq k}^{3}\big[
p_{3}(t)\mathop\circ\limits_{k}
\sigma_{x}(\omega,t,x,\bar{u}(t))\big]\mathop\bullet\limits_{l}
\big(\sigma(\omega,t,x,u)-\sigma(\omega,t,x,\bar{u}(t))\big)
\\
& &
+\frac{1}{12}\sum\limits_{k=1}^{3}\sum\limits_{l=k+1}^{4}
p_{4}(t)\mathop\bullet\limits_{k,l}
\big(\sigma(\omega,t,x,u)-\sigma(\omega,t,x,\bar{u}(t)),\\
& &\qquad \qquad \qquad \qquad \qquad\qquad \qquad\qquad \qquad
\sigma(\omega,t,x,u)-\sigma(\omega,t,x,\bar{u}(t))\big),\\
& &\qquad \qquad \qquad \qquad \qquad\qquad \qquad \qquad \qquad\quad
(\omega,t,x,u)\in \Omega\times[0,T]\times\mrn\times U,
\end{eqnarray*}
where, $(p_{1}(\cdot),q_{1}(\cdot))$ and $(p_{2}(\cdot),q_{2}(\cdot))$ are respectively the solutions to (\ref{firstajointequ}) and (\ref{secondajointequ}), $(p_{3}(\cdot),q_{3}(\cdot))$ and $(p_{4}(\cdot),q_{4}(\cdot))$ are respectively the solutions to (\ref{thirdajointequ}) and (\ref{forthajointequ}).

We have the following variational formulation for
the cost functional.

\begin{proposition}\label{variational formulation for noncov}
Let (C1) and (C2) hold.  Then,
\begin{eqnarray}\label{shorttaylor}
& &J(u^{\varepsilon}(\cdot))-J(\bar{u}(\cdot))\nonumber\\
\qquad&=&-\me\int_{0}^{T}\Big[
\mh(t,\bar{x}(t),u(t))+\inner{\mss(t,\bar{x}(t),u(t))}{\gamma(t)}\nonumber\\
& &\qquad\quad\quad
+\frac{1}{2}\inner{\mt(t,\bar{x}(t),u(t))y_{1}^{\varepsilon}(t)}{y_{1}^{\varepsilon}(t)}
\Big]\chi_{E_{\varepsilon}}(t)dt+o(\varepsilon^{2}),\ \  (\ephs\to 0^+).
\end{eqnarray}
\end{proposition}
\begin{proof}
See Appendix B.
\end{proof}

\section{Second-order necessary conditions}

In this section, we establish some second-order necessary conditions for
stochastic singular optimal controls in the sense of Pontryagin-type maximum principle. Firstly, we introduce the concept of the singular control (The corresponding concept for deterministic control systems can be found in \cite{Gabasov72} and the references cited therein).
\begin{definition}\label{singularcont Def}
An admissible control $\tilde{u}(\cdot)$ is called a singular control in the sense of
Pontryagin-type maximum principle on a control region $V$, if $V$ is a nonempty subset of $U$ and
\begin{eqnarray}\label{singularcont concept}
\qquad 0&=&\mh(t,\tilde{x}(t), v)\\
&=&\hh(t,\tilde{x}(t), v, \tilde{p}_{1}(t), \tilde{q}_{1}(t))
-\hh(t,\tilde{x}(t), \tilde{u}(t), \tilde{p}_{1}(t),\tilde{q}_{1}(t))\nonumber\\
& &+\frac{1}{2}\inner{\tilde{p}_{2}(t)
(\sigma(t,\tilde{x}(t),v)-\sigma(t,\tilde{x}(t),\tilde{u}(t)))}{
\sigma(t,\tilde{x}(t),v)-\sigma(t,\tilde{x}(t),\tilde{u}(t))},\nonumber\\
& &\qquad\qquad\qquad\qquad\qquad\qquad \qquad\qquad\
\ \ \forall \ v\in V,\  a.s.,\ a.e.\ t\in [0,T].\nonumber
\end{eqnarray}
where $\tilde{x}(\cdot)$ is the state with respect to $\tilde{u}(\cdot)$, and $(\tilde{p}_{1}(\cdot),\tilde{q}_{1}(\cdot))$,
$(\tilde{p}_{2}(\cdot),\tilde{q}_{2}(\cdot))$ are the adjoint processes given respectively by (\ref{firstajointequ}) and (\ref{secondajointequ}) with $(\bar x(\cdot),\bar u(\cdot)
)$ replaced by $(\tilde{x}(\cdot),\tilde{u}(\cdot))$. If the singular control $\tilde{u}(\cdot)$ is also optimal, we call it a singular optimal control.
\end{definition}

In the sequel, we shall fix the control subset $V\subset U$ appeared in Definition \ref{singularcont concept}.

\begin{remark}
In \cite{zhangH14a}, we introduced the concept of singular control in the classical sense.
Let us recall that, an admissible control $\tilde{u}(\cdot)$ is called a singular control in the classical sense if $\tilde{u}(\cdot)$ satisfies
\begin{equation}\label{singularcontrol Hu Huu euql zero}
\quad \ \ \ \ \  \left\{\!\!\!
\begin{array}{l}
\hh_{u}(t,\tilde{x}(t), \tilde{u}(t) ,\tilde{p}_{1}(t),\tilde{q}_{1}(t))=0,\ \  \ a.s., \ a.e.\ t\in [0,T], \\
\hh_{uu}(t,\tilde{x}(t), \tilde{u}(t) ,\tilde{p}_{1}(t),\tilde{q}_{1}(t))+\sigma_{u}(t,\tilde{x}(t),\tilde{u}(t))^{\top}
\tilde{p}_{2}(t)\sigma_{u}(t,\tilde{x}(t),\tilde{u}(t))=0,\\
 \qquad \qquad  a.s., \ a.e.\ t\in [0,T].
\end{array}\right.
\end{equation}
If $(\tilde{x}(\cdot),\tilde{u}(\cdot))$ is an optimal pair, the first-order necessary condition (\ref{pengs firstorder condition}) says that the map
$$v\mapsto\mh(\omega,t,\tilde{x}(t), v),\quad v\in U$$
admits its maximum at $\tilde{u}(t)$ for a.e. $(\omega,t)\in\Omega\times[0,T]$. A singular control in the classical sense is the one that satisfies trivially the first- and second-order necessary conditions ( for a.e. $(\omega,t)\in\Omega\times[0,T]$) for the maximization problem
$$\max_{v\in U}\ \mh(t,\tilde{x}(t), v).$$
Obviously, when the set $V$ is open and $\tilde{u}(t)\in V$, $a.e.$ $(\omega,t)\in \Omega\times [0,T]$, any singular control in the sense of Pontryagin-type maximum principle satisfies (\ref{singularcontrol Hu Huu euql zero}), that is, $\tilde{u}$ is also a singular control in the classical sense, but not vice versa.
\end{remark}

\begin{remark}
Since in this paper we consider the case of diffusion term containing the control variable, in (\ref{singularcont concept}) there exists the second order term $$\frac{1}{2}\inner{\tilde{p}_{2}(t)(\sigma(t,\tilde{x}(t),v)
-\sigma(t,\tilde{x}(t),\tilde{u}(t)))}{\sigma(t,\tilde{x}(t),v)
-\sigma(t,\tilde{x}(t),\tilde{u}(t))}.$$
When the diffusion term independent of the control variable this term is equal to $0$. In this case, Definition \ref{singularcont Def} reduces to Definition 2.1. in \cite{Tang10}.
\end{remark}

We need the following simple result.

\begin{lemma}\label{bouned of S}
Let (C1) and (C2) hold. Then $\mss(\cdot,\bar{x}(\cdot),u(\cdot))
\in L_{\mmf}^{4}(\Omega; L^{2}([0,T]; \mrn))$ and  $\mt(\cdot,\bar{x}(\cdot),u(\cdot))\in L_{\mmf}^{2}(\Omega; L^{2}([0,T]; L(\mathop\Pi\limits_{2}\mrn;\mr)))$ for any $u(\cdot)\in \mmu_{ad}$.
\end{lemma}
\begin{proof}
It is sufficient to prove that
$$\me\Big[\int_{0}^{T}|\mss(t,\bar{x}(t),u(t))|^2dt\Big]^{2} <\infty$$
and
$$\me\int_{0}^{T}|\mt(t,\bar{x}(t),u(t))|^2dt<\infty.$$
By  (C1)--(C2), there exists a constant $C$ such that, for $\varphi=b,\ \delta,\  f$,
$$ |\varphi_{x} (t)|\le C,\ \  |\delta \varphi (t)|\le C, \   \mbox{ and }\ |\delta \varphi_{x} (t)|\le C,
\ \ a.e. \ (\omega,t)\in\Omega\times [0,T].$$
Therefore,
\begin{eqnarray*}
& &\me\Big[\int_{0}^{T}|\mss(t,\bar{x}(t),u(t))|^2dt\Big]^{2}\\
&=& \me\Big[\int_{0}^{T}\big|\delta b_{x}(t)\mathop\bullet\limits_{1}p_{1}(t)
+\delta \sigma_{x}(t)\mathop\bullet\limits_{1}q_{1}(t) -\delta f_x(t)\\
& &+\frac{1}{2}\sum\limits_{k=1}^{2} p_{2}(t)\mathop\bullet\limits_{k}\delta b(t)
+\frac{1}{2}\sum\limits_{k=1}^{2} q_{2}(t)\mathop\bullet\limits_{k}\delta \sigma(t)
+\frac{1}{2} [p_{2}(t)\mathop\circ\limits_{1}\sigma_{x}(t)]
\mathop\bullet\limits_{2}\delta \sigma(t)
\\
& &
+\frac{1}{2} [p_{2}(t)\mathop\circ\limits_{2}\sigma_{x}(t)]
\mathop\bullet\limits_{1}\delta \sigma(t)
+\frac{1}{2} [p_{2}(t)\mathop\circ\limits_{1}\delta\sigma_{x}(t)]
\mathop\bullet\limits_{2}\delta \sigma(t)\\
& &+\frac{1}{2} [p_{2}(t)\mathop\circ\limits_{2}\delta\sigma_{x}(t)]
\mathop\bullet\limits_{1}\delta \sigma(t)
+\frac{1}{6}\sum\limits_{k=1}^{2}\sum\limits_{l=k+1}^{3}
p_{3}(t)\mathop\bullet\limits_{k,l}
\big(\delta \sigma(t),\delta \sigma(t)\big)
\big|^2dt\Big]^{2}\\
&\le& C + C\me\Big[\int_{0}^{T}\big(|p_{1}(t)|^2+|q_{1}(t)|^2+|p_{2}(t)|^2
+|q_{2}(t)|^2+|p_{3}(t)|^2\big)dt\Big]^{2}\\
&\le& C + C\big(\|p_{1}\|_{\infty,4}^{4}+\|q_{1}\|_{2,4}^{4}
+\|p_{2}\|_{\infty,4}^{4}+\|q_{2}\|_{2,4}^{4}+\|p_{3}\|_{\infty,4}^{4}\big)\\
&<& \infty.
\end{eqnarray*}
In a similar way, we can prove that
$$\me\int_{0}^{T}|\mt(t,\bar{x}(t),u(t))|^2dt<\infty.$$
\end{proof}

By Lemma \ref{bouned of S}, it follows that $\mss(\cdot,\bar{x}(\cdot),v)
\in L_{\mmf}^{2}(\Omega; L^{2}([0,T]; \mrn))$ for any $v\in V\subset U$. By \cite[Lemma 3.8]{zhangH14a}, there exists a $\phi_{v}\in L^{2}(0,T;  L_{\mmf}^{2}(\Omega\times[0,T]; \mrn))$ such that
\begin{equation}\label{martingale exp of mss(t)}
\mss(t,\bar{x}(t),v)=\me~\mss(t,\bar{x}(t),v)+\int_{0}^{t}\phi_{v}(s,t)dW(s),\ \ a.s.,\ a.e.\ t\in[0,T].
\end{equation}
Denote by $\Phi(\cdot)$ the solution to the following stochastic differential equation
\begin{equation}\label{Phi}
\left\{
\begin{array}{l}
d\Phi(t)= b_{x}(t)\Phi(t)dt+\sigma_{x}(t)\Phi(t)dW(t),
\qquad \ \ \ t\in[0,T], \qquad\qquad\\
\Phi(0)=I,
\end{array}\right.
\end{equation}
where $I$ is the identity matrix in $\mr^{n\times n}$.

Using the martingale representation formula (\ref{martingale exp of mss(t)}), we obtain the following second-order necessary condition:
\begin{theorem}\label{2orderconditionth}
Let (C1) and (C2) hold. If $\bar{u}(\cdot)$ is a singular optimal control in the sense of Pontryagin-type maximum principle on the control subset $V\subset U$, then, for any $v\in V$, it holds that
\begin{eqnarray}\label{2ordercondition}
\quad& &\me ~\Big\langle\mss(\t,\bar{x}(\t),v),
b(\t,\bar{x}(\t),v)-b(\t,\bar{x}(\t),\bar{u}(\t))\Big\rangle\nonumber\\
& &+\partial^{+}_{\t}\Big(\mss(\t,\bar{x}(\t),v);
\sigma(\t,\bar{x}(\t),v)-\sigma(\t,\bar{x}(\t),\bar{u}(\t))\Big)\nonumber\\
& &+\frac{1}{2}\me ~\Big\langle\mt(\t,\bar{x}(\t),v)
\big(\sigma(\t,\bar{x}(\t),v)-\sigma(\t,\bar{x}(\t),\bar{u}(\t))\big),\nonumber\\
& &\qquad\qquad\qquad\qquad\quad
\sigma(\t,\bar{x}(\t),v)-\sigma(\t,\bar{x}(\t),\bar{u}(\t))
\Big\rangle\le 0, \ \  a.e.\  \t\in [0,T].
\end{eqnarray}
where
\begin{eqnarray}\label{2ozxndition}
& &\frac{1}{2}\partial^{+}_{\t}\Big(\mss(\t,\bar{x}(\t),v);
\sigma(\t,\bar{x}(\t),v)-\sigma(\t,\bar{x}(\t),\bar{u}(\t))\Big)\\
&:=&\limsup_{\theta\to 0^+}\frac{1}{\theta^2}\me\int_{\t}^{\t+\theta}\int_{\t}^{t}
\Big\langle\phi_{v}(s,t),\nonumber\\
& &\qquad\qquad\qquad\qquad\quad \Phi(\t)\Phi(s)^{-1}
\big(\sigma(s,\bar{x}(s),v)-\sigma(s,\bar{x}(s),\bar{u}(s))\big)\Big\rangle dsdt,\nonumber
\end{eqnarray}
$\phi_{v}(\cdot,\cdot)$ is determined by (\ref{martingale exp of mss(t)}).
\end{theorem}

The proof of Theorem \ref{2orderconditionth} will be given in Subsection 5.1.

Note that the second-order necessary condition (\ref{2ordercondition}) is only a pointwise type condition with respect to the time variable $t$ ($\in[0,T]$). To obtain the pointwise second-order necessary conditions with respect to both the time $t$ and the sample point $\omega$ ($\in\Omega$), similar to the first part of our work (see \cite{zhangH14a}), we need the following regularity condition.
\vspace{4mm}
\begin{enumerate}
\item [(C3)] For any $v\in V$, $\mss(\cdot,\bar{x}(\cdot),v)\in \ml_{2,\mmf}^{1,2}(\mrn)$, and the map $v\mapsto \nabla\mss(t,\bar{x}(t),v)$ is continuous on $V$ a.s., a.e. $t\in[0,T]$.
\end{enumerate}
\vspace{4mm}

We have the following result.

\begin{theorem}\label{Th 2ordercondition nonconvex malliavin}
Let (C1)--(C3) hold. If $\bar{u}(\cdot)$ is a singular optimal control in the sense of Pontryagin-type maximum principle on the control subset $V\subset U$, then, for a.e. $\t\in [0,T]$, it holds that
\begin{eqnarray}\label{2ordercondition nonconvex corollary}
& &\inner{\mss(\t,\bar{x}(\t), v)}{
b(\t,\bar{x}(\t),v)-b(\t,\bar{x}(\t),\bar{u}(\t))}\nonumber\\
& &+ \inner{\nabla\mss(\t,\bar{x}(\t),v)}{
\sigma(\t,\bar{x}(\t),v)-\sigma(\t,\bar{x}(\t),\bar{u}(\t))}\nonumber\\
& &
+\frac{1}{2}\big\langle\mt(\t,\bar{x}(\t),v)
\big(\sigma(\t,\bar{x}(\t),v)-\sigma(\t,\bar{x}(\t),\bar{u}(\t))\big),\nonumber\\
& &\qquad\qquad\qquad\quad\
\sigma(\t,\bar{x}(\t),v)-\sigma(\t,\bar{x}(\t),\bar{u}(\t))
\big\rangle\le 0,
\ \forall \ \ v\in V, \ a.s.
\end{eqnarray}
\end{theorem}

The proof of Theorem \ref{Th 2ordercondition nonconvex malliavin} will be given in Subsection 5.2.

As an easy consequence of Theorem \ref{Th 2ordercondition nonconvex malliavin}, the following pointwise second-order condition immediately  holds.
\begin{corollary}\label{coroll for S equ zero}
Let (C1)--(C2) hold. If $\bar{u}(\cdot)$ is a singular optimal control in the sense of Pontryagin-type maximum principle on the control subset $V\subset U$ and
\begin{equation}\label{singularcont with seuqalszero}
\mss(t,\bar{x}(t), v)=0, \qquad \forall \ v\in V,\ a.s.,\ a.e.\ t\in[0,T],
\end{equation}
then, for a.e. $\t\in [0,T]$, it holds that
\begin{eqnarray}\label{2orderconditionth sequalszero}
& &\big\langle\mt(\t,\bar{x}(\t),v)
\big(\sigma(\t,\bar{x}(\t),v)-\sigma(\t,\bar{x}(\t),\bar{u}(\t))\big),\\
& &\qquad\qquad\qquad\quad\
\sigma(\t,\bar{x}(\t),v)-\sigma(\t,\bar{x}(\t),\bar{u}(\t))
\big\rangle\le 0,
\ \  \forall \ \ v\in V, \ a.s.\nonumber
\end{eqnarray}
\end{corollary}

\begin{remark}
When the diffusion term is independent of the control variable, $\sigma(t,\bar{x}(t),v)-\sigma(\t,\bar{x}(t),\bar{u}(t))= 0$ for any $(\omega,t)\in\Omega\times[0,T]$. Therefore,
$$\inner{\nabla\mss(\t,\bar{x}(\t),v)}{
\sigma(\t,\bar{x}(\t),v)-\sigma(\t,\bar{x}(\t),\bar{u}(\t))}\equiv 0,$$
$$\inner{\mt(\t,\bar{x}(\t),v)\big(\sigma(\t,\bar{x}(\t),v)
-\sigma(\t,\bar{x}(\t),\bar{u}(\t))\big)}{
\sigma(\t,\bar{x}(\t),v)-\sigma(\t,\bar{x}(\t),\bar{u}(\t))}\equiv 0,$$
and the condition (\ref{2ordercondition nonconvex corollary}) is reduced to
$$
\inner{\mss(\t,\bar{x}(\t), v)}{
b(\t,\bar{x}(\t),v)-b(\t,\bar{x}(\t),\bar{u}(\t))}\le 0, \quad \forall \ v\in V,\ a.s.,\ a.e.\  \t\in [0,T],
$$
where, in this case,
\begin{eqnarray*}
\mss(\omega,t,x,u)&=&\hh_{x}(\omega,t,x,u, p_{1}(t),q_{1}(t))-\hh_{x}(\omega,t,x,\bar{u}(t), p_{1}(t),q_{1}(t)),\\
& &
+\frac{1}{2}p_{2}(t)(b(\omega,t,x,u)-b(\omega,t,x,\bar{u}(t)))\\
& &
+\frac{1}{2}(b(\omega,t,x,u)-b(\omega,t,x,\bar{u}(t)))^{\top}p_{2}(t),\\
& &\qquad\qquad\qquad\qquad\qquad (\omega,t,x,u)\in \Omega\times[0,T]\times \mrn\times U.
\end{eqnarray*}
The corresponding result coincides with \cite[Theorem 2.1]{Tang10}.
In addition, since the diffusion term is independent of the control variable, $y_{1}^{\varepsilon}(t)\equiv 0$, and hence
$$
\lim_{\varepsilon\to 0^+}\frac{1}{\varepsilon^2}
\me\int_{0}^{T}~\mss(t,\bar{x}(t),v)y_{1}^{\varepsilon}(t)
\chi_{E_{\varepsilon}}(t)dt=0.
$$
In this case, it is unnecessary to introduce the regularity assumption (C3) to prove the desired condition (\ref{2ordercondition nonconvex corollary}).
\end{remark}

\begin{remark}
In Theorem  \ref{Th 2ordercondition nonconvex malliavin}, we obtain a pointwise second-order necessary condition for stochastic optimal controls under  relatively weak assumptions on the control set $U$ through the perturbation technique of needle variation. However, this approach needs considerably high smoothness assumptions on the coefficients $b$, $\sigma$, $f$, and $h$ with respect to the state variable $x$ (differentiable with respect $x$ up to the forth order). Furthermore, four adjoint equations are introduced to represent this condition.  When the set $U$ has good structure such that the first- and second-order adjacent sets of $U$ on the boundary point of $U$ is nonempty (but $U$ is still allowed to be nonconvex), some perturbation technique from the classical variational analysis can be used to establish the second-order necessary conditions for stochastic optimal controls under lower regularity assumption on the coefficients $b$, $\sigma$, $f$, and $h$ (with respect to the state variable $x$) and only two adjoint equations are introduced to derive the  second-order necessary conditions. We refer the reader to \cite{FZZ15} for a detailed discussion in this respect.
\end{remark}

Two illustrative examples are as follows.

\begin{example}
{\em
Let
$$
\left\{
\begin{array}{l}
dx(t)=b(x(t))u(t)dt+u(t)dW(t),\qquad\qquad\qquad\qquad t\in[0,1],\\
x(0)=0,
\end{array}\right.
$$
$U=\big\{-1,\ 0,\ 1\big\}$,
and let
$$J(u(\cdot))=\frac{1}{2}\me\int_{0}^{1} |u(t)|^2dt-\frac{1}{2}\me~|x(1)|^2.$$

Assume that $b(\cdot):\mr\to\mr$ is bounded and continuously differentiable up to order 5 with bounded derivatives, $b_{x}(0)>0 $. Then, the conditions (C1)--(C2) hold.

For the above optimal control problem, the Hamiltonian is defined by
$$\hh(t,x,u,p_1,q_1)=p_1b(x)u+q_1u-\frac{1}{2}u^2,$$
$(t,x,u,p_1,q_1)\in [0,1]\times\mr\times U\times\mr\times\mr.$

Let $(\bar{x}(t),\bar{u}(t))\!\equiv\!(0,0)$.
The four adjoint equations with respect to $(\bar{x}(\cdot),\bar{u}(\cdot))$ are given below:
$$
\left\{
\begin{array}{l}
dp_{1}(t)=q_{1}(t)dW(t), \ t\in[0,1],\qquad\\
p_{1}(1)=0;
\end{array}\right.\qquad
\left\{
\begin{array}{l}
dp_{2}(t)=q_{2}(t)dW(t), \ t\in[0,1],\\
p_{2}(1)=1;
\end{array}\right.
$$
$$
\left\{
\begin{array}{l}
dp_{3}(t)=q_{3}(t)dW(t), \ t\in[0,1],\qquad\\
p_{3}(1)=0;
\end{array}\right.\qquad
\left\{
\begin{array}{l}
dp_{4}(t)=q_{4}(t)dW(t), \ t\in[0,1],\\
p_{4}(1)=0.
\end{array}\right.
$$

It is easy to check that
\begin{eqnarray*}
& &(p_{1}(t),q_{1}(t))=(0,0),\ \ (p_{2}(t),q_{2}(t))=(1,0),\ \ \\
& &(p_{3}(t),q_{3}(t))=(0,0),\ \ (p_{4}(t),q_{4}(t))=(0,0), \ \ \ \forall\ (\omega,t)\in\Omega\times[0,1];
\end{eqnarray*}
and,
\begin{eqnarray*}
& &\mh(t,\bar{x}(t),v)=0, \quad
\mss(t,\bar{x}(t),v)=b(0)v, \quad
\mt(t,\bar{x}(t),v)=2b_{x}(0)v,\\
& &\qquad \qquad\qquad\qquad\qquad\qquad\qquad\qquad\qquad\qquad
\forall  v\in U,\ \ \forall\ (\omega,t)\in\Omega\times[0,1].
\end{eqnarray*}
Thus, $\bar{u}(t)\equiv 0$ is a singular control in the sense of Pontryagin-type maximum principle on $U$.

Let $ v= 1$, we have
$$\mss(t,\bar{x}(t),v)=b(0)=\me~\mss(t,\bar{x}(t),v).$$
In this case, $\nabla\mss(t,\bar{x}(t),v)\equiv 0$, and
\begin{eqnarray*}
& &\inner{\mss(\t,\bar{x}(\t), v)}{
b(\t,\bar{x}(\t),v)-b(\t,\bar{x}(\t),\bar{u}(\t))}\\
& &+ \inner{\nabla\mss(\t,\bar{x}(\t),v)}{
\sigma(\t,\bar{x}(\t),v)-\sigma(\t,\bar{x}(\t),\bar{u}(\t))}\\
& &
+\frac{1}{2}\big\langle\mt(\t,\bar{x}(\t),v)
\big(\sigma(\t,\bar{x}(\t),v)-\sigma(\t,\bar{x}(\t),\bar{u}(\t))\big),\\
& &\qquad\qquad\qquad\quad\
\sigma(\t,\bar{x}(\t),v)-\sigma(\t,\bar{x}(\t),\bar{u}(\t))
\big\rangle\\
&=&b(0)^2+b_{x}(0)\\
&>& 0,\quad \forall (\omega,t)\in \Omega\times[0,1].
\end{eqnarray*}
Therefore, by Theorem \ref{Th 2ordercondition nonconvex malliavin}, $\bar{u}(t)\equiv 0$ is not an optimal control.
}
\end{example}

\begin{example}{\em
Let
$$
\left\{
\begin{array}{l}
dx(t)=(u(t)-1)dt+(x(t)-u(t))dW(t),\qquad t\in[0,1],\\
x(0)=1,
\end{array}\right.
$$
$U=\big\{-1,\ 0,\ 1\big\}$, and let
$$J(u(\cdot))=\frac{1}{24}\me~|x(1)-1|^4.$$

Obviously, $(\bar{x}(\cdot),\bar{u}(\cdot))\equiv(1,1)$ is the optimal pair.
The four adjoint equations with respect to $(\bar{x}(\cdot),\bar{u}(\cdot))$ are as follows:
$$
\left\{
\begin{array}{l}
dp_{1}(t)=-q_{1}(t)dt+q_{1}(t)dW(t), \qquad t\in[0,1],\qquad\qquad \qquad\ \ \\
p_{1}(1)=0;
\end{array}\right.
$$
$$
\left\{
\begin{array}{l}
dp_{2}(t)=-\Big[p_{2}(t)+2q_{2}(t)\Big]dt+q_{2}(t)dW(t), \qquad t\in[0,1],\qquad\\
p_{2}(1)=0;
\end{array}\right.
$$
$$
\left\{
\begin{array}{l}
dp_{3}(t)=-\Big[3p_{3}(t)+3q_{3}(t)\Big]dt+q_{3}(t)dW(t), \qquad t\in[0,1],\quad  \ \\
p_{3}(1)=0;
\end{array}\right.
$$
and
$$
\left\{
\begin{array}{l}
dp_{4}(t)=-\Big[6p_{4}(t)+4q_{4}(t)\Big]dt+q_{4}(t)dW(t), \qquad t\in[0,1],\quad
\\
p_{4}(1)=-1.
\end{array}\right.
$$
An easy computation shows that
\begin{eqnarray*}
& &(p_{1}(t),q_{1}(t))=(0,0),\ \ (p_{2}(t),q_{2}(t))=(0,0),\ \ \\
& &(p_{3}(t),q_{3}(t))=(0,0),\ \ (p_{4}(t),q_{4}(t))=(-e^{6-6t},0),\ \ \forall\ (\omega,t)\in\Omega\times[0,1].
\end{eqnarray*}
Then, we have
\begin{eqnarray*}
& &\mh(t,\bar{x}(t),v)=0,\quad
\mss(t,\bar{x}(t),v)=0, \quad
\mt(t,\bar{x}(t),v)=-\frac{1}{2}e^{6-6t}(v-1)^2,\\
& &\qquad \qquad \qquad \qquad \qquad \qquad \qquad \qquad \qquad \qquad
\forall v\in U,\ \ \forall\ (\omega,t)\in\Omega\times[0,1],
\end{eqnarray*}
and
\begin{eqnarray*}
& &\Big\langle\mt(t,\bar{x}(t),v)
\big(\sigma(\t,\bar{x}(\t),v)-\sigma(\t,\bar{x}(\t),\bar{u}(\t))\big),\\
& &\qquad \qquad \qquad \qquad \qquad \sigma(\t,\bar{x}(\t),v)-\sigma(\t,\bar{x}(\t),\bar{u}(\t))\Big\rangle\\
&=&-\frac{1}{2}e^{6-6t}(v-1)^4\le0,\quad
\forall v\in U,\ \ \forall\ (\omega,t)\in\Omega\times[0,1].
\end{eqnarray*}
Therefore, $\bar{u}(t)\equiv1$ is a singular optimal control on $U$, and the second-order necessary condition (\ref{2orderconditionth sequalszero}) holds.
}
\end{example}

\section{Proofs of the main results} This section is devoted to proving the main results of this paper, i.e., Theorems \ref{2orderconditionth} and  \ref{Th 2ordercondition nonconvex malliavin}. We need a known result.

\begin{lemma}\label{technical lemma}
{\rm (\cite[Lemma 4.1]{zhangH14a})} Let $\Phi(\cdot),\ \Psi(\cdot)\in L_{\mmf}^{2}(\Omega;L^{2}(0,T;\mrn))$. Then, for a.e. $\t\in [0,T)$, it holds that
\begin{equation}\label{technical lemma limit1}
\lim_{\ephs\to 0^+}\frac{1}{\ephs^2}\me\int_{\t}^{\t+\ephs}\Big\langle\Phi(\t),
\int_{\t}^{t} \Psi(s)ds\Big\rangle dt
=\frac{1}{2}\me~\inner{\Phi(\t)}{\Psi(\t)},
\end{equation}
\begin{equation}\label{technical lemma limit2}
\lim_{\ephs\to 0^+}\frac{1}{\ephs^2}\me\int_{\t}^{\t+\ephs}
\Big\langle\Phi(t), \int_{\t}^{t} \Psi(s)ds \Big\rangle dt
=\frac{1}{2}\me~\inner{\Phi(\t)}{\Psi(\t)}.
\end{equation}
\end{lemma}

\subsection{Proof of Theorem \ref{2orderconditionth}}
Since $u(t)\equiv v$, $v\in U$ is an admissible control, in this subsection, we shall still denote by $\delta\varphi(t)$ the increment $\varphi(t,\bar{x}(t),v)-\varphi(t,\bar{x}(t),\bar{u}(t))$ and by $\delta\varphi_{x}(t)$ the increment $\varphi_{x}(t,\bar{x}(t),v)-\varphi_{x}(t,\bar{x}(t),\bar{u}(t))$ for $\varphi=b, \sigma, f$.
We only need to prove the condition (\ref{2ordercondition}) holds for a.e. $\t\in [0,T)$. Let $\t\in [0,T)$, $\ephs\in(0,T-\t)$ and $E_{\varepsilon}=[\t,\t+\varepsilon)\subset [0,T).$
For any fixed $v\in V$, define
$$
u^{\varepsilon}(t)=\left\{
\begin{array}{l}
v, \qquad\qquad t\in E_{\varepsilon},\\
\bar{u}(t), \qquad \quad t\in [0,T] \setminus E_{\varepsilon}.\\
\end{array}\right.
$$
Clearly, $u^{\varepsilon}(\cdot)\in \mmu_{ad}$. Since $\bar{u}(\cdot)$ is a singular control on $V$ in the sense of Pontryagin-type maximum principle,
$$\mh(t,\bar{x}(t),v)=0,\qquad  a.e.\ (\omega,t)\in \Omega\times[0,T].$$
Then, by Proposition \ref{variational formulation for noncov}, we have
\begin{eqnarray*}
\qquad 0&\ge&\frac{J(\bar{u}(\cdot))-J(u^{\varepsilon}(\cdot))}{\varepsilon^2}\\
&=&\frac{1}{\varepsilon^2}\me\int_{0}^{T}\Big[
\mh(t,\bar{x}(t),v)+\Big\langle\mss(t,\bar{x}(t),v),\gamma(t)\Big\rangle\\
& &\qquad\qquad\qquad
+\frac{1}{2}\Big\langle\mt(t,\bar{x}(t),v)y_{1}^{\varepsilon}(t),
y_{1}^{\varepsilon}(t)\Big\rangle
\Big]\chi_{E_{\varepsilon}}(t)dt+o(1)\qquad\ (\ephs\to 0^+)\\
&=&\frac{1}{\varepsilon^2}\me\int_{0}^{T}
\Big[\Big\langle\mss(t,\bar{x}(t),v),y_{1}^{\varepsilon}(t)
+y_{2}^{\varepsilon}(t)\Big\rangle\chi_{E_{\varepsilon}}(t)\\
& &\qquad\qquad\qquad
+\frac{1}{2}\Big\langle\mt(t,\bar{x}(t),v)y_{1}^{\varepsilon}(t),
y_{1}^{\varepsilon}(t)\Big\rangle
\chi_{E_{\varepsilon}}(t)\Big]dt+ o(1),\qquad (\ephs\to 0^+).
\end{eqnarray*}

Now, we divide the proof of (\ref{2ordercondition}) into 4 steps.

\textbf{Step 1:} In this step, we prove that
\begin{eqnarray}\label{limt for Sy1}
& &\limsup_{\varepsilon\to 0^+}\frac{1}{\varepsilon^2}
\me\int_{0}^{T}\Big\langle\mss(t,\bar{x}(t),v),y_{1}^{\varepsilon}(t)\Big\rangle
\chi_{E_{\varepsilon}}(t)dt\nonumber\\
&=& \frac{1}{2}\partial^{+}_{\t}\Big(\mss(\t,\bar{x}(\t),v);
\delta\sigma(\t)\Big),\quad \ a.e.\ \t\in [0,T).
\end{eqnarray}

By  \cite[Theorem 1.6.14, p. 47)]{Yong99}, $y_{1}^{\varepsilon}(\cdot)$ has the following explicit representation:
\begin{eqnarray}\label{y1(t)}
y_{1}^{\varepsilon}(t)&=&-\Phi(t)\int_{0}^{t}\Phi(s)^{-1}\sigma_{x}(s)\delta\sigma(s)
\chi_{E_{\varepsilon}}(s)ds\nonumber\\
& &+\Phi(t)\int_{0}^{t}\Phi(s)^{-1}\delta\sigma(s)\chi_{E_{\varepsilon}}(s)dW(s).
\end{eqnarray}
Consequently,
\begin{eqnarray}\label{th41 eq1}
& &\frac{1}{\varepsilon^2}\me\int_{0}^{T}\Big\langle\mss(t,\bar{x}(t),v),
y_{1}^{\varepsilon}(t)\Big\rangle\chi_{E_{\varepsilon}}(t)dt\nonumber \\
&=&-\frac{1}{\varepsilon^2}\me\int_{\t}^{\t+\varepsilon}
\Big\langle\mss(t,\bar{x}(t),v), \Phi(t)\int_{\t}^{t}\Phi(s)^{-1}
\sigma_{x}(s)\delta\sigma(s)ds\Big\rangle dt\nonumber\\
& &+\frac{1}{\varepsilon^2}\me\int_{\t}^{\t+\varepsilon}
\Big\langle\mss(t,\bar{x}(t),v),\Phi(t)
\int_{\t}^{t}\Phi(s)^{-1}\delta\sigma(s)dW(s)\Big\rangle dt.
\end{eqnarray}
By Lemma \ref{technical lemma}, it follows that
\begin{eqnarray}\label{limit s part1}
& &\lim_{\varepsilon\to 0^+}\Big[-\frac{1}{\varepsilon^2}\me\int_{\t}^{\t+\varepsilon}
\Big\langle\mss(t,\bar{x}(t),v), \Phi(t)\int_{\t}^{t}\Phi(s)^{-1}
\sigma_{x}(s)\delta\sigma(s)ds\Big\rangle dt\Big]\quad\nonumber\\
&=&-\frac{1}{2}\me~\Big\langle\mss(\t,\bar{x}(\t),v),
\sigma_{x}(\t)\delta\sigma(\t)\Big\rangle,\ \ a.e.\ \  \t\in [0,T).
\end{eqnarray}
Next, by (\ref{Phi}), we deduce that
\begin{eqnarray}\label{th41 eq2}
& &\limsup_{\varepsilon\to 0^+}\frac{1}{\varepsilon^2}\me\int_{\t}^{\t+\varepsilon}
\Big\langle\mss(t,\bar{x}(t),v),\Phi(t)\int_{\t}^{t}\Phi(s)^{-1}\delta\sigma(s)
dW(s)\Big\rangle  dt\nonumber\\
&=&\limsup_{\varepsilon\to 0^+}\frac{1}{\varepsilon^2}\me\int_{\t}^{\t+\varepsilon}
\Big\langle\mss(t,\bar{x}(t),v),\Phi(\t)\int_{\t}^{t}\Phi(s)^{-1}\delta\sigma(s)
dW(s)\Big\rangle dt\nonumber\\
& &+\limsup_{\varepsilon\to 0^+}\frac{1}{\varepsilon^2}\me\int_{\t}^{\t+\varepsilon}
\Big\langle\mss(t,\bar{x}(t),v),\int_{\t}^{t}b_{x}(s)\Phi(s)ds\cdot\nonumber\\
& &\qquad\qquad\qquad\qquad\qquad\qquad\qquad\qquad\qquad\ \
\int_{\t}^{t}\Phi(s)^{-1}\delta\sigma(s)dW(s)
\Big\rangle dt\nonumber\\
& & +\limsup_{\varepsilon\to 0^+}\frac{1}{\varepsilon^2}\me\int_{\t}^{\t+\varepsilon}
\Big\langle\mss(t,\bar{x}(t),v),\int_{\t}^{t}\sigma_{x}(s)\Phi(s)dW(s)\cdot\nonumber\\
& &\qquad\qquad\qquad\qquad\qquad\qquad\qquad\qquad\qquad\ \
\int_{\t}^{t}\Phi(s)^{-1}\delta\sigma(s)dW(s)
\Big\rangle dt.\nonumber\\
\end{eqnarray}
By (\ref{martingale exp of mss(t)}) and (\ref{2ozxndition}), it holds that
\begin{eqnarray}\label{th41 eq3}
& &\limsup_{\varepsilon\to 0^+}\frac{1}{\varepsilon^2}\me\int_{\t}^{\t+\varepsilon}
\Big\langle\mss(t,\bar{x}(t),v), \Phi(\t)\int_{\t}^{t}\Phi(s)^{-1}\delta\sigma(s)
dW(s)\Big\rangle dt\nonumber\\
&=&\limsup_{\varepsilon\to 0^+}\frac{1}{\varepsilon^2}\me\int_{\t}^{\t+\varepsilon}
\Big\langle\me~\mss(t,\bar{x}(t),v)+\int_{0}^{t}\phi_{v}(s,t)dW(s),\nonumber\\
& & \qquad\qquad\qquad\qquad\qquad\qquad\qquad
\int_{\t}^{t}\Phi(\t)
\Phi(s)^{-1}\delta\sigma(s)dW(s)\Big\rangle dt\nonumber\\
&=&\limsup_{\varepsilon\to 0^+}\frac{1}{\varepsilon^2}\me\int_{\t}^{\t+\varepsilon}
\int_{\t}^{t}\Big\langle\phi_{v}(s,t),
\Phi(\t)\Phi(s)^{-1}\delta\sigma(s)\Big\rangle dsdt\nonumber\\
&=&\frac{1}{2}\partial^{+}_{\t}\Big(\mss(\t,\bar{x}(\t),v);
\delta\sigma(\t)\Big)
\qquad a.e. \ \t\in[0,T].
\end{eqnarray}
On the other hand,
\begin{eqnarray}\label{th41 eq3add}
& &\lim_{\varepsilon\to 0^+}\Big|\frac{1}{\varepsilon^2}\me\int_{\t}^{\t+\varepsilon}
\Big\langle\mss(t,\bar{x}(t),v),\int_{\t}^{t}b_{x}(s)\Phi(s)ds\cdot\nonumber\\
& & \qquad\qquad\qquad\qquad\quad\qquad\qquad\qquad\qquad
\int_{\t}^{t}\Phi(s)^{-1}\delta\sigma(s)dW(s)
\Big\rangle dt\Big|\nonumber\\
&\le&\lim_{\varepsilon\to 0^+}\frac{1}{\varepsilon^2}\int_{\t}^{\t+\varepsilon}
\Big[\Big(\me~\big|\mss(t,\bar{x}(t),v)\big|^2\Big)^{\frac{1}{2}}
\Big(\me~\Big|\int_{\t}^{t}b_{x}(s)\Phi(s)ds\Big|^4\Big)^{\frac{1}{4}}\cdot\nonumber\\
& &\qquad\qquad\qquad\qquad\qquad\qquad\qquad\quad
\Big(\me~\Big|\int_{\t}^{t}\Phi(s)^{-1}
\delta\sigma(s)dW(s)\Big|^4\Big)^{\frac{1}{4}}\Big]dt\nonumber\\
&\le&\lim_{\varepsilon\to 0^+}\frac{C}{\varepsilon^2}\int_{\t}^{\t+\varepsilon}
(t-\t)^{\frac{3}{2}}\Big(\me~\big|\mss(t,\bar{x}(t),v)\big|^2\Big)^{\frac{1}{2}}dt\nonumber\\
&=& 0  \ \ \qquad\qquad \ \ \ \ a.e.\ \ \t\in[t_{i},T].
\end{eqnarray}
Also,
\begin{eqnarray}\label{th41 eq4}
& &\lim_{\varepsilon\to 0^+}\Big|\frac{1}{\varepsilon^2}\me\int_{\t}^{\t+\varepsilon}
\Big\langle\mss(t,\bar{x}(t),v),\int_{\t}^{t}\sigma_{x}(s)\Phi(s)dW(s)\cdot\nonumber\\
& & \qquad\quad
\int_{\t}^{t}\Phi(s)^{-1}\delta\sigma(s)dW(s)\Big\rangle dt -\frac{1}{2}\me~\Big\langle\mss(\t,\bar{x}(\t),v),
\sigma_{x}(\t)\delta\sigma(\t)\Big\rangle\Big|\nonumber\\
\quad\ \ &\le&\lim_{\varepsilon\to 0^+}\Big|\frac{1}{\varepsilon^2}\me\int_{\t}^{\t+\varepsilon}
\Big\langle\mss(t,\bar{x}(t),v)-\mss(\t,\bar{x}(\t),v),\nonumber\\
& &\qquad\qquad\qquad\qquad\quad
\int_{\t}^{t}\sigma_{x}(s)\Phi(s)dW(s)
\int_{\t}^{t}\Phi(s)^{-1}\delta\sigma(s)dW(s)
\Big\rangle dt\Big|\nonumber\\
& &+\lim_{\varepsilon\to 0^+} \Big|
\frac{1}{\varepsilon^2}\me\int_{\t}^{\t+\varepsilon}
\Big\langle\mss(\t,\bar{x}(\t),v),\int_{\t}^{t}\sigma_{x}(s)\Phi(s)dW(s)\cdot\nonumber\\
& & \qquad\quad\
\int_{\t}^{t}\Phi(s)^{-1}\delta\sigma(s)dW(s)\Big\rangle dt -\frac{1}{2}\me~\Big\langle\mss(\t,\bar{x}(\t),v),
\sigma_{x}(\t)\delta\sigma(\t)\Big\rangle \Big|\nonumber\\
&\le& \lim_{\varepsilon\to 0^+}\frac{1}{\varepsilon^2}\int_{\t}^{\t+\varepsilon}
\Big[\me~\big|\mss(t,\bar{x}(t),v)
-\mss(\t,\bar{x}(\t),v)\big|^2\Big]^{\frac{1}{2}}\cdot\nonumber\\
& &\qquad\quad
\Big[\me~\Big(\int_{\t}^{t}\big|\sigma_{x}(s)\Phi(s)
\big|^2ds\Big)^2\Big]^{\frac{1}{4}}
\Big[\me~\Big(\int_{\t}^{t}\big|\Phi(s)^{-1}\delta\sigma(s)
\big|^2ds\Big)^2\Big]^{\frac{1}{4}}dt\nonumber\\
& &  +\lim_{\varepsilon\to 0^+}\Big|\frac{1}{\varepsilon^2}\int_{\t}^{\t+\varepsilon}
\int_{\t}^{t}\me~\Big\langle\mss(\t,\bar{x}(\t),v),\sigma_{x}(s)\delta\sigma(s)
-\sigma_{x}(\t)\delta\sigma(\t)\Big\rangle dsdt\Big|\nonumber\\
&\le&\lim_{\varepsilon\to 0^+}\frac{C}{\varepsilon}\int_{\t}^{\t+\varepsilon}
\Big[\me~\big|\mss(t,\bar{x}(t),v)-\mss(\t,\bar{x}(\t),v)
\big|^2\Big]^{\frac{1}{2}}dt\nonumber\\
& & +\lim_{\varepsilon\to 0^+}\Big|\frac{1}{\varepsilon^2}\int_{\t}^{\t+\varepsilon}
\int_{\t}^{t}\me~\Big\langle\mss(\t,\bar{x}(\t),v),\sigma_{x}(s)\delta\sigma(s)
-\sigma_{x}(\t)\delta\sigma(\t)\Big\rangle dsdt\Big|\nonumber\\
&=&0, \qquad\  \ a.e.\ \ \t\in[0,T).
\end{eqnarray}
Then, by (\ref{th41 eq2})--(\ref{th41 eq4}), it follows that, for
 a.e.  $\t\in[0,T)$,
\begin{eqnarray}\label{Limit S part2}
& &\limsup_{\varepsilon\to 0^+}\frac{1}{\varepsilon^2}\me\int_{\t}^{\t+\varepsilon}
\Big\langle \mss(t,\bar{x}(t),v), \Phi(t)
\int_{\t}^{t}\Phi(s)^{-1}\delta\sigma(s)dW(s)\Big\rangle dt \nonumber\\
&=& \frac{1}{2}\partial^{+}_{\t}\Big(\mss(\t,\bar{x}(\t),v);
\delta\sigma(\t)\Big)+\frac{1}{2}\me~\Big\langle\mss(\t,\bar{x}(\t),v),
\sigma_{x}(\t)\delta\sigma(\t)\Big\rangle.
\end{eqnarray}

Combining (\ref{limit s part1}) with (\ref{Limit S part2}), we obtain (\ref{limt for Sy1}).

\textbf{Step 2:} In this step, we prove that, for  a.e. $\t\in[0,T)$,
\begin{eqnarray}\label{conclusion in step2}
\quad\lim_{\varepsilon\to 0^+}\frac{1}{\varepsilon^2}
\me\int_{0}^{T}\Big\langle\mss(t,\bar{x}(t),v),
y_{2}^{\varepsilon}(t)\Big\rangle\chi_{E_{\varepsilon}}(t)dt
= \frac{1}{2}\me ~\Big\langle\mss(\t,\bar{x}(\t),v), \delta b(\t)\Big\rangle.
\end{eqnarray}

Similar to (\ref{y1(t)}), the explicit representation of $y_{2}^{\varepsilon}(\cdot)$ is given as follows:
\begin{eqnarray}\label{y2(t)}
y_{2}^{\varepsilon}(t)
&=&\Phi(t)\int_{0}^{t}\Phi(s)^{-1}
\Big[\frac{1}{2}b_{xx}(s)\big(y_{1}^{\varepsilon}(s),y_{1}^{\varepsilon}(s)\big)
+\delta b(s)\chi_{E_{\varepsilon}}(s)\nonumber\\
& &\qquad\qquad\
-\frac{1}{2}\sigma_{x}(s)
\sigma_{xx}(s)\big(y_{1}^{\varepsilon}(s),y_{1}^{\varepsilon}(s)\big)
-\sigma_{x}(s)\delta\sigma_{x}(s)y_{1}^{\varepsilon}(s)
\chi_{E_{\varepsilon}}(s)\Big]ds\nonumber\\
& &+\Phi(t)\int_{0}^{t}\Phi(s)^{-1}
\Big[\frac{1}{2}\sigma_{xx}(s)\big(y_{1}^{\varepsilon}(s),y_{1}^{\varepsilon}(s)\big)
+\delta\sigma_{x}(s)y_{1}^{\varepsilon}(s)\chi_{E_{\varepsilon}}(s)\Big]dW(s).
\end{eqnarray}
Then,
\begin{eqnarray}\label{s(t)y_2(t)}
& &\lim_{\varepsilon\to 0^+}\frac{1}{\varepsilon^2}
\me\int_{0}^{T}\Big\langle\mss(t,\bar{x}(t),v), y_{2}^{\varepsilon}(t)\Big\rangle
\chi_{E_{\varepsilon}}(t)dt\nonumber\\
&=&\lim_{\varepsilon\to 0^+}\frac{1}{2\varepsilon^2}\me\int_{\t}^{\t+\varepsilon}\Big\langle
\mss(t,\bar{x}(t),v),\Phi(t)\int_{0}^{t}\Phi(s)^{-1}\Big[b_{xx}(s)
\big(y_{1}^{\varepsilon}(s),y_{1}^{\varepsilon}(s)\big)\nonumber\\
& &\qquad\qquad\qquad\qquad\qquad\qquad\qquad\qquad\qquad
-\sigma_{x}(s)\sigma_{xx}(s)
\big(y_{1}^{\varepsilon}(s),y_{1}^{\varepsilon}(s)\big)\Big]ds\Big\rangle dt\nonumber\\
& &+\lim_{\varepsilon\to 0^+}\frac{1}{\varepsilon^2}\me\int_{\t}^{\t+\varepsilon}
\Big\langle\mss(t,\bar{x}(t),v),\Phi(t)\int_{0}^{t}\Phi(s)^{-1}\delta b(s)\chi_{E_{\varepsilon}}(s)ds\Big\rangle dt\nonumber\\
& &-\lim_{\varepsilon\to 0^+}\frac{1}{\varepsilon^2}\me\int_{\t}^{\t+\varepsilon}
\Big\langle\mss(t,\bar{x}(t),v),
\Phi(t)\int_{0}^{t}\Phi(s)^{-1}\sigma_{x}(s)\delta\sigma_{x}(s)y_{1}^{\varepsilon}(s)
\chi_{E_{\varepsilon}}(s)ds\Big\rangle dt\nonumber\\
& &+\lim_{\varepsilon\to 0^+}\frac{1}{\varepsilon^2}\me\int_{\t}^{\t+\varepsilon}
\Big\langle\mss(t,\bar{x}(t),v), \Phi(t)\int_{0}^{t}\Phi(s)^{-1}
\Big[\frac{1}{2}\sigma_{xx}(s)
\big(y_{1}^{\varepsilon}(s),y_{1}^{\varepsilon}(s)\big)\nonumber\\
& &\qquad\qquad\quad \qquad\qquad\quad \qquad\qquad\quad \qquad\qquad\
+\delta\sigma_{x}(s)y_{1}^{\varepsilon}(s)
\chi_{E_{\varepsilon}}(s)\Big]dW(s)\Big\rangle dt.\nonumber\\
\end{eqnarray}

By (\ref{y1(t)}),  $y_{1}^{\varepsilon}(t)=0$ for any $t\in[0,\t)$. Therefore, by Lemmas \ref{estimateofvariequ} and \ref{bouned of S},
\begin{eqnarray}\label{s(t)y_2(t) prat1}
& &\lim_{\varepsilon\to 0^+}\Big|\frac{1}{2\varepsilon^2}
\me\int_{\t}^{\t+\varepsilon}\Big\langle
\mss(t,\bar{x}(t),v),\Phi(t)\int_{0}^{t}\Phi(s)^{-1}\Big[b_{xx}(s)
\big(y_{1}^{\varepsilon}(s),y_{1}^{\varepsilon}(s)\big)\nonumber\\
& &\qquad\qquad\qquad\qquad\qquad\qquad\qquad\qquad\qquad
-\sigma_{x}(s)\sigma_{xx}(s)
\big(y_{1}^{\varepsilon}(s),y_{1}^{\varepsilon}(s)\big)\Big]ds\Big\rangle dt\Big|\nonumber\\
&=&\lim_{\varepsilon\to 0^+}\Big|\frac{1}{2\varepsilon^2}
\me\int_{\t}^{\t+\varepsilon}\Big\langle
\mss(t,\bar{x}(t),v),\Phi(t)\int_{\t}^{t}\Phi(s)^{-1}\Big[b_{xx}(s)
\big(y_{1}^{\varepsilon}(s),y_{1}^{\varepsilon}(s)\big)\nonumber\\
& &\qquad\qquad\qquad\qquad\qquad\qquad\qquad\qquad\qquad
-\sigma_{x}(s)\sigma_{xx}(s)
\big(y_{1}^{\varepsilon}(s),y_{1}^{\varepsilon}(s)\big)\Big]ds\Big\rangle dt\Big|\nonumber\\
&\le& \lim_{\varepsilon\to 0^+}\frac{1}{2\varepsilon^2}
\Big\{\int_{\t}^{\t+\varepsilon}\me~
\big|\mss(t,\bar{x}(t),v)\big|^2dt\Big\}^{\frac{1}{2}}
\Big\{\int_{\t}^{\t+\varepsilon}\me~\Big|
\Phi(t)\int_{\t}^{t}\Phi(s)^{-1}\cdot\nonumber\\
& &\qquad\qquad\quad\qquad\ \
\Big[b_{xx}(s)\big(y_{1}^{\varepsilon}(s),y_{1}^{\varepsilon}(s)\big)
-\sigma_{x}(s)\sigma_{xx}(s)
\big(y_{1}^{\varepsilon}(s),y_{1}^{\varepsilon}(s)\big)\Big]
ds\Big|^2dt\Big\}^{\frac{1}{2}}\nonumber\\
&\le& \lim_{\varepsilon\to 0^+}\frac{C}{\varepsilon^2}
\Big\{\int_{\t}^{\t+\varepsilon}\me~
\big|\mss(t,\bar{x}(t),v)\big|^2dt\Big\}^{\frac{1}{2}}\cdot\nonumber\\
& &\qquad\quad\qquad\qquad\
\Big\{\int_{\t}^{\t+\varepsilon}(t-\t)^{2}\me~\Big[\sup_{s\in[0,T]}
\big|\Phi(t)\Phi(s)^{-1}\big|\cdot\sup_{s\in[0,T]}
\big|y_{1}^{\varepsilon}(s)\big|^{2}\Big]^{2}dt
\Big\}^{\frac{1}{2}}\nonumber\\
&\le& \lim_{\varepsilon\to 0^+}\frac{C}{\varepsilon^{\frac{1}{2}}}
\Big\{\int_{\t}^{\t+\varepsilon}\me~
\big|\mss(t,\bar{x}(t),v)\big|^2dt\Big\}^{\frac{1}{2}}\cdot
\Big\{\me~\Big[\sup_{t\in[0,T]}
\big|y_{1}^{\varepsilon}(t)\big|^{8}\Big]\Big\}^{\frac{1}{4}}\nonumber\\
&=& 0, \qquad\ a.e.\ \ \t\in[0,T).\nonumber\\
\end{eqnarray}

Next, from Lemma \ref{technical lemma} we conclude that
\begin{eqnarray}\label{th41 eq6}
& &\lim_{\varepsilon\to 0^+}\frac{1}{\varepsilon^2}\me\int_{\t}^{\t+\varepsilon}
\Big\langle\mss(t,\bar{x}(t),v), \Phi(t)\int_{0}^{t}\Phi(s)^{-1}\delta b(s)
\chi_{E_{\varepsilon}}(s)ds \Big\rangle dt\nonumber\\
&=& \frac{1}{2}\me~\Big\langle\mss(\t,\bar{x}(\t),v),\delta b(\t)\Big\rangle\qquad a.e.\ \t\in [0,T).
\end{eqnarray}

Also, by Lemmas \ref{estimateofvariequ} and \ref{bouned of S}, we deduce  that
\begin{eqnarray}\label{th41 eq7}
& &\lim_{\varepsilon\to 0^+}\Big|\frac{1}{\varepsilon^2}\me\int_{\t}^{\t+\varepsilon}
\Big\langle \mss(t,\bar{x}(t),v), \Phi(t)
\int_{0}^{t}\Phi(s)^{-1}\sigma_{x}(s)\delta\sigma_{x}(s)y_{1}^{\varepsilon}(s)
\chi_{E_{\varepsilon}}(s)ds\Big\rangle dt\Big|\nonumber\\
&=&\lim_{\varepsilon\to 0^+}\Big|\frac{1}{\varepsilon^2}\me\int_{\t}^{\t+\varepsilon}
\Big\langle\mss(t,\bar{x}(t),v), \Phi(t)
\int_{\t}^{t}\Phi(s)^{-1}\sigma_{x}(s)\delta\sigma_{x}(s)
y_{1}^{\varepsilon}(s)ds\Big\rangle dt\Big|\nonumber\\
&\le&\lim_{\varepsilon\to 0^+}\frac{1}{\varepsilon^2}
\Big[\int_{\t}^{\t+\varepsilon}
\me~\big|\mss(t,\bar{x}(t),v)\big|^2dt\Big]^{\frac{1}{2}}\cdot\nonumber\\
& &\qquad\qquad\qquad\qquad\qquad
\Big[\int_{\t}^{\t+\varepsilon}\me~\big|\Phi(t)\int_{\t}^{t}
\Phi(s)^{-1}\sigma_{x}(s)\delta\sigma_{x}(s)y_{1}^{\varepsilon}(s)ds\big|^2dt
\Big]^{\frac{1}{2}}\nonumber\\
&\le&\lim_{\varepsilon\to 0^+}\frac{C}{\varepsilon^2}
\Big\{\int_{\t}^{\t+\varepsilon}
\me~\big|\mss(t,\bar{x}(t),v)\big|^2dt\Big\}^{\frac{1}{2}}\cdot\nonumber\\
& &\qquad\qquad\qquad\quad
\Big\{\int_{\t}^{\t+\varepsilon}(t-\t)^2
\me~\Big[\sup_{s\in[0,T]}\big|
\Phi(t)\Phi(s)^{-1}\big|\cdot\sup_{s\in[0,T]} \big|y_{1}^{\varepsilon}(s)\big|\Big]^{2}dt\Big\}^{\frac{1}{2}}\nonumber\\
&\le&\lim_{\varepsilon\to 0^+}\frac{C}{\varepsilon^{\frac{1}{2}}}
\Big\{\int_{\t}^{\t+\varepsilon}
\me~\big|\mss(t,\bar{x}(t),v)\big|^2dt\Big\}^{\frac{1}{2}}\cdot
\Big\{\me~\Big[\sup_{s\in[0,T]} |y_{1}^{\varepsilon}(s)|^4\Big]\Big\}^{\frac{1}{4}}\nonumber\\
&=&0,\ \qquad\ \ a.e.\ \ \t\in[0,T).\nonumber\\
\end{eqnarray}

In a similar way, we obtain that
\begin{eqnarray}\label{th41 eq8}
& &\lim_{\varepsilon\to 0^+}\Big|\frac{1}{\varepsilon^2}\me\int_{\t}^{\t+\varepsilon}
\Big\langle\mss(t,\bar{x}(t),v),\Phi(t)\int_{0}^{t}\Phi(s)^{-1}
\Big[\frac{1}{2}\sigma_{xx}(s)
\big(y_{1}^{\varepsilon}(s),y_{1}^{\varepsilon}(s)\big)\nonumber\\
& &\qquad\qquad\qquad\qquad\qquad\qquad\qquad\qquad\qquad
+\delta\sigma_{x}(s)y_{1}^{\varepsilon}(s)\chi_{E_{\varepsilon}}(s)
\Big]dW(s)\Big\rangle dt\Big|\nonumber\\
&=&\lim_{\varepsilon\to 0^+}\Big|\frac{1}{\varepsilon^2}\me\int_{\t}^{\t+\varepsilon}
\Big\langle\mss(t,\bar{x}(t),v),\Phi(t)\int_{\t}^{t}\Phi(s)^{-1}
\Big[\frac{1}{2}\sigma_{xx}(s)
\big(y_{1}^{\varepsilon}(s),y_{1}^{\varepsilon}(s)\big)\nonumber\\
& &\qquad\qquad\qquad\qquad\qquad\qquad\qquad\qquad\qquad\qquad\quad
+\delta\sigma_{x}(s)y_{1}^{\varepsilon}(s)
\Big]dW(s)\Big\rangle dt\Big|\nonumber\\
&\le&\lim_{\varepsilon\to 0^+}\Big|\frac{1}{\varepsilon^2}\me\int_{\t}^{\t+\varepsilon}\Big\langle
\Phi(t)^{\top}\mss(t,\bar{x}(t),v)-\Phi(\t)^{\top}\mss(\t,\bar{x}(\t),v),\nonumber\\
& &\qquad\quad\qquad\quad
\int_{\t}^{t}\Phi(s)^{-1}
\Big[\frac{1}{2}\sigma_{xx}(s)
\big(y_{1}^{\varepsilon}(s),y_{1}^{\varepsilon}(s)\big)
+\delta\sigma_{x}(s)y_{1}^{\varepsilon}(s)
\Big]dW(s)\Big\rangle dt\Big|\nonumber\\
& &+\lim_{\varepsilon\to 0^+}\Big|\frac{1}{\varepsilon^2}\int_{\t}^{\t+\varepsilon}
\me\Big\langle\mss(\t,\bar{x}(\t),v),
\Phi(\t)\int_{\t}^{t}\Phi(s)^{-1}\Big[\frac{1}{2}\sigma_{xx}(s)
\big(y_{1}^{\varepsilon}(s),y_{1}^{\varepsilon}(s)\big)\nonumber\\
& &\qquad\qquad\qquad\qquad\qquad\qquad\qquad\qquad\qquad\qquad\quad
+\delta\sigma_{x}(s)y_{1}^{\varepsilon}(s)
\Big]dW(s)\Big\rangle dt\Big|\nonumber\\
&\le&\lim_{\varepsilon\to 0^+}\frac{1}{\varepsilon^2}\Big\{\int_{\t}^{\t+\varepsilon}
\me~\Big|\Phi(t)^{\top}\mss(t,\bar{x}(t),v)-
\Phi(\t)^{\top}\mss(\t,\bar{x}(\t),v)\Big|^2dt\Big\}^{\frac{1}{2}}\cdot\nonumber\\
& &\qquad\
\Big\{\int_{\t}^{\t+\varepsilon}
\int_{\t}^{t}\me~\Big|\Phi(s)^{-1}
\Big[\frac{1}{2}\sigma_{xx}(s)
\big(y_{1}^{\varepsilon}(s),y_{1}^{\varepsilon}(s)\big)
+\delta\sigma_{x}(s)y_{1}^{\varepsilon}(s)
\Big]
\Big|^2dsdt\Big\}^{\frac{1}{2}}\nonumber\\
&\le&\lim_{\varepsilon\to 0^+}\frac{C}{\varepsilon}\Big\{\int_{\t}^{\t+\varepsilon}
\me~\Big|\Phi(t)^{\top}\mss(t,\bar{x}(t),v)-
\Phi(\t)^{\top}\mss(\t,\bar{x}(\t),v)\Big|^2dt\Big\}^{\frac{1}{2}}\cdot\nonumber\\
& &\qquad\qquad\qquad\qquad\qquad
\Big\{\Big[\me\Big(\sup_{t\in[0,T]}\big|y_{1}^{\varepsilon}(t)
\big|^8\Big)\Big]^{\frac{1}{2}}
+\Big[\me\Big(\sup_{t\in[0,T]}\big|y_{1}^{\varepsilon}(t)
\big|^4\Big)\Big]^{\frac{1}{2}}\Big\}^{\frac{1}{2}}\nonumber\\
&=&0, \ \ \qquad \ a.e.\ \ \t\in[0,T).\nonumber\\
\end{eqnarray}
Here, we have used the fact that
 $$
 \me\Big\langle\mss(\t,\bar{x}(\t),v),
\Phi(\t)\int_{\t}^{t}\Phi(s)^{-1}\Big[\frac{1}{2}\sigma_{xx}(s)
\big(y_{1}^{\varepsilon}(s),y_{1}^{\varepsilon}(s)\big)
+\delta\sigma_{x}(s)y_{1}^{\varepsilon}(s)
\Big]dW(s)\Big\rangle=0 ,
$$
for any $t\in [\tau,T]$.

Combining (\ref{s(t)y_2(t) prat1})--(\ref{th41 eq8}) with (\ref{s(t)y_2(t)}), we obtain (\ref{conclusion in step2}).

\textbf{Step 3:} In this step, we prove that
\begin{eqnarray}\label{th41 step3}
& &\lim_{\varepsilon\to 0^+}\frac{1}{\varepsilon^2}
\me\int_{0}^{T}\Big\langle\mt(t,\bar{x}(t),v)y_{1}^{\varepsilon}(t),
y_{1}^{\varepsilon}(t)\Big\rangle
\chi_{E_{\varepsilon}}(t)dt
=\frac{1}{2}\me ~\Big\langle\mt(\t,\bar{x}(\t),v)\delta\sigma(\t),\delta\sigma(\t)\Big\rangle.\nonumber\\
\end{eqnarray}

Similar to the pervious discusses, we have
\begin{eqnarray*}
& &\lim_{\varepsilon\to 0^+}\frac{1}{\varepsilon^2}
\me\int_{0}^{T}\Big\langle\mt(t,\bar{x}(t),v)y_{1}^{\varepsilon}(t),
y_{1}^{\varepsilon}(t)\Big\rangle
\chi_{E_{\varepsilon}}(t)dt\nonumber\\
&=&\lim_{\varepsilon\to 0^+}\frac{1}{\varepsilon^2}
\me\int_{\t}^{\t+\varepsilon}\Big\langle
\mt(t,\bar{x}(t),v)
\Big(-\Phi(t)\int_{\t}^{t}\Phi(s)^{-1}\sigma_{x}(s)\delta\sigma(s)
ds\nonumber\\
& & \qquad\qquad \quad\qquad\qquad \quad\qquad\qquad\qquad
+\Phi(t)\int_{\t}^{t}\Phi(s)^{-1}\delta\sigma(s)dW(s)\Big),
\nonumber\\
& &\qquad \qquad
-\Phi(t)\int_{\t}^{t}\Phi(s)^{-1}\sigma_{x}(s)\delta\sigma(s)ds
+\Phi(t)\int_{\t}^{t}\Phi(s)^{-1}\delta\sigma(s)dW(s)\Big\rangle dt\\
&=&\lim_{\varepsilon\to 0^+}\frac{1}{\varepsilon^2}
\me\int_{\t}^{\t+\varepsilon}\Big\langle
\mt(t,\bar{x}(t),v)\Phi(t)\int_{\t}^{t}\Phi(s)^{-1}\delta\sigma(s)dW(s),\\
& &\qquad \qquad\qquad \qquad\qquad \qquad\qquad \qquad\qquad
\Phi(t)\int_{\t}^{t}\Phi(s)^{-1}\delta\sigma(s)dW(s)\Big\rangle dt\\
&=&\frac{1}{2}\me ~\Big\langle\mt(\t,\bar{x}(\t),v)\delta\sigma(\t),\delta\sigma(\t)\Big\rangle,
 \qquad \ \ \ \ a.e.\ \ \t\in[0,T).
\end{eqnarray*}
This proves (\ref{th41 step3}).

\textbf{Step 4:}
From Step 1--Step 3, we have proved that, for any $v\in V$,
\begin{eqnarray*}
0&\ge&\limsup_{\varepsilon\to 0^+}\frac{J(\bar{u}(\cdot))-J(u^{\varepsilon}(\cdot))}{\varepsilon^2}\\
&=&\frac{1}{2}\me ~\Big\langle\mss(\t,\bar{x}(\t),v),\delta b(\t)\Big\rangle+\frac{1}{2}\partial^{+}_{\t}\Big(\mss(\t,\bar{x}(\t),v);
\delta\sigma(\t)\Big)\\
& &+\frac{1}{4}\me ~\Big\langle\mt(\t,\bar{x}(\t),v)\delta\sigma(\t),\delta\sigma(\t)\Big\rangle,
\ \qquad  a.e.\  \t\in [0,T).
\end{eqnarray*}
Therefore, for any $v\in V$, it follows that
\begin{eqnarray*}
& &\me ~\Big\langle\mss(\t,\bar{x}(\t),v),\delta b(\t)\Big\rangle+\partial^{+}_{\t}\Big(\mss(\t,\bar{x}(\t),v);
\delta\sigma(\t)\Big)\\
& &+\frac{1}{2}\me ~\Big\langle\mt(\t,\bar{x}(\t),v)\delta\sigma(\t),\delta\sigma(\t)\Big\rangle\le 0 \ \ \  a.e.\  \t\in [0,T).\qquad
\end{eqnarray*}
This completes the proof of Theorem \ref{2orderconditionth}.

\subsection{Proof of Theorem \ref{Th 2ordercondition nonconvex malliavin}} We borrow some idea from the proof of  \cite[Theorem 3.9]{zhangH14a}.
Denote by $\{t_{i}\}_{i=1}^{\infty}$ the totality of rational number in $[0,T)$, by $\{v^{k}\}_{k=1}^{\infty}$ a dense subset of $V$, and by $\{A_{ij}\}_{j=1}^{\infty}$ the countable subfamily of $\mf_{t_i}$, $i\in \mn$ such that for any $A\in \mf_{t_i}$, there exists $\{A_{ij_{n}}\}_{n=1}^{\infty}\subset \{A_{ij}\}_{j=1}^{\infty}$ such that
$\lim_{n\to \infty} P(A \Delta A_{ij_{n}})=0$,
where $A \Delta A_{ij_{n}}=(A\setminus A_{ij_{n}})\cup (A_{ij_{n}}\setminus A)$.

For any fixed $t_{i}$, $v^{k}$ and $A_{ij}\in \mf_{t_i}$, let $\t\in [t_{i},T)$,  $\ephs\in(0,T-\t)$, $E_{\varepsilon}=[\t,\t+\varepsilon)$, and write
$
u^{k}_{ij}(\omega,t)=\left\{
\begin{array}{l}
\bar{u}(\omega,t), \;\; (\omega,t)\in (\Omega\times [0,T])\setminus(A_{ij}\times [t_{i},T]),\\
v^{k}, \qquad\;\, (\omega,t)\in A_{ij}\times [t_{i},T].\\
\end{array}\right.
$
Clearly, $u^{k}_{ij}(\cdot)\in\mmu_{ad}$. Put
$
\hat{u}^{\varepsilon}(t)=\left\{
\begin{array}{l}
u^{k}_{ij}(t), \;\ t\in E_{\varepsilon},\\
\bar{u}(t), \quad \; t\in  [0,T] \setminus  E_{\varepsilon}.\\
\end{array}\right.
$
By Proposition \ref{variational formulation for noncov} and using the condition (\ref{singularcont concept}), we have
\begin{eqnarray}\label{step0in th mallivain}
0&\ge&\frac{J(\bar{u}(\cdot))-J(\hat{u}^{\varepsilon}(\cdot))}{\varepsilon^2}\nonumber\\
&=&\frac{1}{\varepsilon^2}\me\int_{0}^{T}\Big[
\mh(t,\bar{x}(t),v^{k})+\inner{\mss(t,\bar{x}(t),v^{k})}{
\hat{y}_{1}^{\varepsilon}(t)+\hat{y}_{2}^{\varepsilon}(t)}\nonumber\\
& &\qquad\qquad\quad
+\frac{1}{2}\inner{\mt(t,\bar{x}(t),v^{k})\hat{y}_{1}^{\varepsilon}(t)}
{\hat{y}_{1}^{\varepsilon}(t)}
\Big]\chi_{A_{ij}}\chi_{E_{\varepsilon}}(t)dt+o(1)\quad (\ephs\to 0^+)\nonumber\\
&=&\frac{1}{\varepsilon^2}\me\int_{0}^{T}\Big[
\inner{\mss(t,\bar{x}(t),v^{k})}{
\hat{y}_{1}^{\varepsilon}(t)+\hat{y}_{2}^{\varepsilon}(t)}\nonumber\\
& &\qquad\qquad\quad
+\frac{1}{2}\inner{\mt(t,\bar{x}(t),v^{k})\hat{y}_{1}^{\varepsilon}(t)}
{\hat{y}_{1}^{\varepsilon}(t)}
\Big]\chi_{A_{ij}}\chi_{E_{\varepsilon}}(t)dt+o(1)\quad (\ephs\to 0^+).
\end{eqnarray}
where $\hat{y}_{1}^{\varepsilon}(\cdot),\ \hat{y}_{2}^{\varepsilon}(\cdot)$ are the solutions to the variational equations (\ref{firstvariequ}) and (\ref{secondvariequ}) with respect to $\hat{u}^{\varepsilon}(\cdot)$, respectively.

We first prove that there exists a sequence $\{\varepsilon_{\ell}\}_{\ell=1}^{\infty}$, $\varepsilon_{\ell}\to 0^+$ as $\ell\to \infty$, and,
\begin{eqnarray}\label{step1in th mallivain}
\quad & &\lim_{\ell\to \infty}\frac{1}{\varepsilon_{\ell}^2}
\me\int_{0}^{T}\Big\langle\mss(t,\bar{x}(t),v^{k}),
\hat{y}_{1}^{\varepsilon_{\ell}}(t)\Big\rangle
\chi_{A_{ij}}\chi_{E_{\varepsilon_{\ell}}}(t)dt\\
&=&\frac{1}{2}\me\Big[\Big\langle\nabla\mss(\t,\bar{x}(\t),v^{k}),
\sigma(\t,\bar{x}(\t),v^{k})
-\sigma(\t,\bar{x}(\t),\bar{u}(\t))\Big\rangle\chi_{A_{ij}}\Big],\quad a.e.\ \t\in [t_{i},T],\nonumber
\end{eqnarray}

By (\ref{y1(t)}), $\hat{y}_{1}^{\varepsilon}(\cdot)$ enjoys the following explicit representation:
\begin{eqnarray}\label{haty1(t)}
\hat{y}_{1}^{\varepsilon}(t)&=&-\Phi(t)\int_{0}^{t}\Phi(s)^{-1}\sigma_{x}(s)
\big(\sigma(s,\bar{x}(s),v^{k})-\sigma(s,\bar{x}(s),\bar{u}(s))\big)\chi_{A_{ij}}
\chi_{E_{\varepsilon}}(s)ds\nonumber\\
& &+\Phi(t)\int_{0}^{t}\Phi(s)^{-1}
\big(\sigma(s,\bar{x}(s),v^{k})-\sigma(s,\bar{x}(s),\bar{u}(s))\big)\chi_{A_{ij}}
\chi_{E_{\varepsilon}}(s)dW(s).
\end{eqnarray}
Then,
\begin{eqnarray}\label{th pointwise mallivain eq1}
& &\frac{1}{\varepsilon^2}\me\int_{0}^{T}\Big\langle\mss(t,\bar{x}(t),v^{k}),
\hat{y}_{1}^{\varepsilon}(t)\Big\rangle\chi_{A_{ij}}\chi_{E_{\varepsilon}}(t)dt\nonumber \\
&=&-\frac{1}{\varepsilon^2}\me\int_{\t}^{\t+\varepsilon}
\Big\langle\mss(t,\bar{x}(t),v^{k}),
\Phi(t)\int_{\t}^{t}\Phi(s)^{-1}
\sigma_{x}(s)\nonumber\\
& &\qquad\qquad\qquad \qquad\qquad
\cdot\big(\sigma(s,\bar{x}(s),v^{k})
-\sigma(s,\bar{x}(s),\bar{u}(s))\big)\chi_{A_{ij}}ds\Big\rangle \chi_{A_{ij}}dt\nonumber\\
& &+\frac{1}{\varepsilon^2}\me\int_{\t}^{\t+\varepsilon}
\Big\langle\mss(t,\bar{x}(t),v^{k}),
\Phi(t)\int_{\t}^{t}\Phi(s)^{-1}\nonumber\\
& &\qquad\qquad\qquad \qquad\
\cdot\big(\sigma(s,\bar{x}(s),v^{k})
-\sigma(s,\bar{x}(s),\bar{u}(s))\big)\chi_{A_{ij}}dW(s)\Big\rangle\chi_{A_{ij}} dt.
\end{eqnarray}
By Lemma \ref{technical lemma}, we obtain that for a.e. $\t\in [t_{i},T]$
\begin{eqnarray}\label{th pointwise mallivain eq2}
& &\lim_{\varepsilon\to 0^+}\Big[-\frac{1}{\varepsilon^2}\me\int_{\t}^{\t+\varepsilon}
\Big\langle\mss(t,\bar{x}(t),v),
 \Phi(t)\int_{\t}^{t}\Phi(s)^{-1}
\sigma_{x}(s)\nonumber\\
& &\qquad\qquad\quad\qquad\qquad\quad
\cdot\big(\sigma(s,\bar{x}(s),v^{k})
-\sigma(s,\bar{x}(s),\bar{u}(s))\big)\chi_{A_{ij}}ds
\Big\rangle\chi_{A_{ij}}dt\Big]\nonumber\\
&=&-\frac{1}{2}\me~\Big[\Big\langle\mss(\t,\bar{x}(\t),v^{k}),
\sigma_{x}(\t)\big(\sigma(\t,\bar{x}(\t),v^{k})
-\sigma(\t,\bar{x}(\t),\bar{u}(\t))\big)
\Big\rangle\chi_{A_{ij}}\Big].\nonumber\\
\end{eqnarray}
On the other hand, by (\ref{Phi}), we deduce that
\begin{eqnarray}\label{th pointwise mallivain eq3}
& &\lim_{\varepsilon\to 0^+}\frac{1}{\varepsilon^2}\me\int_{\t}^{\t+\varepsilon}
\Big\langle\mss(t,\bar{x}(t),v^{k}),
\Phi(t)\int_{\t}^{t}\Phi(s)^{-1}\nonumber\\
& &\qquad\qquad\qquad\quad\qquad\quad
\cdot\big(\sigma(s,\bar{x}(s),v^{k})-\sigma(s,\bar{x}(s),\bar{u}(s))\big)\chi_{A_{ij}}
dW(s)\Big\rangle\chi_{A_{ij}}  dt\nonumber\\
&=&\lim_{\varepsilon\to 0^+}\frac{1}{\varepsilon^2}\me\int_{\t}^{\t+\varepsilon}
\Big\langle\mss(t,\bar{x}(t),v^{k}),
\Phi(\t)\int_{\t}^{t}\Phi(s)^{-1}\nonumber\\
& &\qquad\qquad\qquad\quad\qquad\quad
\cdot
\big(\sigma(s,\bar{x}(s),v^{k})-\sigma(s,\bar{x}(s),\bar{u}(s))\big)\chi_{A_{ij}}
dW(s)\Big\rangle\chi_{A_{ij}} dt\nonumber\\
& &+\lim_{\varepsilon\to 0^+}\frac{1}{\varepsilon^2}\me\int_{\t}^{\t+\varepsilon}
\Big\langle\mss(t,\bar{x}(t),v^{k}),\int_{\t}^{t}b_{x}(s)\Phi(s)ds
\int_{\t}^{t}\Phi(s)^{-1}\nonumber\\
& &\qquad\qquad\qquad\quad\qquad\quad
\cdot
\big(\sigma(s,\bar{x}(s),v^{k})-\sigma(s,\bar{x}(s),\bar{u}(s))\big)\chi_{A_{ij}}dW(s)
\Big\rangle\chi_{A_{ij}} dt\nonumber\\
& & +\lim_{\varepsilon\to 0^+}\frac{1}{\varepsilon^2}\me\int_{\t}^{\t+\varepsilon}
\Big\langle\mss(t,\bar{x}(t),v^{k}),\int_{\t}^{t}\sigma_{x}(s)\Phi(s)dW(s)
\int_{\t}^{t}\Phi(s)^{-1}\nonumber\\
& &\qquad\qquad\qquad\quad\qquad\quad
\cdot
\big(\sigma(s,\bar{x}(s),v^{k})-\sigma(s,\bar{x}(s),\bar{u}(s))\big)\chi_{A_{ij}}dW(s)
\Big\rangle \chi_{A_{ij}}dt.\nonumber\\
\end{eqnarray}

Similar to the proof of (\ref{th41 eq3add})--(\ref{th41 eq4}), we obtain that, for a.e. $\t\in[t_{i},T]$,
\begin{eqnarray}\label{th pointwise mallivain eq4}
& &\lim_{\varepsilon\to 0^+}\frac{1}{\varepsilon^2}\me\int_{\t}^{\t+\varepsilon}
\Big\langle\mss(t,\bar{x}(t),v^{k}),\int_{\t}^{t}b_{x}(s)\Phi(s)ds
\int_{\t}^{t}\Phi(s)^{-1}\\
& &\qquad\qquad\qquad\quad\qquad\quad
\cdot\big(\sigma(s,\bar{x}(s),v^{k})-\sigma(s,\bar{x}(s),\bar{u}(s))\big)\chi_{A_{ij}}dW(s)
\Big\rangle\chi_{A_{ij}} dt\nonumber\\
&=&0,\nonumber
\end{eqnarray}
and
\begin{eqnarray}\label{th pointwise mallivain eq5}
& &\lim_{\varepsilon\to 0^+}\frac{1}{\varepsilon^2}\me\int_{\t}^{\t+\varepsilon}
\Big\langle\mss(t,\bar{x}(t),v^{k}),\int_{\t}^{t}\sigma_{x}(s)\Phi(s)dW(s)
\int_{\t}^{t}\Phi(s)^{-1}\nonumber\\
& &\qquad\qquad\qquad\quad\qquad\
\cdot\big(\sigma(s,\bar{x}(s),v^{k})-\sigma(s,\bar{x}(s),\bar{u}(s))\big)\chi_{A_{ij}}dW(s)
\Big\rangle\chi_{A_{ij}} dt\nonumber\\
&=&\frac{1}{2}\me~\Big[\Big\langle\mss(\t,\bar{x}(\t),v^{k}),
\sigma_{x}(\t)\big(\sigma(\t,\bar{x}(\t),v^{k})-\sigma(\t,\bar{x}(\t),\bar{u}(\t))\big)
\Big\rangle\chi_{A_{ij}}\Big].\nonumber\\
\end{eqnarray}
Then, by (\ref{th pointwise mallivain eq1})--(\ref{th pointwise mallivain eq5}), in order to prove (\ref{step1in th mallivain}), it remains to show that there exists a sequence $\{\varepsilon_{\ell}\}_{\ell=1}^{\infty} $, $\varepsilon_{\ell}\to 0^+$ as $\ell\to \infty$ such that
\begin{eqnarray}\label{th pointwise mallivain eq6}
& &\lim_{\ell\to \infty}\frac{1}{\varepsilon_{\ell}^2}\me\int_{\t}^{\t+\varepsilon_{\ell}}
\Big\langle\mss(t,\bar{x}(t),v^{k}),
\Phi(\t)\int_{\t}^{t}\Phi(s)^{-1}\nonumber\\
& &\qquad\qquad\qquad\quad\qquad\quad
\cdot\big(\sigma(s,\bar{x}(s),v^{k})-\sigma(s,\bar{x}(s),\bar{u}(s))\big)\chi_{A_{ij}}
dW(s)\Big\rangle\chi_{A_{ij}} dt\nonumber\\
&=&\frac{1}{2}\me~\Big[\Big\langle\nabla\mss(\t,\bar{x}(\t),v^{k}),
\sigma(\t,\bar{x}(\t),v^{k})-\sigma(\t,\bar{x}(\t),\bar{u}(\t))\Big\rangle\chi_{A_{ij}}\Big],
\quad \ a.e.\ \t\in [t_{i},T].\nonumber\\
\end{eqnarray}
By the regularity  assumption (C3) and the Clark-Ocone representation formula, we have that

\begin{eqnarray}\label{th pointwise mallivain eq7}
& &\frac{1}{\varepsilon^2}\me\int_{\t}^{\t+\varepsilon}
\Big\langle\mss(t,\bar{x}(t),v^{k}),
\Phi(\t)\int_{\t}^{t}\Phi(s)^{-1}\nonumber\\
& &\qquad\qquad\qquad\quad\qquad\quad
\cdot\big(\sigma(s,\bar{x}(s),v^{k})-\sigma(s,\bar{x}(s),\bar{u}(s))\big)\chi_{A_{ij}}
dW(s)\Big\rangle\chi_{A_{ij}} dt\nonumber\\
&=&\frac{1}{\varepsilon^2}\me\int_{\t}^{\t+\varepsilon}
\Big\langle\me~\mss(t,\bar{x}(t),v^{k}),
\Phi(\t)\int_{\t}^{t}\Phi(s)^{-1}\nonumber\\
& &\qquad\qquad\qquad\quad\qquad\quad
\cdot
\big(\sigma(s,\bar{x}(s),v^{k})-\sigma(s,\bar{x}(s),\bar{u}(s))\big)\chi_{A_{ij}}
dW(s)\Big\rangle\chi_{A_{ij}} dt\nonumber\\
& &+\frac{1}{\varepsilon^2}\me\int_{\t}^{\t+\varepsilon}
\Big\langle\int_{0}^{t}\me\Big[\dd_{s}\mss(t,\bar{x}(t),v^{k})\ \Big|
\ \mf_{s}\Big]dW(s),
\Phi(\t)\int_{\t}^{t}\Phi(s)^{-1}\nonumber\\
& &\qquad\qquad\qquad\quad\qquad\quad
\cdot
\big(\sigma(s,\bar{x}(s),v^{k})-\sigma(s,\bar{x}(s),\bar{u}(s))\big)\chi_{A_{ij}}
dW(s)\Big\rangle\chi_{A_{ij}} dt\nonumber\\
&=&\frac{1}{\varepsilon^2}\me\int_{\t}^{\t+\varepsilon}\int_{\t}^{t}
\Big\langle \dd_{s}\mss(t,\bar{x}(t),v^{k}),\nonumber\\
& &\qquad\qquad\qquad\quad
\Phi(\t)\Phi(s)^{-1}
\big(\sigma(s,\bar{x}(s),v^{k})-\sigma(s,\bar{x}(s),\bar{u}(s))\big)\chi_{A_{ij}}
\Big\rangle \chi_{A_{ij}}ds dt\nonumber\\
&=&\frac{1}{\varepsilon^2}\me\int_{\t}^{\t+\varepsilon}\int_{\t}^{t}
\Big\langle \dd_{s}\mss(t,\bar{x}(t),v^{k})-\nabla\mss(s,\bar{x}(s),v^{k}),\nonumber\\
& &\qquad\qquad\qquad\quad
\Phi(\t)\Phi(s)^{-1}
\big(\sigma(s,\bar{x}(s),v^{k})-\sigma(s,\bar{x}(s),\bar{u}(s))\big)\chi_{A_{ij}}
\Big\rangle\chi_{A_{ij}} ds dt\nonumber\\
& &+\frac{1}{\varepsilon^2}\me\int_{\t}^{\t+\varepsilon}\int_{\t}^{t}
\Big\langle\nabla\mss(s,\bar{x}(s),v^{k}),\nonumber\\
& &\qquad\qquad\qquad\quad
\Phi(\t)\Phi(s)^{-1}
\big(\sigma(s,\bar{x}(s),v^{k})-\sigma(s,\bar{x}(s),\bar{u}(s))\big)\chi_{A_{ij}}
\Big\rangle\chi_{A_{ij}} ds dt.
\end{eqnarray}
By the assumptions (C1)--(C3) and \cite[Lemma 2.1]{zhangH14a},  there exists a sequence $\{\varepsilon_{\ell}\}_{\ell=1}^{\infty} $, $\varepsilon_{\ell}\to 0^+$ as $\ell\to \infty$ such that
\begin{eqnarray}\label{th pointwise mallivain eq8}
& &\lim_{\ell\to \infty}\frac{1}{\varepsilon_{\ell}^2}\Big|\me\int_{\t}^{\t+\varepsilon_{\ell}}\int_{\t}^{t}
\Big\langle \dd_{s}\mss(t,\bar{x}(t),v^{k})-\nabla\mss(s,\bar{x}(s),v^{k}),\nonumber\\
& &\qquad\qquad\qquad
\Phi(\t)\Phi(s)^{-1}
\big(\sigma(s,\bar{x}(s),v^{k})-\sigma(s,\bar{x}(s),\bar{u}(s))\big)\chi_{A_{ij}}
\Big\rangle \chi_{A_{ij}}ds dt\Big|\nonumber\\
&\le&\lim_{\ell\to \infty}\frac{C}{\varepsilon_{\ell}}
\Big[\me\Big(\sup_{s\in[\t,T]}
\big|\Phi(\t)\Phi(s)^{-1}\big|^{2}\Big)\Big]^{\frac{1}{2}}\nonumber\\
& &\qquad\qquad\qquad\qquad\
\cdot
\Big[\me\int_{\t}^{\t+\varepsilon_{\ell}}\int_{\t}^{t}
\Big| \dd_{s}\mss(t,\bar{x}(t),v^{k})-\nabla\mss(s,\bar{x}(s),v^{k})\Big|^{2} ds dt\Big]^{\frac{1}{2}}\nonumber\\
&=&0, \quad  a.e.\  \t\in[t_{i},T].\nonumber\\
\end{eqnarray}
On the other hand, by Lemma \ref{technical lemma},
\begin{eqnarray}\label{th pointwise mallivain eq9}
& &\lim_{\varepsilon\to 0^+}\frac{1}{\varepsilon^2}\me\int_{\t}^{\t+\varepsilon}\int_{\t}^{t}
\Big\langle\nabla\mss(s,\bar{x}(s),v^{k}),\nonumber\\
& &\qquad\qquad\qquad\qquad
\Phi(\t)\Phi(s)^{-1}
\big(\sigma(s,\bar{x}(s),v^{k})-\sigma(s,\bar{x}(s),\bar{u}(s))\big)\chi_{A_{ij}}
\Big\rangle \chi_{A_{ij}}ds dt\nonumber\\
&=&\frac{1}{2}\me~\Big[\Big\langle\nabla\mss(\t,\bar{x}(\t),v^{k}),
\sigma(\t,\bar{x}(\t),v^{k})-\sigma(\t,\bar{x}(\t),\bar{u}(\t))\Big\rangle\chi_{A_{ij}}\Big],
\quad \ a.e.\ \t\in [t_{i},T].\nonumber\\
\end{eqnarray}
Combining (\ref{th pointwise mallivain eq7}), (\ref{th pointwise mallivain eq8}) with (\ref{th pointwise mallivain eq9}), we obtain (\ref{th pointwise mallivain eq6}). By (\ref{th pointwise mallivain eq1})--(\ref{th pointwise mallivain eq6}), we obtain (\ref{step1in th mallivain}).

Next, similar to Steps 2 and 3 in the proof of Theorem \ref{2orderconditionth}, we obtain that

\begin{eqnarray}\label{step2in th mallivain}
& &\lim_{\varepsilon\to 0^+}\frac{1}{\varepsilon^2}
\me\int_{0}^{T}\Big\langle\mss(t,\bar{x}(t),v^{k}),
\hat{y}_{2}^{\varepsilon}(t)\Big\rangle\chi_{E_{\varepsilon}}(t)dt\\
&=& \frac{1}{2}\me ~\Big[\Big\langle\mss(\t,\bar{x}(\t),v^{k}), b(\t,\bar{x}(\t),v^{k})-b(\t,\bar{x}(\t),\bar{u}(\t))\Big\rangle\chi_{A_{ij}}\Big],  \ \ \ a.e.\ \ \t\in [t_{i},T].\nonumber
\end{eqnarray}
and
\begin{eqnarray}\label{step3in th mallivain}
& &\lim_{\varepsilon\to 0^+}\frac{1}{\varepsilon^2}
\me\int_{0}^{T}\Big\langle\mt(t,\bar{x}(t),v)\hat{y}_{1}^{\varepsilon}(t),
\hat{y}_{1}^{\varepsilon}(t)\Big\rangle
\chi_{E_{\varepsilon}}(t)dt\\
&=&\frac{1}{2}
\me ~\Big[\Big\langle\mt(\t,\bar{x}(\t),v^{k})
\big(\sigma(\t,\bar{x}(\t),v^{k})-\sigma(\t,\bar{x}(\t),\bar{u}(\t))\big),\nonumber\\
& &\qquad\qquad\qquad\qquad\quad
\sigma(\t,\bar{x}(\t),v^{k})-\sigma(\t,\bar{x}(\t),\bar{u}(\t))
\Big\rangle\chi_{A_{ij}}\Big],
\  \ a.e. \ \t\in [t_{i},T].\nonumber
\end{eqnarray}

Finally, combining (\ref{step0in th mallivain}), (\ref{step1in th mallivain}),  (\ref{step2in th mallivain}) and (\ref{step3in th mallivain}), we end up with
\begin{eqnarray}\label{step4in th mallivain}
& &\me ~\Big[\inner{\mss(\t,\bar{x}(\t), v^{k})}{
b(\t,\bar{x}(\t),v^{k})-b(\t,\bar{x}(\t),\bar{u}(\t))}\chi_{A_{ij}}\Big]\nonumber\\
& &\quad+ \me ~\Big[\inner{\nabla\mss(\t,\bar{x}(\t),v^{k})}{
\sigma(\t,\bar{x}(\t),v^{k})-\sigma(\t,\bar{x}(\t),\bar{u}(\t))}\chi_{A_{ij}}\Big]\nonumber\\
& &
\quad+\frac{1}{2}\me ~\Big[\big\langle\mt(\t,\bar{x}(\t),v^{k})
\big(\sigma(\t,\bar{x}(\t),v^{k})-\sigma(\t,\bar{x}(\t),\bar{u}(\t))\big),\nonumber\\
& &\quad\qquad\qquad\qquad\qquad\qquad
\sigma(\t,\bar{x}(\t),v^{k})-\sigma(\t,\bar{x}(\t),\bar{u}(\t))
\big\rangle\chi_{A_{ij}}\Big]\le 0.
\end{eqnarray}
By the arbitrariness of $i,\ j,\ k$, the  construction of  $\{A_{ij}\}_{i=1}^{\infty}$, the continuities of the filter $\mmf$ and the map $v\mapsto \nabla\mss(\t,\bar{x}(\t),v)$,  and the density of $\{v^{k}\}_{k=1}^{\infty}$, we conclude that the desired necessary condition (\ref{2ordercondition nonconvex corollary}) holds. This completes the proof of Theorem \ref{Th 2ordercondition nonconvex malliavin}.

\appendix

\section{Proof of Lemma \ref{estimateofvariequ}}
To simplify the notation, we only prove the $1$-dimensional case (The high dimensional case can be proved in the same way).

The proof is long and requires heavy computations (The main idea comes from the proof of \cite[Theorem 4.4, p. 128]{Yong99}). We will divide it into 4 steps

\textbf{Step 1:} Estimation of $\|y_{1}^{\varepsilon}\|_{\infty,\beta}^{\beta}$,
$\|y_{2}^{\varepsilon}\|_{\infty,\beta}^{\beta}$, $\|y_{3}^{\varepsilon}\|_{\infty,\beta}^{\beta}$ and ~$\|y_{4}^{\varepsilon}\|_{\infty,\beta}^{\beta}$.

By the conditions (C1)--(C2) and the estimate (\ref{estimateofx}), we have
\begin{equation}\label{estlemma y1}
\me~\Big[\sup_{t\in[0,T]}|y_{1}^{\varepsilon}(t)|^{\beta}\Big]
\le C\me~\Big[\int_{0}^{T}\big|\delta \sigma(t)\chi_{E_{\varepsilon}}(t)\big|^2dt\Big]^{\frac{\beta}{2}}
\le C \varepsilon^{\frac{\beta}{2}}.
\end{equation}

In a similar way, we have
\begin{eqnarray}\label{estlemma y2}
\me~\Big[\sup_{t\in[0,T]}|y_{2}^{\varepsilon}(t)|^{\beta}\Big]
&\le &C\me~\Big[\int_{0}^{T}\big|\frac{1}{2}b_{xx}(t)y_{1}^{\varepsilon}(t)^{2}
+\delta b(t)\chi_{E_{\varepsilon}}(t)\big|dt\Big]^{\beta}\nonumber\\
& &+C\me~\Big[\int_{0}^{T}\big|\frac{1}{2}\sigma_{xx}(t)y_{1}^{\varepsilon}(t)^{2}
+\delta\sigma_{x}(t)y_{1}^{\varepsilon}(t)\chi_{E_{\varepsilon}}(t)
\big|^{2}dt\Big]^{\frac{\beta}{2}}\nonumber\\
&\le& C\Big\{\me~\Big[\sup_{t\in[0,T]}|y_{1}^{\varepsilon}(t)|^{2\beta}\Big]
+\varepsilon^{\beta}\nonumber\\
& &\ \quad+\me~\Big[\sup_{t\in[0,T]}|y_{1}^{\varepsilon}(t)|^{2\beta}\Big]
+\me~\Big[\sup_{t\in[0,T]}|y_{1}^{\varepsilon}(t)|^{\beta}\Big]\varepsilon^{\frac{\beta}{2}}
\Big\}
\le C \varepsilon^{\beta},
\end{eqnarray}
\begin{eqnarray}\label{estlemma y3}
& &\me~\Big[\sup_{t\in[0,T]}|y_{3}^{\varepsilon}(t)|^{\beta}\Big]\nonumber\\
&\le &C\me~\Big[\int_{0}^{T}\big|\frac{1}{2}b_{xx}(t)\big(2 y_{1}^{\varepsilon}(t)
 y_{2}^{\varepsilon}(t)+ y_{2}^{\varepsilon}(t)^{2}\big)+\frac{1}{6}b_{xxx}(t) y_{1}^{\varepsilon}(t)^{3}\nonumber\\
& &\ \ +\delta b_{x}(t) y_{1}^{\varepsilon}(t) \chi_{E_{\varepsilon}}(t)\big|dt\Big]^{\beta}
+C\me~\Big[\int_{0}^{T}\big|\frac{1}{2}\sigma_{xx}(t)\big(2 y_{1}^{\varepsilon}(t)
 y_{2}^{\varepsilon}(t)+ y_{2}^{\varepsilon}(t)^{2}\big)\nonumber\\
& &\ \ +\frac{1}{6}\sigma_{xxx}(t) y_{1}^{\varepsilon}(t)^{3}
+\delta \sigma_{x}(t) y_{2}^{\varepsilon}(t) \chi_{E_{\varepsilon}}(t) +\frac{1}{2}\delta\sigma_{xx}(t)y_{1}^{\varepsilon}(t)^{2}
\chi_{E_{\varepsilon}}(t)\big|^{2}dt\Big]^{\frac{\beta}{2}}\nonumber\\
&\le& C\Big\{\me~\Big[\sup_{t\in[0,T]}|y_{1}^{\varepsilon}(t)|^{\beta}
 |y_{2}^{\varepsilon}(t)|^{\beta}\Big]
+\me~\Big[\sup_{t\in[0,T]}|y_{2}^{\varepsilon}(t)|^{2\beta}\Big]\nonumber\\
& &\qquad+\me~\Big[\sup_{t\in[0,T]}|y_{1}^{\varepsilon}(t)|^{3\beta}\Big]
+\me~\Big[\sup_{t\in[0,T]}
|y_{1}^{\varepsilon}(t)|^{\beta}\Big]\varepsilon^{\beta}\nonumber\\
& &\qquad+\me~\Big[\sup_{t\in[0,T]}|y_{2}^{\varepsilon}(t)|^{\beta}\Big]
\varepsilon^{\frac{\beta}{2}}+
\me~\Big[\sup_{t\in[0,T]}|y_{1}^{\varepsilon}(t)|^{2\beta}\Big]\varepsilon^{\frac{\beta}{2}}
\Big\}
\le C \varepsilon^{\frac{3\beta}{2}},
\end{eqnarray}
and
\begin{eqnarray}\label{estlemma y4}
& &\me~\Big[\sup_{t\in[0,T]}|y_{4}^{\varepsilon}(t)|^{\beta}\Big]\nonumber\\
&\le &C\me~\Big[\int_{0}^{T}\big|\frac{1}{2}b_{xx}(t)\big(2 y_{1}^{\varepsilon}(t)
y_{3}^{\varepsilon}(t)+ 2 y_{2}^{\varepsilon}(t) y_{3}^{\varepsilon}(t)+ y_{3}^{\varepsilon}(t)^{2}\big)\nonumber\\
& &+\frac{1}{6}b_{xxx}(t) \big(3y_{1}^{\varepsilon}(t)^{2}
y_{2}^{\varepsilon}(t)+3y_{1}^{\varepsilon}(t) y_{2}^{\varepsilon}(t)^{2}
+y_{2}^{\varepsilon}(t)^{3}\big)\nonumber\\
& &+\frac{1}{24}b_{xxxx}(t) y_{1}^{\varepsilon}(t)^{4}
+\delta b_{x}(t) y_{2}^{\varepsilon}(t) \chi_{E_{\varepsilon}}(t)+\frac{1}{2}\delta b_{xx}(t) y_{1}^{\varepsilon}(t)^{2}\chi_{E_{\varepsilon}}(t)\big|dt\Big]^{\beta}\nonumber\\
& &+C\me~\Big[\int_{0}^{T}\big|\frac{1}{2}\sigma_{xx}(t)
\big(2 y_{1}^{\varepsilon}(t) y_{3}^{\varepsilon}(t)+ 2 y_{2}^{\varepsilon}(t) y_{3}^{\varepsilon}(t)+ y_{3}^{\varepsilon}(t)^{2}\big)\nonumber\\
& &+\frac{1}{6}\sigma_{xxx}(t) \big(3y_{1}^{\varepsilon}(t)^{2}
y_{2}^{\varepsilon}(t)+3y_{1}^{\varepsilon}(t) y_{2}^{\varepsilon}(t)^{2}+y_{2}^{\varepsilon}(t)^{3}\big)\nonumber\\
& &+\frac{1}{24}\sigma_{xxxx}(t) y_{1}^{\varepsilon}(t)^{4}
+\delta \sigma_{x}(t) y_{3}^{\varepsilon}(t) \chi_{E_{\varepsilon}}(t)\nonumber\\
& & +\frac{1}{2}\delta\sigma_{xx}(t)\big(2 y_{1}^{\varepsilon}(t) y_{2}^{\varepsilon}(t)
+ y_{2}^{\varepsilon}(t)^{2}\big)\chi_{E_{\varepsilon}}(t)
+\frac{1}{6}\delta\sigma_{xxx}(t)
y_{1}^{\varepsilon}(t)^{3}\chi_{E_{\varepsilon}}(t)\big|^{2}dt\Big]^{\frac{\beta}{2}}
\nonumber\\[+0.6em]
&\le& C\Big\{\me~\Big[\sup_{t\in[0,T]}|y_{1}^{\varepsilon}(t)|^{\beta}
|y_{3}^{\varepsilon}(t)|^{\beta}\Big]
+\me~\Big[\sup_{t\in[0,T]}|y_{2}^{\varepsilon}(t)|^{\beta} |y_{3}^{\varepsilon}(t)|^{\beta}\Big]\nonumber\\[+0.3em]
& &+ \me~\Big[\sup_{t\in[0,T]}|y_{3}^{\varepsilon}(t)|^{2\beta}\Big]
+\me~\Big[\sup_{t\in[0,T]}|y_{1}^{\varepsilon}(t)|^{2\beta}
|y_{2}^{\varepsilon}(t)|^{\beta}\Big]\nonumber\\[+0.3em]
& &+\me~\Big[\sup_{t\in[0,T]}|y_{1}^{\varepsilon}(t)|^{\beta}
|y_{2}^{\varepsilon}(t)|^{2\beta}\Big]
+\me~\Big[\sup_{t\in[0,T]}|y_{2}^{\varepsilon}(t)|^{3\beta}\Big]\nonumber\\[+0.3em]
& &+\me~\Big[\sup_{t\in[0,T]}|y_{1}^{\varepsilon}(t)|^{4\beta}\Big]
+\me~\Big[\sup_{t\in[0,T]}
|y_{2}^{\varepsilon}(t)|^{\beta}\Big]\varepsilon^{\beta}\nonumber\\[+0.3em]
& &+\me~\Big[\sup_{t\in[0,T]}|y_{1}^{\varepsilon}(t)|^{2\beta}\Big]
\varepsilon^{\beta}
+\me~\Big[\sup_{t\in[0,T]}|y_{3}^{\varepsilon}(t)|^{\beta}
\Big]\varepsilon^{\frac{\beta}{2}}\nonumber\\[+0.3em]
& &+\me~\Big[\sup_{t\in[0,T]}|y_{1}^{\varepsilon}(t)|^{\beta}
|y_{2}^{\varepsilon}(t)|^{\beta}\Big]\varepsilon^{\frac{\beta}{2}}
+\me~\Big[\sup_{t\in[0,T]}|y_{2}^{\varepsilon}(t)|^{2\beta}
\Big]\varepsilon^{\frac{\beta}{2}}\nonumber\\[+0.3em]
& &+\me~\Big[\sup_{t\in[0,T]}|y_{1}^{\varepsilon}(t)|^{3\beta}
\Big]\varepsilon^{\frac{\beta}{2}}\Big\}
\le C \varepsilon^{2\beta}
\end{eqnarray}

\textbf{Step 2:} Estimation of $\|r_{1}\|_{\infty,\beta}^{\beta}$, $\|r_{2}\|_{\infty,\beta}^{\beta}$ and $\|r_{3}\|_{\infty,\beta}^{\beta}$.

By (\ref{continuousofxwithx0andu}) and the condition (C1)--(C2), we have
\begin{eqnarray}\label{estimateofdeltax}
\me~\Big[\sup_{t\in[0,T]}|\delta x(t)|^{\beta}\Big]
&\le& C\me~\Big[\int_{0}^{T}\big|\delta b(t)\chi_{E_{\varepsilon}}(t)\big|dt\Big]^{\beta}
+C\me~\Big[\int_{0}^{T}\big|\delta\sigma(t)\chi_{E_{\varepsilon}}(t)\big|^{2}dt\Big]
^{\frac{\beta}{2}}\nonumber\\
&\le & C\varepsilon^{\beta}+C\varepsilon^{\frac{\beta}{2}}
\le C\varepsilon^{\frac{\beta}{2}}.
\end{eqnarray}

Define
\begin{equation}\label{coffofdeltax}
\left\{
\begin{array}{l}
\tilde{b}_{x}^{\varepsilon}(t):=\int_{0}^{1}b_{x}(t,\theta\bar{x}(t)+
(1-\theta)x^{\varepsilon}(t),u^{\varepsilon}(t))d\theta,\\[+0.6em]
\tilde{\sigma}_{x}^{\varepsilon}(t):=\int_{0}^{1}\sigma_{x}(t,\theta\bar{x}(t)+
(1-\theta)x^{\varepsilon}(t),u^{\varepsilon}(t))d\theta.
\end{array}\right.
\end{equation}
Then, $\delta x(\cdot)\!=\!x^{\varepsilon}(\cdot)-\bar{x}(\cdot)$ is the solution to the following stochastic differential equation:
\begin{equation}\nonumber
\left\{
\begin{array}{l}
d\delta x(t)= \Big[\tilde{b}_{x}^{\varepsilon}(t)\delta x(t)
+\delta b(t)\chi_{E_{\varepsilon}}(t)\Big]dt\\[+0.7em]
\qquad\qquad+\Big[\tilde{\sigma}_{x}^{\varepsilon}(t)\delta x(t)
+\delta \sigma(t)\chi_{E_{\varepsilon}}(t)\Big]dW(t),\ \  t\in[0,T],\\[+0.7em]
\delta x(0)=0.
\end{array}\right.
\end{equation}
Also, $r_{1}(\cdot)=\delta x(\cdot)-y_{1}^{\varepsilon}(\cdot)$ is the solution to the following stochastic differential equation:
\begin{equation}\nonumber
\qquad\;\;\left\{
\begin{array}{l}
dr_{1}(t)= \Big[\tilde{b}_{x}^{\varepsilon}(t)r_{1}(t)
+\delta b(t)\chi_{E_{\varepsilon}}(t)
+ \big(\tilde{b}_{x}^{\varepsilon}(t)-b_{x}(t)\big)y_{1}^{\varepsilon}(t)\Big]dt\\[+0.7em]
\qquad \qquad\qquad\qquad +\Big[\tilde{\sigma}_{x}^{\varepsilon}(t)r_{1}(t)
+ \big(\tilde{\sigma}_{x}^{\varepsilon}(t)
-\sigma_{x}(t)\big)y_{1}^{\varepsilon}(t)\Big]dW(t),\ \ t\in[0,T].\\[+0.7em]
r_{1}(0)=0.
\end{array}\right.
\end{equation}

Since
\begin{eqnarray*}
& &\me~\Big[\int_{0}^{T}|\tilde{b}_{x}^{\varepsilon}(t)-b_{x}(t)|dt\Big]^{\beta}\\
&=&\me~\Big[\int_{0}^{T}\Big|\int_{0}^{1}b_{x}(t,\theta\bar{x}(t)+
(1-\theta)x^{\varepsilon}(t),u^{\varepsilon}(t))
-b_{x}(t)d\theta\Big|dt\Big]^{\beta}\\
&\le&\me\Big[\int_{0}^{T}\Big(L\big|x^{\varepsilon}(t)-\bar{x}(t)\big|+\big|\delta b_{x}(t)\chi_{E_{\varepsilon}}(t)\big|\Big)dt\Big]^{\beta}\\
&\le&C\me~\Big[\sup_{t\in[0,T]}|\delta x(t)|^{\beta}\Big]+C\varepsilon^{\beta}
\le C\varepsilon^{\frac{\beta}{2}},
\end{eqnarray*}
and
\begin{eqnarray*}
& &\me~\Big[\int_{0}^{T}\big|\tilde{\sigma}_{x}^{\varepsilon}(t)
-\sigma_{x}(t)\big|^{2}dt\Big]^{\frac{\beta}{2}}\\
&=&\me~\Big[\int_{0}^{T}\Big|\int_{0}^{1}\sigma_{x}(t,\theta\bar{x}(t)+
(1-\theta)x^{\varepsilon}(t),u^{\varepsilon}(t))
-\sigma_{x}(t)d\theta\Big|^{2}dt\Big]^{\frac{\beta}{2}}\\
&\le&C\me~\Big[\int_{0}^{T}\Big(L|x^{\varepsilon}(t)-\bar{x}(t)|^{2}+|\delta \sigma_{x}(t)\chi_{E_{\varepsilon}}(t)|^{2}\Big)dt\Big]^{\frac{\beta}{2}}\\
&\le&C\me~\Big[\sup_{t\in[0,T]}|\delta x(t)|^{\beta}\Big]+C\varepsilon^{\frac{\beta}{2}}
\le C\varepsilon^{\frac{\beta}{2}},
\end{eqnarray*}
we have
\begin{eqnarray}\label{estlemma r1}
\me~\Big[\sup_{t\in[0,T]}|r_{1}(t)|^{\beta}\Big]
&\le &C\me~\Big[\int_{0}^{T}\big|\delta b(t)\chi_{E_{\varepsilon}}(t)+ \big(\tilde{b}_{x}^{\varepsilon}(t)
-b_{x}(t)\big)y_{1}^{\varepsilon}(t)\big|dt\Big]^{\beta}\nonumber\\
& &
+C~\me~\Big[\int_{0}^{T}\big|\big(\tilde{\sigma}_{x}^{\varepsilon}(t)
-\sigma_{x}(t)\big)y_{1}^{\varepsilon}(t)\big|^{2}dt\Big]^{\frac{\beta}{2}}\nonumber\\
&\le&C\Big\{\varepsilon^{\beta}
+\me\Big[\Big(\sup_{t\in[0,T]}|y_{1}^{\varepsilon}(t)|^{\beta}\Big)
\Big(\int_{0}^{T}\big|\tilde{\sigma}_{x}^{\varepsilon}(t)
-\sigma_{x}(t)\big|dt\Big)^{\beta}\Big]\nonumber\\
& &+\me \Big[\Big(\sup_{t\in[0,T]}|y_{1}^{\varepsilon}(t)|^{\beta}\Big)
\Big(\int_{0}^{T}\big|\tilde{b}_{x}^{\varepsilon}(t)
-b_{x}(t)\big|^{2}dt\Big)^{\frac{\beta}{2}}\Big]\Big\}
\le C \varepsilon^{\beta}.
\end{eqnarray}
This gives the estimation for $r_1(\cdot)$.

Next, we prove the estimation for $r_2(\cdot)$.
For $\varphi=b,\sigma$, by Taylor's formula, we have
\begin{eqnarray}
& & \varphi(t,x^{\varepsilon}(t),u^{\varepsilon}(t))
-\varphi(t,\bar{x}(t),\bar{u}(t))\nonumber\\[+0.3em]
&=& \varphi(t,x^{\varepsilon}(t),\bar{u}(t))-\varphi(t,\bar{x}(t),\bar{u}(t))
+\varphi(t,\bar{x}(t),u^{\varepsilon}(t))\nonumber\\[+0.3em]
& &-\varphi(t,\bar{x}(t),\bar{u}(t))
+\varphi(t,x^{\varepsilon}(t),u^{\varepsilon}(t))
-\varphi(t,\bar{x}(t),u^{\varepsilon}(t))\nonumber\\[+0.3em]
& &-\varphi(t,x^{\varepsilon}(t),\bar{u}(t))
+\varphi(t,\bar{x}(t),\bar{u}(t))\nonumber\\[+0.3em]
&=&\varphi_x(t)\delta x(t)+\frac{1}{2}\varphi_{xx}(t)\delta x(t)^2
+\delta \varphi(t)\chi_{E_{\varepsilon}}(t)\nonumber\\[-0.2em]
& & +\frac{1}{2}\int_{0}^{1}\theta^2 \varphi_{xxx}(t,\theta\bar{x}(t)+
(1-\theta)x^{\varepsilon}(t),\bar{u}(t))\delta x(t)^3d\theta\nonumber\\[-0.2em]
& &+\int_{0}^{1} \Big( \varphi_{x}(t,\theta\bar{x}(t)+
(1-\theta)x^{\varepsilon}(t),u^{\varepsilon}(t))\nonumber\\[-0.2em]
& &\qquad\qquad\qquad\qquad
-\varphi_{x}(t,\theta\bar{x}(t)
+(1-\theta)x^{\varepsilon}(t),\bar{u}(t))\Big)\delta x(t) d\theta
\label{expdeltax 1}\\
&=& \varphi_x(t)\delta x(t)+\frac{1}{2}\varphi_{xx}(t)\delta x(t)^2
+\frac{1}{6}\varphi_{xxx}(t)\delta x(t)^3\nonumber\\[+0.3em]
& &+\delta \varphi(t)\chi_{E_{\varepsilon}}(t)
+\delta \varphi_x(t)\delta x(t)\chi_{E_{\varepsilon}}(t)\nonumber\\[+0.3em]
& & +\frac{1}{6}\int_{0}^{1}\theta^3 \varphi_{xxxx}(t,\theta\bar{x}(t)+
(1-\theta)x^{\varepsilon}(t),\bar{u}(t))\delta x(t)^4d\theta\nonumber\\
& &+\int_{0}^{1}\theta \Big( \varphi_{xx}(t,\theta\bar{x}(t)+
(1-\theta)x^{\varepsilon}(t),u^{\varepsilon}(t))\nonumber\\
& & \qquad\qquad\qquad\qquad
-\varphi_{xx}(t,\theta\bar{x}(t)
+(1-\theta)x^{\varepsilon}(t),\bar{u}(t))\Big)\delta x(t)^2 d\theta
\label{expdeltax 2}\\
&=& \varphi_x(t)\delta x(t)+\frac{1}{2}\varphi_{xx}(t)\delta x(t)^2
+\frac{1}{6}\varphi_{xxx}(t)\delta x(t)^3\nonumber\\
& & +\delta \varphi(t)\chi_{E_{\varepsilon}}(t)
+\delta \varphi_x(t)\delta x(t)\chi_{E_{\varepsilon}}(t)
+\frac{1}{2}\delta \varphi_{xx}(t)\delta x(t)^2\chi_{E_{\varepsilon}}(t)\nonumber\\
& & +\frac{1}{6}\int_{0}^{1}\theta^3 \varphi_{xxxx}(t,\theta\bar{x}(t)+
(1-\theta)x^{\varepsilon}(t),\bar{u}(t))\delta x(t)^4d\theta\nonumber\\
& &+\frac{1}{2}\int_{0}^{1}\theta^2 \Big( \varphi_{xxx}(t,\theta\bar{x}(t)+
(1-\theta)x^{\varepsilon}(t),u^{\varepsilon}(t))\nonumber\\
& &\qquad\qquad\qquad\quad
 -\varphi_{xxx}(t,\theta\bar{x}(t)
+(1-\theta)x^{\varepsilon}(t),\bar{u}(t))\Big)\delta x(t)^3 d\theta
\qquad\label{expdeltax 3}\\
&=& \varphi_x(t)\delta x(t)+\frac{1}{2}\varphi_{xx}(t)\delta x(t)^2
+\frac{1}{6}\varphi_{xxx}(t)\delta x(t)^3+\delta \varphi(t)\chi_{E_{\varepsilon}}(t)\nonumber\\
& &
+\delta \varphi_x(t)\delta x(t)\chi_{E_{\varepsilon}}(t)+\frac{1}{2}\delta \varphi_{xx}(t)\delta x(t)^2\chi_{E_{\varepsilon}}(t)
+\frac{1}{6}\delta \varphi_{xxx}(t)
\delta x(t)^3\chi_{E_{\varepsilon}}(t)\nonumber\\
& &+\frac{1}{6}\int_{0}^{1}\theta^3 \varphi_{xxxx}(t,\theta\bar{x}(t)+
(1-\theta)x^{\varepsilon}(t),\bar{u}(t))\delta x(t)^4d\theta\nonumber\\
& &+\frac{1}{6}\int_{0}^{1}\theta^3 \Big( \varphi_{xxxx}(t,\theta\bar{x}(t)+
(1-\theta)x^{\varepsilon}(t),u^{\varepsilon}(t))\nonumber\\
& &\qquad\qquad\qquad\quad-\varphi_{xxxx}(t,\theta\bar{x}(t)
+(1-\theta)x^{\varepsilon}(t),\bar{u}(t))\Big)\delta x(t)^4 d\theta.
\qquad\ \label{expdeltax 4}
\end{eqnarray}

By (\ref{expdeltax 1}) and (\ref{expdeltax 2}), we find that $\delta x(\cdot)$ is the solution to the following differential equation:
\begin{equation}\label{expdeltax equ1}
\left\{
\begin{array}{l}
d\delta x(t)= \Big[b_x(t)\delta x(t)+\frac{1}{2}b_{xx}(t)\delta x(t)^2+\delta b(t)\chi_{E_{\varepsilon}}(t)\\[+0.7em]
\qquad\qquad \ \ +\frac{1}{2}\int_{0}^{1}\theta^2 b_{xxx}(t,\theta\bar{x}(t)+
(1-\theta)x^{\varepsilon}(t),\bar{u}(t))\delta x(t)^3d\theta\\[+0.7em]
\qquad\qquad \ \ +\int_{0}^{1} \big( b_{x}(t,\theta\bar{x}(t)+
(1-\theta)x^{\varepsilon}(t),u^{\varepsilon}(t)) \\[+0.7em]
\qquad\qquad \qquad\qquad \ -b_{x}(t,\theta\bar{x}(t)
+(1-\theta)x^{\varepsilon}(t),\bar{u}(t))\big)\delta x(t) d\theta\Big]dt\\[+0.7em]
\qquad\qquad \ \ +\Big[ \sigma_x(t)\delta x(t)
+\frac{1}{2}\sigma_{xx}(t)\delta x(t)^2\\[+0.7em]
\qquad\qquad \ \
+\delta \sigma(t)\chi_{E_{\varepsilon}}(t)
+\delta \sigma_x(t)\delta x(t)\chi_{E_{\varepsilon}}(t)\\[+0.7em]
\qquad\qquad \ \ +\frac{1}{2}\int_{0}^{1}\theta^2 \sigma_{xxx}(t,\theta\bar{x}(t)+
(1-\theta)x^{\varepsilon}(t),\bar{u}(t))\delta x(t)^3d\theta\\[+0.7em]
\qquad\qquad \ \  +\int_{0}^{1}\theta \big( \sigma_{xx}(t,\theta\bar{x}(t)+
(1-\theta)x^{\varepsilon}(t),u^{\varepsilon}(t))\\[+0.7em]
\qquad\qquad\qquad\quad  -\sigma_{xx}(t,\theta\bar{x}(t)
+(1-\theta)x^{\varepsilon}(t),\bar{u}(t))\big)\delta x(t)^2
d\theta\Big]dW(t),\\[+0.7em]
\qquad\qquad \ \ t\in [0,T],\\[+1.1em]
\delta x(0)=0.
\end{array}\right.
\end{equation}

Similarly, by (\ref{expdeltax 2})--(\ref{expdeltax 4}), $\delta x(\cdot)$ is the solution to the stochastic differential equation
\begin{equation}\label{expdeltax equ2}
\quad\left\{
\begin{array}{l}
d\delta x(t)= \Big[b_x(t)\delta x(t)+\frac{1}{2}b_{xx}(t)\delta x(t)^2
+\frac{1}{6}b_{xxx}(t)\delta x(t)^3\\[+0.7em]
\qquad\qquad \ \
+\delta b(t)\chi_{E_{\varepsilon}}(t)
+\delta b_x(t)\delta x(t)\chi_{E_{\varepsilon}}(t)\\[+0.7em]
\qquad\qquad \ \ +\frac{1}{6}\int_{0}^{1}\theta^3 b_{xxxx}(t,\theta\bar{x}(t)+
(1-\theta)x^{\varepsilon}(t),\bar{u}(t))\delta x(t)^4d\theta\\[+0.7em]
\qquad\qquad \ \ +\int_{0}^{1}\theta\big( b_{xx}(t,\theta\bar{x}(t)+
(1-\theta)x^{\varepsilon}(t),u^{\varepsilon}(t))\\[+0.7em]
\qquad\qquad\qquad\qquad\quad -b_{xx}(t,\theta\bar{x}(t)
+(1-\theta)x^{\varepsilon}(t),\bar{u}(t))\big)
\delta x(t)^2 d\theta\Big]dt\\[+0.7em]
\qquad\qquad \ \ +\Big[ \sigma_x(t)\delta x(t)+\frac{1}{2}\sigma_{xx}(t)\delta x(t)^2
+\frac{1}{6}\sigma_{xxx}(t)\delta x(t)^3\\[+0.7em]
\qquad\qquad \ \
+\delta \sigma(t)\chi_{E_{\varepsilon}}(t)
+\delta \sigma_x(t)\delta x(t)\chi_{E_{\varepsilon}}(t)+\frac{1}{2}\delta \sigma_{xx}(t)\delta x(t)^{2}
\chi_{E_{\varepsilon}}(t)\\[+0.7em]
\qquad\qquad \ \
+\frac{1}{6}\int_{0}^{1}\theta^3 \sigma_{xxxx}(t,\theta\bar{x}(t)+
(1-\theta)x^{\varepsilon}(t),\bar{u}(t))\delta x(t)^4d\theta\\[+0.7em]
\qquad\qquad \ \ +\frac{1}{2}\int_{0}^{1}\theta^{2}
\big( \sigma_{xxx}(t,\theta\bar{x}(t)+
(1-\theta)x^{\varepsilon}(t),u^{\varepsilon}(t))\\[+0.7em]
\qquad\qquad \qquad\quad\ -\sigma_{xxx}(t,\theta\bar{x}(t)
+(1-\theta)x^{\varepsilon}(t),\bar{u}(t))\big)
\delta x(t)^3 d\theta\Big]dW(t),\\[+0.7em]
\qquad\qquad \ \ t\in [0,T],\\[+1.1em]
\delta x(0)=0,
\end{array}\right.
\end{equation}

and the stochastic differential equation
\begin{equation}\label{expdeltax equ3}
\left\{
\begin{array}{l}
d\delta x(t)= \Big[b_x(t)\delta x(t)+\frac{1}{2}b_{xx}(t)\delta x(t)^2
+\frac{1}{6}b_{xxx}(t)\delta x(t)^3\\[+0.6em]
\qquad\qquad \ \
+\delta b(t)\chi_{E_{\varepsilon}}(t)
+\delta b_x(t)\delta x(t)\chi_{E_{\varepsilon}}(t)
+\frac{1}{2}\delta b_{xx}(t)\delta x(t)^{2}\chi_{E_{\varepsilon}}(t)\\[+0.6em]
\qquad\qquad \ \
+\frac{1}{6}\int_{0}^{1}\theta^3 b_{xxxx}(t,\theta\bar{x}(t)+
(1-\theta)x^{\varepsilon}(t),\bar{u}(t))\delta x(t)^4d\theta\\[+0.6em]
\qquad\qquad \qquad  +\frac{1}{2}\int_{0}^{1}\theta^{2} \big( b_{xxx}(t,\theta\bar{x}(t)
+(1-\theta)x^{\varepsilon}(t),u^{\varepsilon}(t))\\[+0.6em]
\qquad\qquad\qquad\qquad\ -b_{xxx}(t,\theta\bar{x}(t)
+(1-\theta)x^{\varepsilon}(t),\bar{u}(t))\big)\delta x(t)^3
d\theta\Big]dt\\[+0.6em]
\qquad\qquad \ \ +\Big[ \sigma_x(t)\delta x(t)+\frac{1}{2}\sigma_{xx}(t)\delta x(t)^2
+\frac{1}{6}\sigma_{xxx}(t)\delta x(t)^3\\[+0.6em]
\qquad\qquad \ \
+\delta \sigma(t)\chi_{E_{\varepsilon}}(t)
+\delta \sigma_x(t)\delta x(t)\chi_{E_{\varepsilon}}(t)
\\[+0.6em]
\qquad\qquad \ \
+\frac{1}{2}\delta \sigma_{xx}(t)\delta x(t)^{2}
\chi_{E_{\varepsilon}}(t)
+\frac{1}{6}\delta \sigma_{xxx}(t)\delta x(t)^{3}\chi_{E_{\varepsilon}}(t)\\[+0.6em]
\qquad\qquad \ \ +\frac{1}{6}\int_{0}^{1}\theta^3 \sigma_{xxxx}(t,\theta\bar{x}(t)+
(1-\theta)x^{\varepsilon}(t),\bar{u}(t))\delta x(t)^4d\theta\\[+0.6em]
\qquad\qquad \ \ +\frac{1}{6}\int_{0}^{1}\theta^{3}
\big( \sigma_{xxxx}(t,\theta\bar{x}(t)+
(1-\theta)x^{\varepsilon}(t),u^{\varepsilon}(t))\\[+0.6em]
\qquad\qquad\qquad\ -\sigma_{xxxx}(t,\theta\bar{x}(t)+
(1-\theta)x^{\varepsilon}(t),\bar{u}(t))\big)\delta x(t)^4
 d\theta\Big]dW(t),\\[+0.6em]
\qquad\qquad \ \ t\in [0,T],\\[+0.6em]
\delta x(0)=0.
\end{array}\right.
\end{equation}

Combining the variational equations (\ref{firstvariequ}) and (\ref{secondvariequ}) with the equation (\ref{expdeltax equ1}), we see that $r_{2}(\cdot)$ is the solution to the stochastic differential equation:

\begin{equation}\label{r2}
\quad\left\{
\begin{array}{l}
dr_{2}(t)=\Big[b_x(t)r_{2}(t)
+\frac{1}{2}b_{xx}(t)\big(\delta x(t)^2-y_{1}^{\varepsilon}(t)^{2}\big)\\[+0.6em]
\qquad\qquad
 +\frac{1}{2}\int_{0}^{1}\theta^2 b_{xxx}(t,\theta\bar{x}(t)+
(1-\theta)x^{\varepsilon}(t),\bar{u}(t))\delta x(t)^3d\theta\\[+0.6em]
\qquad\qquad
+\int_{0}^{1} \big( b_{x}(t,\theta\bar{x}(t)+
(1-\theta)x^{\varepsilon}(t),u^{\varepsilon}(t))\\[+0.6em]
\qquad\qquad\qquad\qquad
 -b_{x}(t,\theta\bar{x}(t)+(1-\theta)x^{\varepsilon}(t),\bar{u}(t))\big)
\delta x(t)d\theta
\Big]dt\\[+0.6em]
\qquad \qquad  +\Big[ \sigma_x(t)r_{2}(t)+\frac{1}{2}\sigma_{xx}(t)
\big(\delta x(t)^2-y_{1}^{\varepsilon}(t)^{2}\big)
+\delta \sigma_x(t)r_{1}(t)\chi_{E_{\varepsilon}}(t)\\[+0.6em]
\qquad\qquad +\frac{1}{2}\int_{0}^{1}\theta^2 \sigma_{xxx}(t,\theta\bar{x}(t)
+(1-\theta)x^{\varepsilon}(t),\bar{u}(t))\delta x(t)^3d\theta\\[+0.6em]
\qquad\qquad
 +\int_{0}^{1}\theta ( \sigma_{xx}(t,\theta\bar{x}(t)+
(1-\theta)x^{\varepsilon}(t),u^{\varepsilon}(t))\\[+0.6em]
\qquad\qquad\qquad\qquad
-\sigma_{xx}(t,\theta\bar{x}(t)+(1-\theta)x^{\varepsilon}(t),\bar{u}(t)))\delta x(t)^2 d\theta
\Big]dW(t),\\[+0.6em]
\qquad\qquad \ t\in [0,T],\\
r_{2}(0)=0.
\end{array}\right.
\end{equation}

By the conditions  (C1)--(C2), we have
\begin{eqnarray}\label{lemma3 1equ1add}
& &\me~\Big[\int_{0}^{T}\Big|\int_{0}^{1}\theta^2 b_{xxx}(t,\theta\bar{x}(t)+
(1-\theta)x^{\varepsilon}(t),\bar{u}(t))\delta x(t)^3 d\theta\Big|dt\Big]^{\beta}
\qquad\ \nonumber\\
&\le& C~\me\Big[\int_{0}^{T}\big|\delta x(t)^3\big|dt\Big]^{\beta}
\le C~\me \Big[\sup_{t\in[0,T]}|\delta x(t)|^{3\beta}\Big]
\le C\varepsilon^{\frac{3\beta}{2}},
\end{eqnarray}
\begin{eqnarray}\label{lemma3 1equ1}
& &\me~\Big[\int_{0}^{T}\Big|\int_{0}^{1}\Big( b_{x}(t,\theta\bar{x}(t)
+(1-\theta)x^{\varepsilon}(t),u^{\varepsilon}(t))\nonumber\\
& &\qquad\qquad\qquad
-b_{x}(t,\theta\bar{x}(t)
+(1-\theta)x^{\varepsilon}(t),\bar{u}(t))\Big)\delta x(t) d\theta\Big|dt\Big]^{\beta}\qquad\ \ \nonumber\\
&\le& C~\me\Big[\int_{0}^{T}\big|\delta x(t) \chi_{E_{\varepsilon}}(t)\big|dt\Big]^{\beta}
\le C~\me \Big[\sup_{t\in[0,T]}|\delta x(t)|^{\beta}\Big]\varepsilon^{\beta}
\le C\varepsilon^{\frac{3\beta}{2}},\quad
\end{eqnarray}
\begin{eqnarray}\label{lemma3 1equ2}
& &\me~\Big[\int_{0}^{T}\Big|\int_{0}^{1}\theta^2 \sigma_{xxx}(t,\theta\bar{x}(t)+
(1-\theta)x^{\varepsilon}(t),\bar{u}(t))\delta x(t)^3 d\theta\Big|^2dt\Big]^{\frac{\beta}{2}}
\qquad\ \nonumber\\
&\le& C\me~\Big[\int_{0}^{T}\big|\delta x(t)^3\big|^{2}dt\Big]^{\frac{\beta}{2}}
\le C\me~\Big[\sup_{t\in[0,T]}|\delta x(t)|^{3\beta}\Big]
\le C\varepsilon^{\frac{3\beta}{2}},
\end{eqnarray}
and
\begin{eqnarray}\label{lemma3 1equ3}
\quad\ \ & &\me~\Big[\int_{0}^{T}\Big|\int_{0}^{1}\theta \big( \sigma_{xx}(t,\theta\bar{x}(t)
+(1-\theta)x^{\varepsilon}(t),u^{\varepsilon}(t))\nonumber\\
& &\qquad\qquad\qquad\qquad
-\sigma_{xx}(t,\theta\bar{x}(t)
+(1-\theta)x^{\varepsilon}(t),\bar{u}(t))\big)\delta x(t)^2 d\theta\Big|^{2}dt\Big]^{\frac{\beta}{2}}\quad\quad\nonumber\\
&\le& C\me~\Big[\int_{0}^{T}|\delta x(t)^2 \chi_{E_{\varepsilon}}(t)|^{2}dt\Big]^{\frac{\beta}{2}}
\le C\me~ \Big[\sup_{t\in[0,T]}|\delta x(t)|^{2\beta}\Big]\varepsilon^{\frac{\beta}{2}}
\le C\varepsilon^{\frac{3\beta}{2}}.
\end{eqnarray}

On the other hand, combining (\ref{estlemma y1}) and (\ref{estimateofdeltax}) with (\ref{estlemma r1}), we have
\begin{eqnarray}\label{deltx2-y12}
& &\me\Big[\sup_{t\in[0,T]}|\delta x(t)^2-y_{1}^{\varepsilon}(t)^{2}|^{\beta}\Big]
=\me\Big[\sup_{t\in[0,T]}\big(|r_{1}(t)|^{\beta}|\delta x(t)+y_{1}^{\varepsilon}(t)
|^{\beta}\big)\Big]\nonumber\\
&\le&\Big(\me\Big[\sup_{t\in[0,T]}|r_{1}(t)|^{2\beta}\Big]\Big)^{\frac{1}{2}}
\Big(\me\Big[\sup_{t\in[0,T]}|\delta x(t)+y_{1}^{\varepsilon}(t)
|^{2\beta}\Big]\Big)^{\frac{1}{2}}
\le C\varepsilon^{\frac{3\beta}{2}}.\quad\quad
\end{eqnarray}

Then, combining (\ref{estimateofdeltax}), (\ref{lemma3 1equ1add})--(\ref{deltx2-y12}) with (\ref{r2}), we obtain that
\begin{eqnarray}\label{estlemma r2}
& &\me~\Big[\sup_{t\in[0,T]}|r_{2}(t)|^{\beta}\Big]\nonumber\\
&\le &C\me~\Big[\int_{0}^{T}\Big|\frac{1}{2}b_{xx}(t)
\big(\delta x(t)^2-y_{1}^{\varepsilon}(t)^{2}\big)\nonumber\\
& &\qquad+\frac{1}{2}\int_{0}^{1}\theta^2 b_{xxx}(t,\theta\bar{x}(t)+
(1-\theta)x^{\varepsilon}(t),\bar{u}(t))\delta x(t)^3d\theta\nonumber\\
& &\qquad+\int_{0}^{1} \big( b_{x}(t,\theta\bar{x}(t)+
(1-\theta)x^{\varepsilon}(t),u^{\varepsilon}(t))\nonumber\\
& &\qquad\qquad\qquad
-b_{x}(t,\theta\bar{x}(t)+(1-\theta)x^{\varepsilon}(t),\bar{u}(t))\big)\delta x(t) d\theta\Big|dt\Big]^{\beta}\nonumber\\
& &+C\me~\Big[\int_{0}^{T}\Big|\frac{1}{2}\sigma_{xx}(t)
\big(\delta x(t)^2-y_{1}^{\varepsilon}(t)^{2}\big)
+\delta \sigma_x(t)r_{1}(t)\chi_{E_{\varepsilon}}(t)\nonumber\\
& &\qquad+\frac{1}{2}\int_{0}^{1}\theta^2 \sigma_{xxx}(t,\theta\bar{x}(t)+
(1-\theta)x^{\varepsilon}(t),\bar{u}(t))\delta x(t)^3d\theta\nonumber\\
& &\qquad
 +\int_{0}^{1}\theta \big( \sigma_{xx}(t,\theta\bar{x}(t)+
(1-\theta)x^{\varepsilon}(t),u^{\varepsilon}(t))\nonumber\\
& &\qquad\qquad\qquad-\sigma_{xx}(t,\theta\bar{x}(t)
+(1-\theta)x^{\varepsilon}(t),\bar{u}(t))\big)\delta x(t)^2 d\theta\Big|^{2}dt\Big]^{\frac{\beta}{2}}\nonumber\\
&\le&C\Big\{\me\Big[\sup_{t\in[0,T]}|\delta x(t)^2
-y_{1}^{\varepsilon}(t)^{2}|^{\beta}\Big]
+\me\Big[\sup_{t\in[0,T]}|\delta x(t)|^{3\beta}\Big]
\nonumber\\
& &\quad
+\me~\Big[\sup_{t\in[0,T]}|\delta x(t)|^{\beta}\Big]\varepsilon^{\beta}
+\me~\Big[\sup_{t\in[0,T]}|\delta x(t)|^{2\beta}\Big]\varepsilon^{\frac{\beta}{2}}
\Big\}
\le C \varepsilon^{\frac{3\beta}{2}}.
\end{eqnarray}
This proves the estimation for $ r_{2}(\cdot)$.

Now, we prove the estimate for $r_3(\cdot)$.

Combining the variational equation (\ref{firstvariequ}),
(\ref{secondvariequ}) and (\ref{thirdvariequ}) with (\ref{expdeltax equ2}), we see that, $r_3(\cdot)$ is the solution to the stochastic differential equation:
\begin{equation}\label{r3}
\qquad\left\{
\begin{array}{l}
dr_{3}(t)=\Big[b_x(t)r_{3}(t)+\frac{1}{2}b_{xx}(t)\big(\delta x(t)^2-\gamma(t)^{2}\big)\\[+0.6em]
\qquad\qquad+\frac{1}{6}b_{xxx}(t)\big(\delta x(t)^3-y_{1}^{\varepsilon}(t)^{3}\big)
+\delta b_x(t)r_{1}(t)\chi_{E_{\varepsilon}}(t)\\[+0.6em]
\qquad\qquad
+\frac{1}{6}\int_{0}^{1}\theta^3 b_{xxxx}(t,\theta\bar{x}(t)+
(1-\theta)x^{\varepsilon}(t),\bar{u}(t))\delta x(t)^4d\theta\\[+0.6em]
\qquad \qquad
+\int_{0}^{1}\theta \big( b_{xx}(t,\theta\bar{x}(t)+
(1-\theta)x^{\varepsilon}(t),u^{\varepsilon}(t))\\[+0.6em]
\qquad\qquad
-b_{xx}(t,\theta\bar{x}(t)+(1-\theta)x^{\varepsilon}(t),\bar{u}(t))\big)
\delta x(t)^2 d\theta
\Big]dt\\[+0.6em]
\qquad \qquad  +\Big[ \sigma_x(t)r_{3}(t)+\frac{1}{2}\sigma_{xx}(t)
\big(\delta x(t)^2-\gamma(t)^{2}\big)\\[+0.6em]
\qquad\qquad
+\frac{1}{6}\sigma_{xxx}(t)\big(\delta x(t)^3-y_{1}^{\varepsilon}(t)^{3}\big)
+\delta \sigma_x(t)r_{2}(t)\chi_{E_{\varepsilon}}(t)\\[+0.6em]
\qquad\qquad+\frac{1}{2}\delta\sigma_{xx}(t)\big(\delta x(t)^2-y_{1}^{\varepsilon}(t)^{2}\big)\chi_{E_{\varepsilon}}(t)\\[+0.6em]
\qquad\qquad
+\frac{1}{6}\int_{0}^{1}\theta^3
 \sigma_{xxxx}(t,\theta\bar{x}(t)+
(1-\theta)x^{\varepsilon}(t),\bar{u}(t))\delta x(t)^4d\theta\\[+0.6em]
\qquad\qquad
+\frac{1}{2}\int_{0}^{1}\theta^{2} \big( \sigma_{xxx}(t,\theta\bar{x}(t)+
(1-\theta)x^{\varepsilon}(t),u^{\varepsilon}(t))\\[+0.6em]
\qquad\qquad \ -\sigma_{xxx}(t,\theta\bar{x}(t)
+(1-\theta)x^{\varepsilon}(t),\bar{u}(t))\big)\delta x(t)^3 d\theta
\Big]dW(t),\qquad\qquad\\[+0.6em]
r_{3}(0)=0.
\end{array}\right.
\end{equation}

Similar to (\ref{lemma3 1equ1add}), we can prove that
\begin{eqnarray}\label{r3 equ1}
& &\me~\Big[\int_{0}^{T}\Big|\int_{0}^{1}\theta^3 b_{xxxx}(t,\theta\bar{x}(t)+
(1-\theta)x^{\varepsilon}(t),\bar{u}(t))
\delta x(t)^4 d\theta\Big|dt\Big]^{\beta} \nonumber\\
&\le& C\me~ \Big[\sup_{t\in[0,T]}|\delta x(t)|^{4\beta}\Big]
\le C\varepsilon^{2\beta}.
\end{eqnarray}
Similar to (\ref{lemma3 1equ1}), we have
\begin{eqnarray}\label{r3 equ2}
& &\me~\Big[\int_{0}^{T}\Big|\int_{0}^{1}\theta^2 \big\{ \sigma_{xxx}(t,\theta\bar{x}(t)
+(1-\theta)x^{\varepsilon}(t),u^{\varepsilon}(t))\nonumber\\
& &\qquad\qquad\qquad-\sigma_{xxx}(t,\theta\bar{x}(t)
+(1-\theta)x^{\varepsilon}(t),\bar{u}(t))\big\}\delta x(t)^3 d\theta\Big|dt\Big]^{\beta}\nonumber\\
&\le& C\me~\Big[\int_{0}^{T}|\delta x(t)^3 \chi_{E_{\varepsilon}}(t)|dt\Big]^{\beta}
\le C\me ~\Big[\sup_{t\in[0,T]}|\delta x(t)|^{3\beta}\Big]\varepsilon^{\frac{\beta}{2}}
\le C\varepsilon^{2\beta}.
\end{eqnarray}
In a similar way, we have
\begin{eqnarray}\label{r3 equ3}
& &\me~\Big[\int_{0}^{T}\Big|\int_{0}^{1}\theta^3 \sigma_{xxxx}(t,\theta\bar{x}(t)+
(1-\theta)x^{\varepsilon}(t),\bar{u}(t))\delta x(t)^4 d\theta\Big|^{2}dt\Big]^{\frac{\beta}{2}} \nonumber\\
&\le& C\me~\Big[\sup_{t\in[0,T]}|\delta x(t)|^{4\beta}\Big]
\le C\varepsilon^{2\beta},
\end{eqnarray}
and
\begin{eqnarray}\label{r3 equ4}
& &\me~\Big(\int_{0}^{T}\Big|\int_{0}^{1}\theta^2 \big\{ \sigma_{xxx}(t,\theta\bar{x}(t)
+(1-\theta)x^{\varepsilon}(t),u^{\varepsilon}(t))\nonumber\\
& &\qquad\qquad\qquad-\sigma_{xxx}(t,\theta\bar{x}(t)
+(1-\theta)x^{\varepsilon}(t),\bar{u}(t))\big\}\delta x(t)^3 d\theta\Big|^{2}dt\Big)^{\frac{\beta}{2}}\nonumber\\
&\le& C~\me\Big(\int_{0}^{T}|\delta x(t)^3 \chi_{E_{\varepsilon}}(t)|^{2}dt\Big)^{\frac{\beta}{2}}
\le C~\me \Big[\sup_{t\in[0,T]}|\delta x(t)|^{3\beta}\Big]\varepsilon^{\frac{\beta}{2}}
\le C\varepsilon^{2\beta}.
\end{eqnarray}

On the other hand, by (\ref{estlemma y1}), (\ref{estlemma y2}), (\ref{estimateofdeltax}) and (\ref{estlemma r2}), we get that
\begin{eqnarray}\label{r3 equ5}
& &\me~\Big[\sup_{t\in[0,T]}|\delta x(t)^2-\gamma(t)^{2}|^{\beta}\Big]\nonumber\\
&=&\me~\Big[\sup_{t\in[0,T]}\Big(|r_{2}(t)|^{\beta}\cdot|\delta x(t)+y_{1}^{\varepsilon}(t)
+y_{2}^{\varepsilon}(t)
|^{\beta}\Big)\Big]\nonumber\\
&\le&\Big(\me~\Big[\sup_{t\in[0,T]}|r_{2}(t)|^{2\beta}\Big]\Big)^{\frac{1}{2}}
\cdot\Big(\me~\Big[\sup_{t\in[0,T]}|\delta x(t)+y_{1}^{\varepsilon}(t)
+y_{2}^{\varepsilon}(t)
|^{2\beta}\Big]\Big)^{\frac{1}{2}}\nonumber\\
&\le& C\varepsilon^{2\beta}.
\end{eqnarray}
Also, by (\ref{estlemma y1}), (\ref{estimateofdeltax}) and (\ref{estlemma r1}), we have
\begin{eqnarray}\label{r3 equ6}
& &\me~\Big[\sup_{t\in[0,T]}|\delta x^{3}-y_{1}^{\varepsilon}(t)^{3}|^{\beta}\Big]\nonumber\\
&=&\me~\Big[\sup_{t\in[0,T]}\Big(|r_{1}(t)|^{\beta}\cdot|\delta x(t)^2+\delta x(t) y_{1}^{\varepsilon}(t)+y_{1}^{\varepsilon}(t)^2
|^{\beta}\Big)\Big]\nonumber\\
&\le&\Big(\me\Big[\sup_{t\in[0,T]}|r_{1}(t)|^{2\beta}\Big]\Big)^{\frac{1}{2}}\cdot
\Big(\me\Big[\sup_{t\in[0,T]}|\delta x(t)^2+\delta x(t) y_{1}^{\varepsilon}(t)
+y_{1}^{\varepsilon}(t)^2
|^{2\beta}\Big]\Big)^{\frac{1}{2}}\nonumber \\
&\le& C\varepsilon^{2\beta}.
\end{eqnarray}

Combining (\ref{r3 equ1})--(\ref{r3 equ6}) with (\ref{estimateofx}), we obtain that
\begin{eqnarray}\label{estlemma r3}
& & \me~\Big[\sup_{t\in[0,T]}|r_{3}(t)|^{\beta}\Big]\nonumber\\
&\le &C\me~\Big[\int_{0}^{T}\Big|\frac{1}{2}b_{xx}(t)\big(\delta x(t)^2-\gamma(t)^{2}\big)\nonumber\\
& &+\frac{1}{6}b_{xxx}(t)\big(\delta x(t)^3-y_{1}^{\varepsilon}(t)^{3}\big)
+\delta b_x(t)r_{1}(t)\chi_{E_{\varepsilon}}(t)\nonumber\\
& &+\frac{1}{6}\int_{0}^{1}\theta^3 b_{xxx}(t,\theta\bar{x}(t)+
(1-\theta)x^{\varepsilon}(t),\bar{u}(t))\delta x(t)^4d\theta\nonumber\\
& &+\int_{0}^{1}\theta \Big( b_{xx}(t,\theta\bar{x}(t)+
(1-\theta)x^{\varepsilon}(t),u^{\varepsilon}(t))\nonumber\\
& &\qquad\qquad\qquad -b_{xx}(t,\theta\bar{x}(t)+(1-\theta)x^{\varepsilon}(t),\bar{u}(t))\Big)\delta x(t)^2 d\theta\Big|dt\Big]^{\beta}\nonumber\\
& &+C\me~\Big[\int_{0}^{T}\Big|\frac{1}{2}\sigma_{xx}(t)\big(\delta x(t)^2-\gamma(t)^{2}\big)\nonumber\\
& &+\frac{1}{6}\sigma_{xxx}(t)\big(\delta x(t)^3-y_{1}^{\varepsilon}(t)^{3}\big)
+\delta \sigma_x(t)r_{2}(t)\chi_{E_{\varepsilon}}(t)\nonumber\\
& &+\frac{1}{2}\delta\sigma_{xx}(t)\big(\delta x(t)^2
-y_{1}^{\varepsilon}(t)^{2}\big)\chi_{E_{\varepsilon}}(t)\nonumber\\
& &+\frac{1}{6}\int_{0}^{1}\theta^3 \sigma_{xxxx}(t,\theta\bar{x}(t)+
(1-\theta)x^{\varepsilon}(t),\bar{u}(t))\delta x(t)^4d\theta\nonumber\\
& &+\frac{1}{2}\int_{0}^{1}\theta^{2} \Big( \sigma_{xxx}(t,\theta\bar{x}(t)+
(1-\theta)x^{\varepsilon}(t),u^{\varepsilon}(t)) \nonumber\\
& &\qquad\qquad\qquad
-\sigma_{xxx}(t,\theta\bar{x}(t)+(1-\theta)x^{\varepsilon}(t),\bar{u}(t))\Big)\delta x(t)^3 d\theta\Big|^{2}dt\Big]^{\frac{\beta}{2}}\nonumber\\
&\le&C\Big\{\me~\Big[\sup_{t\in[0,T]}|\delta x(t)^2-\gamma(t)^{2}|^{\beta}\Big]
+\me~\Big[\sup_{t\in[0,T]}|\delta x(t)^3-y_{1}^{\varepsilon}(t)^{3}|^{\beta}\Big]\nonumber\\
& &+\me~\Big[\sup_{t\in[0,T]}|r_{1}(t)|^{\beta}\Big]\varepsilon^{\beta}
+\varepsilon^{2\beta}+\varepsilon^{2\beta}+\me~\Big[\sup_{t\in[0,T]}|\delta x(t)^2-\gamma(t)^{2}|^{\beta}\Big]\nonumber\\
& &
+\me~\Big[\sup_{t\in[0,T]}|\delta x(t)^3-y_{1}^{\varepsilon}(t)^{3}|^{\beta}\Big]
+\me~\Big[\sup_{t\in[0,T]}|r_{2}(t)|^{\beta}\Big]\varepsilon^{\frac{\beta}{2}}\nonumber\\
& &+ \me~\Big[\sup_{t\in[0,T]}|\delta x(t)^2
-y_{1}^{\varepsilon}(t)^{2}|^{\beta}\Big]\varepsilon^{\frac{\beta}{2}}
+\varepsilon^{2\beta}+\varepsilon^{2\beta}\Big\}
\le C \varepsilon^{2\beta}.
\end{eqnarray}
This proves the estimate for $r_{3}(\cdot)$.

\textbf{Step 3:} We now estimate $\|(\delta x)^{2}-\eta^2 \|_{\infty,\beta}^{\beta}$,
$\|(\delta x)^{3}-\gamma^3 \|_{\infty,\beta}^{\beta}$ and $\|(\delta x)^{4}-(y_{1}^{\varepsilon})^4 \|_{\infty,\beta}^{\beta}$.

First, by (\ref{estlemma y1})--(\ref{estlemma y3}),
(\ref{estimateofdeltax}) and (\ref{estlemma r3}), we have
\begin{eqnarray}\label{step3eta}
& &\me\Big[\sup_{t\in[0,T]}|\delta x(t)^2-\eta(t)^{2}|^{\beta}\Big]\nonumber\\
&=&\me\Big[\sup_{t\in[0,T]}\Big(|r_{3}(t)|^{\beta}\cdot|\delta x(t)+y_{1}^{\varepsilon}(t)+y_{2}^{\varepsilon}(t)+y_{3}^{\varepsilon}(t)
|^{\beta}\Big)\Big]\nonumber\\
&\le&\Big(\me\Big[\sup_{t\in[0,T]}|r_{3}(t)|^{2\beta}\Big]\Big)^{\frac{1}{2}}
\Big(\me\Big[\sup_{t\in[0,T]}|\delta x(t)+y_{1}^{\varepsilon}(t)+y_{2}^{\varepsilon}(t)+y_{3}^{\varepsilon}(t)
|^{2\beta}\Big]\Big)^{\frac{1}{2}}\nonumber\\
&\le& C\varepsilon^{\frac{5\beta}{2}}.
\end{eqnarray}
Next, by (\ref{estlemma y1}), (\ref{estlemma y2}), (\ref{estimateofdeltax}) and (\ref{estlemma r2}), we get
\begin{eqnarray}\label{step3gamma}
& &\me\Big[\sup_{t\in[0,T]}|\delta x^{3}-\gamma(t)^{3}|^{\beta}\Big]\nonumber\\
&=&\me\Big[\sup_{t\in[0,T]}\Big(|r_{2}(t)|^{\beta}\cdot|\delta x(t)^2+\delta x(t) \gamma(t)+\gamma(t)^2
|^{\beta}\Big)\Big]\nonumber\\
&\le&\Big(\me\Big[\sup_{t\in[0,T]}|r_{2}(t)|^{2\beta}\Big]\Big)^{\frac{1}{2}}\cdot
\Big(\me\Big[\sup_{t\in[0,T]}|\delta x(t)^2
+\delta x(t) \gamma(t)+\gamma(t)^2
|^{2\beta}\Big]\Big)^{\frac{1}{2}}\ \nonumber\\
&\le& C\varepsilon^{\frac{5\beta}{2}}.
\end{eqnarray}
Finally, by (\ref{estlemma y1}), (\ref{estimateofdeltax}) and (\ref{estlemma r1}), we have
\begin{eqnarray}\label{step3y1}
& &\me~\Big[\sup_{t\in[0,T]}|\delta x^{4}-(y_{1}^{\varepsilon})^4|^{\beta}\Big]\nonumber\\
&=&\me~\Big[\sup_{t\in[0,T]}\Big(|r_{1}(t)|^{\beta}\cdot
|\delta x(t)+y_{1}^{\varepsilon}(t)|^{\beta}\cdot
|\delta x(t)^2+y_{1}^{\varepsilon}(t)^2|^{\beta}\Big)\Big]\nonumber\\
&\le&\Big(\me~\Big[\sup_{t\in[0,T]}|r_{1}(t)|^{2\beta}\Big]\Big)^{\frac{1}{2}}\cdot
\Big(\me~\Big[\sup_{t\in[0,T]}|\delta x(t)+y_{1}^{\varepsilon}(t)
|^{4\beta}\Big]\Big)^{\frac{1}{4}}\cdot\nonumber\\
& &\qquad\qquad\qquad\qquad\qquad\qquad
\Big(\me~\Big[\sup_{t\in[0,T]}|\delta x(t)^2+y_{1}^{\varepsilon}(t)^2
|^{4\beta}\Big]\Big)^{\frac{1}{4}}\qquad \nonumber\\
&\le& C\varepsilon^{\frac{5\beta}{2}}.
\end{eqnarray}

\textbf{Step 4:} Estimate for $\|r_{4}\|_{\infty,\beta}^{\beta}$.

By (\ref{firstvariequ})--(\ref{forthvariequ}) and (\ref{expdeltax equ3}), we obtain that
\begin{equation}\nonumber
\left\{
\begin{array}{l}
dr_{4}(t)=\Big[b_x(t)r_{4}(t)+\frac{1}{2}b_{xx}(t)\big(\delta x(t)^2-\eta(t)^{2}\big)\\[+0.6em]
\qquad\qquad+\frac{1}{6}b_{xxx}(t)\big(\delta x(t)^3-\gamma(t)^{3}\big)
+\delta b_x(t)r_{2}(t)\chi_{E_{\varepsilon}}(t)\\[+0.6em]
\qquad\qquad+\frac{1}{2}\delta b_{xx}(t)\big(\delta x(t)^{2}-y_{1}^{\varepsilon}(t)^{2}\big)\chi_{E_{\varepsilon}}(t)\\[+0.6em]
\qquad\qquad +\frac{1}{6}\int_{0}^{1}\theta^3 b_{xxx}(t,\theta\bar{x}(t)+
(1-\theta)x^{\varepsilon}(t),\bar{u}(t))\delta x(t)^4d\theta\\[+0.6em]
\qquad\qquad
-\frac{1}{24}b_{xxxx}(t)y_{1}^{\varepsilon}(t)^{4}\\[+0.6em]
\qquad\qquad
+\frac{1}{2}\int_{0}^{1}\theta^{2} \big( b_{xxx}(t,\theta\bar{x}(t)+
(1-\theta)x^{\varepsilon}(t),u^{\varepsilon}(t)) \\[+0.6em]
\qquad\qquad\qquad\qquad
-b_{xxx}(t,\theta\bar{x}(t)
+(1-\theta)x^{\varepsilon}(t),\bar{u}(t))\big)\delta x(t)^3 d\theta
\Big]dt\\[+0.6em]
\qquad \qquad +\Big[ \sigma_x(t)r_{4}(t)
+\frac{1}{2}\sigma_{xx}(t)\big(\delta x(t)^2-\eta(t)^{2}\big)\\[0.6em]
\qquad\qquad
+\frac{1}{6}\sigma_{xxx}(t)\big(\delta x(t)^3-\gamma(t)^{3}\big)
+\delta \sigma_x(t)r_{3}(t)\chi_{E_{\varepsilon}}(t)\\[+0.6em]
\qquad\qquad+\frac{1}{2}\delta\sigma_{xx}(t)
\big(\delta x(t)^2-\gamma(t)^{2}\big)\chi_{E_{\varepsilon}}(t)\\[+0.6em]
\qquad\qquad
+\frac{1}{6}\delta \sigma_{xxx}(t)\big(\delta x(t)^3-y_{1}^{\varepsilon}(t)^{3}\big)\chi_{E_{\varepsilon}}(t)\\[+0.6em]
\qquad\qquad +\frac{1}{6}\int_{0}^{1}\theta^3 \sigma_{xxxx}(t,\theta\bar{x}(t)
+(1-\theta)x^{\varepsilon}(t),\bar{u}(t))\delta x(t)^4d\theta\\[+0.6em]
\qquad\qquad
-\frac{1}{24}\sigma_{xxxx}(t)y_{1}^{\varepsilon}(t)^{4}\\[+0.6em]
\qquad\qquad
+\frac{1}{6}\int_{0}^{1}\theta^3 \big( \sigma_{xxxx}(t,\theta\bar{x}(t)
+(1-\theta)x^{\varepsilon}(t),u^{\varepsilon}(t))\\[+0.6em]
\qquad\qquad\qquad\
-\sigma_{xxxx}(t,\theta\bar{x}(t)
+(1-\theta)x^{\varepsilon}(t),\bar{u}(t))\big)\delta x(t)^4 d\theta
\Big]dW(t),\\[+0.6em]
\qquad\qquad t\in [0,T],\\[+0.6em]
r_{4}(0)=0.
\end{array}\right.
\end{equation}

By (\ref{step3y1}) and the conditions (C1)--(C2),  we have
\begin{eqnarray}\label{lemma3 1equ7}
& &\me\Big[\int_{0}^{T}\Big|\frac{1}{6}\int_{0}^{1}\theta^3
 b_{xxxx}(t,\theta\bar{x}(t)+
(1-\theta)x^{\varepsilon}(t),\bar{u}(t))\delta x(t)^4 d\theta\nonumber\\
& &\qquad\qquad\qquad\qquad\qquad\qquad\qquad\qquad\qquad
-\frac{1}{24}b_{xxxx}(t)y_{1}^{\varepsilon}(t)^{4}
\Big|dt\Big]^{\beta}\nonumber\\
&\le& C\me~\Big[\frac{1}{6}\int_{0}^{T}\Big|\int_{0}^{1}
\theta^3\Big( b_{xxxx}(t,\theta\bar{x}(t)+
(1-\theta)x^{\varepsilon}(t),\bar{u}(t))\nonumber\\
& &\qquad\qquad\qquad\qquad\qquad\qquad\qquad\qquad\qquad
-b_{xxxx}(t)\Big)\delta x(t)^4d\theta\Big|dt\Big]^{\beta}\nonumber\\
& &+C\me~\Big[\frac{1}{6}\int_{0}^{T}\Big|\int_{0}^{1}
\theta^3 b_{xxxx}(t)\delta x(t)^4d\theta-\frac{1}{4}b_{xxxx}(t)y_{1}^{\varepsilon}(t)^{4}
\Big|dt\Big]^{\beta}\nonumber\\
&\le& C\me~\Big[\sup_{t\in[0,T]}|\delta x(t)^{5\beta}|dt\Big]
+C\me~\Big[\sup_{t\in[0,T]}|\delta x(t)^4-y_{1}^{\varepsilon}(t)^{4}|^{\beta}\Big]\nonumber\\
&\le& C\varepsilon^{\frac{5\beta}{2}}.
\end{eqnarray}
Similarly,
\begin{eqnarray}\label{lemma3 1equ7add}
& &\me\Big[\int_{0}^{T}\Big|\frac{1}{6}\int_{0}^{1}\theta^3
 \sigma_{xxxx}(t,\theta\bar{x}(t)+
(1-\theta)x^{\varepsilon}(t),\bar{u}(t))\delta x(t)^4 d\theta\nonumber\\
& &\qquad\qquad\qquad\qquad\qquad\qquad\qquad\qquad\qquad
-\frac{1}{24} \sigma_{xxxx}(t)y_{1}^{\varepsilon}(t)^{4}
\Big|^{2}dt\Big]^{\frac{\beta}{2}}\nonumber\\
&\le& C\varepsilon^{\frac{5\beta}{2}}.
\end{eqnarray}

Next, similar to (\ref{lemma3 1equ1}), we have
\begin{eqnarray}\label{lemma3 1equ8}
& &\me~\Big[\int_{0}^{T}\Big|\int_{0}^{1}\theta^2 \Big( b_{xxx}(t,\theta\bar{x}(t)
+(1-\theta)x^{\varepsilon}(t),u^{\varepsilon}(t))\nonumber\\
& &\qquad\qquad\qquad
-b_{xxx}(t,\theta\bar{x}(t)
+(1-\theta)x^{\varepsilon}(t),\bar{u}(t))\Big)\delta x(t)^3 d\theta\Big|dt\Big]^{\beta}\nonumber\\
&\le& C\me~\Big[\int_{0}^{T}\big|\delta x(t)^3 \chi_{E_{\varepsilon}}(t)\big|dt\Big]^{\beta}
\le C\me~ \Big[\sup_{t\in[0,T]}|\delta x(t)|^{3\beta}\Big]\varepsilon^{\beta}
\le C\varepsilon^{\frac{5\beta}{2}}.
\end{eqnarray}
and
\begin{eqnarray}\label{lemma3 1equ9}
& &\me~\Big[\int_{0}^{T}\Big|\int_{0}^{1}\theta^3 \Big( \sigma_{xxxx}(t,\theta\bar{x}(t)
+(1-\theta)x^{\varepsilon}(t),u^{\varepsilon}(t))\nonumber\\
& &\qquad\qquad\qquad
-\sigma_{xxxx}(t,\theta\bar{x}(t)
+(1-\theta)x^{\varepsilon}(t),\bar{u}(t))\Big)
\delta x(t)^4 d\theta\Big|^{2}dt\Big]^{\frac{\beta}{2}}\nonumber\\
&\le& C\me~\Big[\int_{0}^{T}\big|\delta x(t)^4 \chi_{E_{\varepsilon}}(t)\big|^{2}dt\Big]^{\frac{\beta}{2}}
\le C\me~ \Big[\sup_{t\in[0,T]}|\delta x(t)|^{4\beta}\Big]\varepsilon^{\frac{\beta}{2}}
\le C\varepsilon^{\frac{5\beta}{2}}.
\end{eqnarray}

Finally, by (\ref{estimateofx}) and (\ref{step3eta})--(\ref{lemma3 1equ9}), we obtain that
\begin{eqnarray*}
& &\me~\Big[\sup_{t\in[0,T]}|r_{4}(t)|^{\beta}\Big]\\
&\le &C\me~\Big[\int_{0}^{T}\Big|\frac{1}{2}b_{xx}(t)\big(\delta x(t)^2-\eta(t)^{2}\big)
+\frac{1}{6}b_{xxx}(t)\big(\delta x(t)^3-\gamma(t)^{3}\big)\\
& &+\delta b_x(t)r_{2}(t)\chi_{E_{\varepsilon}}(t)
+\frac{1}{2}\delta b_{xx}(t)\big(\delta x(t)^{2}
-y_{1}^{\varepsilon}(t)^{2}\big)\chi_{E_{\varepsilon}}(t)\\
& &+\frac{1}{6}\int_{0}^{1}\theta^3 b_{xxxx}(t,\theta\bar{x}(t)+
(1-\theta)x^{\varepsilon}(t),\bar{u}(t))\delta x(t)^4d\theta
-\frac{1}{24}b_{xxxx}(t)y_{1}^{\varepsilon}(t)^{4}\\
& &+\frac{1}{2}\int_{0}^{1}\theta^{2} \Big( b_{xxx}(t,\theta\bar{x}(t)+
(1-\theta)x^{\varepsilon}(t),u^{\varepsilon}(t))\\
& &\qquad\qquad\qquad
-b_{xxx}(t,\theta\bar{x}(t)
+(1-\theta)x^{\varepsilon}(t),\bar{u}(t))\Big)\delta x(t)^3 d\theta\Big|dt\Big]^{\beta}\\
& &+C\me~\Big[\int_{0}^{T}\Big|\frac{1}{2}\sigma_{xx}(t)\big(\delta x(t)^2
-\eta(t)^{2}\big)
+\frac{1}{6}\sigma_{xxx}(t)\big(\delta x(t)^3-\gamma(t)^{3}\big)\\
& &
+\delta \sigma_x(t)r_{3}(t)\chi_{E_{\varepsilon}}(t)
+\frac{1}{2}\delta\sigma_{xx}(t)\big(\delta x(t)^2
-\gamma(t)^{2}\big)\chi_{E_{\varepsilon}}(t)\\
& &+\frac{1}{6}\delta \sigma_{xxx}(t)\big(\delta x(t)^3
-y_{1}^{\varepsilon}(t)^{3}\big)\chi_{E_{\varepsilon}}(t)\\
& &+\frac{1}{6}\int_{0}^{1}\theta^3 \sigma_{xxxx}(t,\theta\bar{x}(t)+
(1-\theta)x^{\varepsilon}(t),\bar{u}(t))\delta x(t)^4d\theta
-\frac{1}{24}\sigma_{xxxx}(t)y_{1}^{\varepsilon}(t)^{4}\\
& &+\frac{1}{6}\int_{0}^{1}\theta^3 \Big( \sigma_{xxxx}(t,\theta\bar{x}(t)+
(1-\theta)x^{\varepsilon}(t),u^{\varepsilon}(t))\\
& &\qquad\qquad\qquad
-\sigma_{xxxx}(t,\theta\bar{x}(t)
+(1-\theta)x^{\varepsilon}(t),\bar{u}(t))\Big)\delta x(t)^4 d\theta\Big|^{2}dt\Big]^{\frac{\beta}{2}}\\
&\le&C\Big\{\me~\Big[\sup_{t\in[0,T]}|\delta x(t)^2-\eta(t)^{2}|^{\beta}\Big]
+\me~\Big[\sup_{t\in[0,T]}|\delta x(t)^3-\gamma(t)^{3}|^{\beta}\Big]\\
& &
+\me~\Big[\sup_{t\in[0,T]}|r_{2}(t)|^{\beta}\Big]\varepsilon^{\beta}
+\me~\Big[\sup_{t\in[0,T]}|\delta x(t)^2
-y_{1}^{\varepsilon}(t)^{2}|^{\beta}\Big]\varepsilon^{\beta}\\
& &
+\varepsilon^{\frac{5\beta}{2}}+\varepsilon^{\frac{5\beta}{2}}
+\me~\Big[\sup_{t\in[0,T]}|\delta x(t)^2-\eta(t)^{2}|^{\beta}\Big]
+\me~\Big[\sup_{t\in[0,T]}|\delta x(t)^3-\gamma(t)^{3}|^{\beta}\Big]\\
& &
+\me~\Big[\sup_{t\in[0,T]}|r_{3}(t)|^{\beta}\Big]\varepsilon^{\frac{\beta}{2}}
+\me~\Big[\sup_{t\in[0,T]}|\delta x(t)^2-\gamma(t)^{2}|^{\beta}\Big]\varepsilon^{\frac{\beta}{2}}\\
& &
+\me~\Big[\sup_{t\in[0,T]}|\delta x(t)^3
-y_{1}^{\varepsilon}(t)^{3}|^{\beta}\Big]\varepsilon^{\frac{\beta}{2}}
+\varepsilon^{\frac{5\beta}{2}}+\varepsilon^{\frac{5\beta}{2}}\Big\}\\
&\le &C \varepsilon^{\frac{5\beta}{2}}.
\end{eqnarray*}
This completes the proof of Lemma \ref{estimateofvariequ}.

\section{Proof of Proposition \ref{variational formulation for noncov}}
First, by (\ref{firstvariequ})--(\ref{forthvariequ}), we have
\begin{eqnarray}\label{xi(t)}
\ \xi(t)&=&\int_{0}^{t}\Big[b_{x}(s)\xi(s)+\frac{1}{2}b_{xx}(s)\big(\eta(s),\eta(s)\big)
+\frac{1}{6}b_{xxx}(s)\big(\gamma(s),\gamma(s),\gamma(s)\big)\nonumber\\
& &+\frac{1}{24}b_{xxxx}(s)\big(y_{1}^{\varepsilon}(s),y_{1}^{\varepsilon}(s),
y_{1}^{\varepsilon}(s),y_{1}^{\varepsilon}(s)\big)
+\delta b(s)\chi_{E_{\varepsilon}}(s)
\nonumber\\
& &+\delta b_{x}(s)\gamma(s)\chi_{E_{\varepsilon}}(s)
+\frac{1}{2}\delta b_{xx}(s)\big(y_{1}^{\varepsilon}(s),y_{1}^{\varepsilon}(s)\big)
\chi_{E_{\varepsilon}}(s)\Big]ds\nonumber\\
& &+\int_{0}^{t}\Big[\sigma_{x}(s)\xi(s)+\frac{1}{2}\sigma_{xx}(s)\big(\eta(s),\eta(s)\big)
+\frac{1}{6}\sigma_{xxx}(s)\big(\gamma(s),\gamma(s),\gamma(s)\big)\nonumber\\
& &+\frac{1}{24}\sigma_{xxxx}(s)\big(y_{1}^{\varepsilon}(s),y_{1}^{\varepsilon}(s),
y_{1}^{\varepsilon}(s),y_{1}^{\varepsilon}(s)\big)
+\delta \sigma(s)\chi_{E_{\varepsilon}}(s)
+\delta \sigma_{x}(s)\eta(s)\chi_{E_{\varepsilon}}(s)\nonumber\\
& &
+\frac{1}{2}\delta \sigma_{xx}(s)\big( \gamma(s), \gamma(s)\big)\chi_{E_{\varepsilon}}(s)
+\frac{1}{6}\delta \sigma_{xxx}(s)\big(y_{1}^{\varepsilon}(s),y_{1}^{\varepsilon}(s),
y_{1}^{\varepsilon}(s)\big)
\chi_{E_{\varepsilon}}(s)\Big]dW(s),
\end{eqnarray}
\begin{eqnarray}\label{eta(t)2}
\eta(t)&=&\int_{0}^{t}\Big[b_{x}(s)\eta(s)+\frac{1}{2}b_{xx}(s)\big( \gamma(s), \gamma(s)\big)
+\frac{1}{6}b_{xxx}(s)\big(y_{1}^{\varepsilon}(s),y_{1}^{\varepsilon}(s),
y_{1}^{\varepsilon}(s)\big)
\nonumber\\
& &+\delta b(s)\chi_{E_{\varepsilon}}(s)+
\delta b_{x}(s)y_{1}^{\varepsilon}(s)\chi_{E_{\varepsilon}}(s)
\Big]ds\nonumber\\
& &+\int_{0}^{t}\Big[\sigma_{x}(s)\eta(s)+\frac{1}{2}\sigma_{xx}(s)\big( \gamma(s), \gamma(s)\big)
+\frac{1}{6}\sigma_{xxx}(s)\big(y_{1}^{\varepsilon}(s),y_{1}^{\varepsilon}(s),
y_{1}^{\varepsilon}(s)\big)\nonumber\\
& &+\delta \sigma(s)\chi_{E_{\varepsilon}}(s)
+\delta \sigma_{x}(s)\gamma(s)\chi_{E_{\varepsilon}}(s)
+\frac{1}{2}\delta \sigma_{xx}(s)\big(y_{1}^{\varepsilon}(s),y_{1}^{\varepsilon}(s)\big)
\chi_{E_{\varepsilon}}(s)\Big]dW(s),
\end{eqnarray}
and
\begin{eqnarray}\label{gamma(t)3}
\gamma(t)&=&\int_{0}^{t}\Big[b_{x}(s)\gamma(s)
+\frac{1}{2}b_{xx}(s)\big(y_{1}^{\varepsilon}(s),y_{1}^{\varepsilon}(s)\big)
+\delta b(s)\chi_{E_{\varepsilon}}(s)\Big]ds\nonumber\\
& &+\int_{0}^{t}\Big[\sigma_{x}(s)\gamma(s)
+\frac{1}{2}\sigma_{xx}(s)\big(y_{1}^{\varepsilon}(s),y_{1}^{\varepsilon}(s)\big)
\qquad\qquad\qquad\qquad\qquad\qquad\qquad\ \ \nonumber\\
& &\qquad\ \ +\delta \sigma(s)\chi_{E_{\varepsilon}}(s)
+\delta \sigma_{x}(s)y_{1}^{\varepsilon}(s)\chi_{E_{\varepsilon}}(s)
\Big]dW(s).
\end{eqnarray}

Then, using the formula (\ref{Ito formulation}) in Lemma \ref{o formu}, we obtain that
\begin{eqnarray}\label{hx(t)xi(t)}
& &\me ~\Big\langle h_{x}(\bar{x}(T)),\xi(T)\Big\rangle
= -\me ~\Big\langle p_{1}(T),\xi(T)\Big\rangle\nonumber\\[+0.3em]
&=&-\me\int_{0}^{T}\Big[\frac{1}{2}\inner{p_{1}(t)}{b_{xx}(t)\big(\eta(t),\eta(t)\big)}
+\frac{1}{2}\inner{q_{1}(t)}{\sigma_{xx}(t)\big(\eta(t),\eta(t)\big)}
\nonumber\\
& &+\frac{1}{6}\inner{p_{1}(t)}{b_{xxx}(t)\big(\gamma(t),\gamma(t),\gamma(t)\big)}
+\frac{1}{6}\inner{q_{1}(t)}{\sigma_{xxx}(t)\big(\gamma(t),\gamma(t),\gamma(t)\big)}
\nonumber\\
& &+\frac{1}{24}\inner{p_{1}(t)}{b_{xxxx}(t)\big(y_{1}^{\varepsilon}(t)
,y_{1}^{\varepsilon}(t),y_{1}^{\varepsilon}(t),y_{1}^{\varepsilon}(t)\big)}
\nonumber\\
& &+\frac{1}{24}\inner{q_{1}(t)}{\sigma_{xxxx}(t)\big(y_{1}^{\varepsilon}(t)
,y_{1}^{\varepsilon}(t),y_{1}^{\varepsilon}(t),y_{1}^{\varepsilon}(t)\big)}
\nonumber\\
& &+\inner{f_{x}(t)}{\xi(t)}+\inner{p_{1}(t)}{\delta b(t)}\chi_{E_{\varepsilon}}(t)
+\inner{q_{1}(t)}{\delta \sigma(t)}\chi_{E_{\varepsilon}}(t)
\nonumber\\[+0.3em]
& &+\inner{p_{1}(t)}{\delta b_{x}(t)\gamma(t)}\chi_{E_{\varepsilon}}(t)
+\inner{q_{1}(t)}{\delta \sigma_{x}(t)\eta(t)}\chi_{E_{\varepsilon}}(t)
\nonumber\\[+0.4em]
& &+\frac{1}{2}\inner{p_{1}(t)}{\delta b_{xx}(t)\big(y_{1}^{\varepsilon}(t),y_{1}^{\varepsilon}(t)\big)}
\chi_{E_{\varepsilon}}(t)\nonumber\\
& &
+\frac{1}{2}\inner{q_{1}(t)}{\delta \sigma_{xx}(t)\big(y_{1}^{\varepsilon}(t),y_{1}^{\varepsilon}(t)\big)}
\chi_{E_{\varepsilon}}(t)\Big]dt
+o(\varepsilon^{2}),\quad (\ephs\to 0^+),
\end{eqnarray}
\begin{eqnarray}\label{hxx(t)eta(t)2}
& &\me ~\big\langle h_{xx}(\bar{x}(T))\eta(T),\eta(T)\big\rangle
= -\me ~\big\langle p_{2}(T)\eta(T),\eta(T)\big\rangle\nonumber\\
&=&-\me\int_{0}^{T}\Big[
\frac{1}{2}\inner{p_{2}(t)b_{xx}(t)\big(\gamma(t),\gamma(t)\big)}{\eta(t)}
+\frac{1}{2}\inner{p_{2}(t)\eta(t)}{b_{xx}(t)\big(\gamma(t),\gamma(t)\big)}
\nonumber\\
& &
+\frac{1}{2}\inner{q_{2}(t)\sigma_{xx}(t)\big(\gamma(t),\gamma(t)\big)}{\eta(t)}
+\frac{1}{2}\inner{q_{2}(t)\eta(t)}{\sigma_{xx}(t)\big(\gamma(t),\gamma(t)\big)}
\nonumber\\
& &
+\frac{1}{6}\inner{p_{2}(t)b_{xxx}(t)\big(y_{1}^{\varepsilon}(t),
y_{1}^{\varepsilon}(t),y_{1}^{\varepsilon}(t)\big)}{\eta(t)}\nonumber\\
& &
+\frac{1}{6}\inner{p_{2}(t)\eta(t)}{b_{xxx}(t)\big(y_{1}^{\varepsilon}(t),
y_{1}^{\varepsilon}(t),y_{1}^{\varepsilon}(t)\big)}
\nonumber\\
& &
+\frac{1}{6}\inner{q_{2}(t)\sigma_{xxx}(t)\big(y_{1}^{\varepsilon}(t),
y_{1}^{\varepsilon}(t),y_{1}^{\varepsilon}(t)\big)}{\eta(t)}\nonumber\\
& &
+\frac{1}{6}\inner{q_{2}(t)\eta(t)}{\sigma_{xxx}(t)\big(y_{1}^{\varepsilon}(t),
y_{1}^{\varepsilon}(t),y_{1}^{\varepsilon}(t)\big)}
\nonumber\\
& &+\inner{p_{2}(t)\delta b(t)}{\eta(t)}\chi_{E_{\varepsilon}}(t)
+\inner{p_{2}(t)\eta(t) }{\delta b(t)}\chi_{E_{\varepsilon}}(t)
\nonumber\\[+0.3em]
& &+\inner{q_{2}(t)\delta \sigma(t)}{\eta(t) }\chi_{E_{\varepsilon}}(t)
+\inner{q_{2}(t)\eta(t)}{\delta \sigma(t)}\chi_{E_{\varepsilon}}(t)
\nonumber\\[+0.3em]
& &
+\inner{p_{2}(t)\delta b_{x}(t)y_{1}^{\varepsilon}(t)}
{\eta(t)}\chi_{E_{\varepsilon}}(t)
+\inner{p_{2}(t)\eta(t) }{\delta b_{x}(t)y_{1}^{\varepsilon}(t)}\chi_{E_{\varepsilon}}(t)
\nonumber\\[+0.3em]
& &
+\inner{q_{2}(t)\delta \sigma_{x}(t)\gamma(t)}
{\eta(t)}\chi_{E_{\varepsilon}}(t)
+\inner{q_{2}(t)\eta(t)}{\delta \sigma_{x}(t)\gamma(t)}\chi_{E_{\varepsilon}}(t)
\nonumber\\[+0.3em]
& &
+\big\langle p_{2}(t) \sigma_{x}(t)\eta(t), \frac{1}{2}\sigma_{xx}(t)\big(\gamma(t),\gamma(t)\big)
+\frac{1}{6}\sigma_{xxx}(t)\big(y_{1}^{\varepsilon}(t),
y_{1}^{\varepsilon}(t),y_{1}^{\varepsilon}(t)\big)
\big\rangle
\nonumber\\[+0.3em]
& &
+\big\langle p_{2}(t)\big[\frac{1}{2}\sigma_{xx}(t)\big(\gamma(t),\gamma(t)\big)
+\frac{1}{6}\sigma_{xxx}(t)\big(y_{1}^{\varepsilon}(t),
y_{1}^{\varepsilon}(t),y_{1}^{\varepsilon}(t)\big)\big],
 \sigma_{x}(t)\eta(t)\big\rangle
\nonumber\\[+0.3em]
& &
+\big\langle p_{2}(t) \sigma_{x}(t)\eta(t), \delta \sigma(t)+ \delta \sigma_{x}(t)\gamma(t) \big\rangle\chi_{E_{\varepsilon}}(t)
\nonumber\\[+0.3em]
& &
+\big\langle p_{2}(t) \big( \delta \sigma(t)+ \delta \sigma_{x}(t)\gamma(t)\big) ,\sigma_{x}(t)\eta(t)\big\rangle\chi_{E_{\varepsilon}}(t)
\nonumber\\[+0.3em]
& &
+\frac{1}{4}\big\langle p_{2}(t)\sigma_{xx}(t)\big(\gamma(t),\gamma(t)\big)
, \sigma_{xx}(t)\big(\gamma(t),\gamma(t)\big)
\big\rangle
\nonumber\\
& &
+\frac{1}{2}\big\langle p_{2}(t)\sigma_{xx}(t)\big(\gamma(t),\gamma(t)\big),
\delta \sigma(t)\big\rangle\chi_{E_{\varepsilon}}(t)
\nonumber\\
& &
+\frac{1}{2}\big\langle p_{2}(t) \delta \sigma(t),
\sigma_{xx}(t)\big(\gamma(t),\gamma(t)\big)
\big\rangle\chi_{E_{\varepsilon}}(t)
\nonumber\\
& &
+\big\langle p_{2}(t) \big(\delta \sigma(t)+ \delta \sigma_{x}(t)\gamma(t)\big),
\delta \sigma(t)+ \delta \sigma_{x}(t)\gamma(t)
\big\rangle\chi_{E_{\varepsilon}}(t)
\nonumber\\
& &
+\frac{1}{2}\big\langle p_{2}(t)\delta \sigma_{xx}(t)\big(y_{1}^{\varepsilon}(t),
y_{1}^{\varepsilon}(t)\big),
\delta \sigma(t)\big\rangle\chi_{E_{\varepsilon}}(t)
\nonumber\\
& &
+\frac{1}{2}\big\langle p_{2}(t)\delta \sigma(t),
\delta \sigma_{xx}(t)\big(y_{1}^{\varepsilon}(t),
y_{1}^{\varepsilon}(t)\big)\big\rangle\chi_{E_{\varepsilon}}(t)
\nonumber\\
& &-\inner{\hh_{xx}(t)\eta(t)}{\eta(t)}\Big]dt
+o(\varepsilon^{2})\quad (\ephs\to 0^+),
\end{eqnarray}
\begin{eqnarray}\label{hxxx(t)gamma(t)3}
& &\me ~\Big[h_{xxx}(\bar{x}(T))\big(\gamma(T),\gamma(T),\gamma(T)\big)\Big]
= -\me ~\Big[p_{3}(T)\big(\gamma(T),\gamma(T),\gamma(T)\big)\Big]\nonumber\\[+0.3em]
&=&-\me\int_{0}^{T}\Big[
p_{3}(t)\Big(\frac{1}{2}b_{xx}(t)\big(y_{1}^{\varepsilon}(t),y_{1}^{\varepsilon}(t) \big) +\delta b(t)\chi_{E_{\varepsilon}}(t),\gamma(t),\gamma(t)\Big)\nonumber\\[+0.3em]
& &
+p_{3}(t)\Big(\gamma(t),\frac{1}{2}b_{xx}(t)
\big(y_{1}^{\varepsilon}(t),y_{1}^{\varepsilon}(t) \big) +\delta b(t)\chi_{E_{\varepsilon}}(t),\gamma(t)\Big)
\nonumber\\[+0.3em]
& &
+p_{3}(t)\Big(\gamma(t),\gamma(t),\frac{1}{2}b_{xx}(t)
\big(y_{1}^{\varepsilon}(t),y_{1}^{\varepsilon}(t) \big) +\delta b(t)\chi_{E_{\varepsilon}}(t)\Big)
\nonumber\\[+0.3em]
& &
-\frac{3}{2}\inner{p_{2}(t)b_{xx}(t)\big(\gamma(t),\gamma(t)\big)}{\gamma(t)}
-\frac{3}{2}\inner{p_{2}(t)\gamma(t)}{b_{xx}(t)\big(\gamma(t),\gamma(t)\big)}
\nonumber\\
& &
-\frac{3}{2}\inner{q_{2}(t)\sigma_{xx}(t)\big(\gamma(t),\gamma(t)\big)}{\gamma(t)}
-\frac{3}{2}\inner{q_{2}(t)\gamma(t)}{\sigma_{xx}(t)\big(\gamma(t),\gamma(t)\big)}
\nonumber\\
& &
-\frac{3}{2}\inner{p_{2}(t)\sigma_{xx}(t)
\big(\gamma(t),\gamma(t)\big)}{\sigma_{x}(t)\gamma(t)}
-\frac{3}{2}\inner{p_{2}(t)\sigma_{x}(t)\gamma(t)}{\sigma_{xx}(t)\big(\gamma(t),\gamma(t)\big)}
\nonumber\\
& &
+q_{3}(t)\Big(\frac{1}{2}\sigma_{xx}(t)
\big(y_{1}^{\varepsilon}(t),y_{1}^{\varepsilon}(t) \big) +\delta \sigma(t)\chi_{E_{\varepsilon}}(t),\gamma(t),\gamma(t)\Big)
\nonumber\\[+0.3em]
& &
+q_{3}(t)\Big(\gamma(t),\frac{1}{2}\sigma_{xx}(t)
\big(y_{1}^{\varepsilon}(t),y_{1}^{\varepsilon}(t) \big) +\delta \sigma(t)\chi_{E_{\varepsilon}}(t),\gamma(t)\Big)
\nonumber\\[+0.3em]
& &
+q_{3}(t)\Big(\gamma(t),\gamma(t),\frac{1}{2}\sigma_{xx}(t)
\big(y_{1}^{\varepsilon}(t),y_{1}^{\varepsilon}(t) \big) +\delta \sigma(t)\chi_{E_{\varepsilon}}(t)\Big)
\nonumber\\[+0.3em]
& &
+\frac{1}{2}p_{3}(t)\Big(\sigma_{x}(t)\gamma(t),\sigma_{xx}(t)
\big(y_{1}^{\varepsilon}(t),y_{1}^{\varepsilon}(t) \big),\gamma(t)\Big)
\nonumber\\
& &
+\frac{1}{2}p_{3}(t)\Big(\sigma_{x}(t)\gamma(t),\gamma(t),\sigma_{xx}(t)
\big(y_{1}^{\varepsilon}(t),y_{1}^{\varepsilon}(t) \big)\Big)
\nonumber\\
& &
+\frac{1}{2}p_{3}(t)\Big(\gamma(t),\sigma_{x}(t)\gamma(t),\sigma_{xx}(t)
\big(y_{1}^{\varepsilon}(t),y_{1}^{\varepsilon}(t) \big)\Big)
\nonumber\\
& &
+\frac{1}{2}p_{3}(t)\Big(\sigma_{xx}(t)
\big(y_{1}^{\varepsilon}(t),y_{1}^{\varepsilon}(t) \big),\sigma_{x}(t)\gamma(t),\gamma(t)\Big)
\nonumber\\
& &
+\frac{1}{2}p_{3}(t)\Big(\sigma_{xx}(t)
\big(y_{1}^{\varepsilon}(t),y_{1}^{\varepsilon}(t) \big),\gamma(t),\sigma_{x}(t)\gamma(t)\Big)
\nonumber\\
& &
+\frac{1}{2}p_{3}(t)\Big(\gamma(t),\sigma_{xx}(t)
\big(y_{1}^{\varepsilon}(t),y_{1}^{\varepsilon}(t) \big),\sigma_{x}(t)\gamma(t)\Big)
\nonumber\\
& &
+p_{3}(t)\Big(\sigma_{x}(t)\gamma(t),\delta \sigma(t),\gamma(t)\Big)\chi_{E_{\varepsilon}}(t)
+p_{3}(t)\Big(\sigma_{x}(t)\gamma(t),\gamma(t),\delta \sigma(t)\Big)\chi_{E_{\varepsilon}}(t)
\nonumber\\[+0.3em]
& &
+p_{3}(t)\Big(\gamma(t),\sigma_{x}(t)\gamma(t),\delta \sigma(t)\Big)\chi_{E_{\varepsilon}}(t)
+p_{3}(t)\Big(\delta\sigma(t),\sigma_{x}(t)\gamma(t),
\gamma(t)\Big)\chi_{E_{\varepsilon}}(t)
\nonumber\\[+0.3em]
& &
+p_{3}(t)\Big(\delta\sigma(t),\gamma(t),
\sigma_{x}(t)\gamma(t)\Big)\chi_{E_{\varepsilon}}(t)
+p_{3}(t)\Big(\gamma(t),\delta \sigma(t),
\sigma_{x}(t)\gamma(t)\Big)\chi_{E_{\varepsilon}}(t)
\nonumber\\[+0.3em]
& &
+p_{3}(t)\Big(\delta \sigma(t),
\delta \sigma(t),\gamma(t)\Big)\chi_{E_{\varepsilon}}(t)
+p_{3}(t)\Big(\delta \sigma(t),\gamma(t),\delta \sigma(t)\Big)\chi_{E_{\varepsilon}}(t)
\nonumber\\[+0.3em]
& &
+p_{3}(t)\Big(\gamma(t),\delta \sigma(t),
\delta \sigma(t)\Big)\chi_{E_{\varepsilon}}(t)
\nonumber\\[+0.3em]
& &
+p_{3}(t)\Big(\delta \sigma(t),
\delta \sigma_{x}(t)y_{1}^{\varepsilon}(t),\gamma(t)\Big)\chi_{E_{\varepsilon}}(t)
+p_{3}(t)\Big(\delta \sigma(t),
\gamma(t),\delta \sigma_{x}(t)y_{1}^{\varepsilon}(t)\Big)\chi_{E_{\varepsilon}}(t)
\nonumber\\[+0.3em]
& &
+p_{3}(t)\Big(\gamma(t),\delta \sigma(t),
\delta \sigma_{x}(t)y_{1}^{\varepsilon}(t)\Big)\chi_{E_{\varepsilon}}(t)
+p_{3}(t)\Big(\delta \sigma_{x}(t)y_{1}^{\varepsilon}(t),\delta \sigma(t),
\gamma(t)\Big)\chi_{E_{\varepsilon}}(t)
\nonumber\\[+0.3em]
& &
+p_{3}(t)\Big(\delta \sigma_{x}(t)y_{1}^{\varepsilon}(t),
\gamma(t),\delta \sigma(t)\Big)\chi_{E_{\varepsilon}}(t)
+p_{3}(t)\Big(\gamma(t),\delta \sigma_{x}(t)y_{1}^{\varepsilon}(t),\delta \sigma(t)\Big)\chi_{E_{\varepsilon}}(t)
\nonumber\\[+0.3em]
%
& &-\hh_{xxx}(t)\big(\gamma(t),\gamma(t),\gamma(t)\big)
\Big]dt+o(\varepsilon^{2}),\quad (\ephs\to 0^+),
\end{eqnarray}
and
\begin{eqnarray}\label{hxxxx(t)y1(t)4}
& &\me ~\Big[h_{xxxx}(\bar{x}(T))\big(y_{1}^{\varepsilon}(T),
y_{1}^{\varepsilon}(T),y_{1}^{\varepsilon}(T),y_{1}^{\varepsilon}(T)\big)\Big]
\nonumber\\[+0.3em]
&=& -\me ~\Big[p_{4}(T)\big(y_{1}^{\varepsilon}(T),
y_{1}^{\varepsilon}(T),y_{1}^{\varepsilon}(T),y_{1}^{\varepsilon}(T)\big)\Big]
\nonumber\\[+0.3em]
&=&-\me\int_{0}^{T}\Big[
p_{4}(t)\big(\delta\sigma(t),\delta\sigma(t),
y_{1}^{\varepsilon}(t),y_{1}^{\varepsilon}(t)\big)\chi_{E_{\varepsilon}}(t)
\nonumber\\[+0.3em]
& &+p_{4}(t)\big(\delta\sigma(t),y_{1}^{\varepsilon}(t),\delta\sigma(t),
y_{1}^{\varepsilon}(t)\big)\chi_{E_{\varepsilon}}(t)
\nonumber\\[+0.3em]
& &+p_{4}(t)\big(\delta\sigma(t),y_{1}^{\varepsilon}(t),
y_{1}^{\varepsilon}(t),\delta\sigma(t)
\big)\chi_{E_{\varepsilon}}(t)
\nonumber\\[+0.3em]
& &+p_{4}(t)\big(y_{1}^{\varepsilon}(t),\delta\sigma(t),\delta\sigma(t),
y_{1}^{\varepsilon}(t)\big)\chi_{E_{\varepsilon}}(t)
\nonumber\\[+0.3em]
& &+p_{4}(t)\big(y_{1}^{\varepsilon}(t),\delta\sigma(t),
y_{1}^{\varepsilon}(t),\delta\sigma(t)
\big)\chi_{E_{\varepsilon}}(t)\nonumber\\[+0.3em]
& &+p_{4}(t)\big(y_{1}^{\varepsilon}(t),y_{1}^{\varepsilon}(t),\delta\sigma(t),
\delta\sigma(t)
\big)\chi_{E_{\varepsilon}}(t)\nonumber\\[+0.3em]
& &-2p_{3}(t)\Big(b_{xx}(t)\big(y_{1}^{\varepsilon}(t),
y_{1}^{\varepsilon}(t)\big),y_{1}^{\varepsilon}(t),y_{1}^{\varepsilon}(t)\Big)
\nonumber\\
& &-2p_{3}(t)\Big(y_{1}^{\varepsilon}(t),
b_{xx}(t)\big(y_{1}^{\varepsilon}(t),
y_{1}^{\varepsilon}(t)\big),y_{1}^{\varepsilon}(t)\Big)
\nonumber\\
& &-2p_{3}(t)\Big(y_{1}^{\varepsilon}(t),y_{1}^{\varepsilon}(t),
b_{xx}(t)\big(y_{1}^{\varepsilon}(t),
y_{1}^{\varepsilon}(t)\big)\Big)
\nonumber\\
& &-2q_{3}(t)\Big(\sigma_{xx}(t)\big(y_{1}^{\varepsilon}(t),
y_{1}^{\varepsilon}(t)\big),y_{1}^{\varepsilon}(t),y_{1}^{\varepsilon}(t)\Big)
\nonumber\\
& &-2q_{3}(t)\Big(y_{1}^{\varepsilon}(t),
\sigma_{xx}(t)\big(y_{1}^{\varepsilon}(t),
y_{1}^{\varepsilon}(t)\big),y_{1}^{\varepsilon}(t)\Big)
\nonumber\\
& &-2q_{3}(t)\Big(y_{1}^{\varepsilon}(t),y_{1}^{\varepsilon}(t),
\sigma_{xx}(t)\big(y_{1}^{\varepsilon}(t),
y_{1}^{\varepsilon}(t)\big)\Big)
\nonumber\\
& &-2p_{3}(t)\Big(\sigma_{x}(t)y_{1}^{\varepsilon}(t),
\sigma_{xx}(t)\big(y_{1}^{\varepsilon}(t),
y_{1}^{\varepsilon}(t)\big),y_{1}^{\varepsilon}(t)\Big)\nonumber\\[+0.3em]
& &-2p_{3}(t)\Big(\sigma_{x}(t)y_{1}^{\varepsilon}(t),y_{1}^{\varepsilon}(t),
\sigma_{xx}(t)\big(y_{1}^{\varepsilon}(t),
y_{1}^{\varepsilon}(t)\big)\Big)
\nonumber\\
& &-2p_{3}(t)\Big(y_{1}^{\varepsilon}(t),\sigma_{x}(t)y_{1}^{\varepsilon}(t),
\sigma_{xx}(t)\big(y_{1}^{\varepsilon}(t),
y_{1}^{\varepsilon}(t)\big)\Big)\nonumber\\
& &-2p_{3}(t)\Big(\sigma_{xx}(t)\big(y_{1}^{\varepsilon}(t),
y_{1}^{\varepsilon}(t)\big),\sigma_{x}(t)y_{1}^{\varepsilon}(t),
y_{1}^{\varepsilon}(t)\Big)
\nonumber\\
& &-2p_{3}(t)\Big(\sigma_{xx}(t)\big(y_{1}^{\varepsilon}(t),
y_{1}^{\varepsilon}(t)\big),y_{1}^{\varepsilon}(t),
\sigma_{x}(t)y_{1}^{\varepsilon}(t)\Big)\nonumber\\
& &-2p_{3}(t)\Big(y_{1}^{\varepsilon}(t),
\sigma_{xx}(t)\big(y_{1}^{\varepsilon}(t),y_{1}^{\varepsilon}(t)\big),
\sigma_{x}(t)y_{1}^{\varepsilon}(t)\Big)
\nonumber\\
& &-2\inner{p_{2}(t)b_{xxx}(t)\big(y_{1}^{\varepsilon}(t),
y_{1}^{\varepsilon}(t),y_{1}^{\varepsilon}(t)\big)}
{y_{1}^{\varepsilon}(t)}\nonumber\\[+0.3em]
& &-2\inner{p_{2}(t)y_{1}^{\varepsilon}(t)}
{b_{xxx}(t)\big(y_{1}^{\varepsilon}(t),
y_{1}^{\varepsilon}(t),y_{1}^{\varepsilon}(t)\big)}\nonumber\\[+0.3em]
& &-2\inner{p_{2}(t)\sigma_{x}(t)y_{1}^{\varepsilon}(t)}
{\sigma_{xxx}(t)\big(y_{1}^{\varepsilon}(t),
y_{1}^{\varepsilon}(t),y_{1}^{\varepsilon}(t)\big)}\nonumber\\[+0.3em]
& &-2\inner{p_{2}(t)\sigma_{xxx}(t)\big(y_{1}^{\varepsilon}(t),
y_{1}^{\varepsilon}(t),y_{1}^{\varepsilon}(t)\big)}
{\sigma_{x}(t)y_{1}^{\varepsilon}(t)}\nonumber\\[+0.3em]
& &-2\inner{q_{2}(t)y_{1}^{\varepsilon}(t)}
{\sigma_{xxx}(t)\big(y_{1}^{\varepsilon}(t),
y_{1}^{\varepsilon}(t),y_{1}^{\varepsilon}(t)\big)}\nonumber\\[+0.3em]
& &-2\inner{q_{2}(t)\sigma_{xxx}(t)\big(y_{1}^{\varepsilon}(t),
y_{1}^{\varepsilon}(t),y_{1}^{\varepsilon}(t)\big)}
{y_{1}^{\varepsilon}(t)}\nonumber\\[+0.3em]
& &-3\inner{p_{2}(t)\sigma_{xx}(t)\big(y_{1}^{\varepsilon}(t),
y_{1}^{\varepsilon}(t)\big)}{\sigma_{xx}(t)\big(y_{1}^{\varepsilon}(t),
y_{1}^{\varepsilon}(t)\big)}\nonumber\\[+0.3em]
& &-\hh_{xxxx}(t)\big(y_{1}^{\varepsilon}(t),
y_{1}^{\varepsilon}(t),y_{1}^{\varepsilon}(t),
y_{1}^{\varepsilon}(t)\big)
\Big]dt+o(\varepsilon^{2}),\qquad (\ephs\to 0^+).
\end{eqnarray}

Substituting (\ref{hx(t)xi(t)})--(\ref{hxxxx(t)y1(t)4}) into the Taylor expansion (\ref{taylorexp}), we obtain that
\begin{eqnarray*}
& & J(u^{\varepsilon})-J(\bar{u})\nonumber\\[+0.3em]
&=&-\me\int_{0}^{T}\Big[
-\delta f(t)\chi_{E_{\varepsilon}}(t)
-\delta f_x(t)\gamma(t)\chi_{E_{\varepsilon}}(t)
-\frac{1}{2}\delta f_{xx}(t)\big(y_{1}^{\varepsilon}(t),y_{1}^{\varepsilon}(t)\big)
\chi_{E_{\varepsilon}}(t)\nonumber\\
%
%
& &+\inner{p_{1}(t)}{\delta b(t)}\chi_{E_{\varepsilon}}(t)
+\inner{q_{1}(t)}{\delta \sigma(t)}\chi_{E_{\varepsilon}}(t)
\nonumber\\[+0.3em]
& &+\inner{p_{1}(t)}{\delta b_{x}(t)\gamma(t)}\chi_{E_{\varepsilon}}(t)
+\inner{q_{1}(t)}{\delta \sigma_{x}(t)\gamma(t)}\chi_{E_{\varepsilon}}(t)
\nonumber\\[+0.3em]
& &+\frac{1}{2}\inner{p_{1}(t)}{\delta b_{xx}(t)\big(y_{1}^{\varepsilon}(t),y_{1}^{\varepsilon}(t)\big)}
\chi_{E_{\varepsilon}}(t)\nonumber\\
& &
+\frac{1}{2}\inner{q_{1}(t)}{\delta \sigma_{xx}(t)\big(y_{1}^{\varepsilon}(t),y_{1}^{\varepsilon}(t)\big)}
\chi_{E_{\varepsilon}}(t)\nonumber\\
%
& &+\frac{1}{2}\inner{p_{2}(t)\delta b(t)}{\gamma(t)}\chi_{E_{\varepsilon}}(t)
+\frac{1}{2}\inner{p_{2}(t)\gamma(t) }{\delta b(t)}\chi_{E_{\varepsilon}}(t)
\nonumber\\
& &+\frac{1}{2}\inner{q_{2}(t)\delta \sigma(t)}{\gamma(t)}\chi_{E_{\varepsilon}}(t)
+\frac{1}{2}\inner{q_{2}(t)\gamma(t)}{\delta \sigma(t)}\chi_{E_{\varepsilon}}(t)
\nonumber\\
& &
+\frac{1}{2}\inner{p_{2}(t)\delta b_{x}(t)y_{1}^{\varepsilon}(t)}
{y_{1}^{\varepsilon}(t)}\chi_{E_{\varepsilon}}(t)
+\frac{1}{2}\inner{p_{2}(t)y_{1}^{\varepsilon}(t)}{\delta b_{x}(t)y_{1}^{\varepsilon}(t)}\chi_{E_{\varepsilon}}(t)
\nonumber\\
& &
+\frac{1}{2}\inner{q_{2}(t)\delta \sigma_{x}(t)y_{1}^{\varepsilon}(t)}
{y_{1}^{\varepsilon}(t)}\chi_{E_{\varepsilon}}(t)
+\frac{1}{2}\inner{q_{2}(t)y_{1}^{\varepsilon}(t)}{\delta \sigma_{x}(t)y_{1}^{\varepsilon}(t)}\chi_{E_{\varepsilon}}(t)
\nonumber\\
& &
+\frac{1}{2}\big\langle p_{2}(t) \sigma_{x}(t)\gamma(t), \delta \sigma(t) \big\rangle\chi_{E_{\varepsilon}}(t)
+\frac{1}{2}\big\langle p_{2}(t)\delta \sigma(t),\sigma_{x}(t)\gamma(t)\big\rangle\chi_{E_{\varepsilon}}(t)
\nonumber\\
& &
+\frac{1}{2}\big\langle p_{2}(t) \sigma_{x}(t)y_{1}^{\varepsilon}(t), \delta \sigma_{x}(t)y_{1}^{\varepsilon}(t)\big\rangle\chi_{E_{\varepsilon}}(t)
\nonumber\\
& &
+\frac{1}{2}\big\langle p_{2}(t)\delta \sigma_{x}(t)y_{1}^{\varepsilon}(t) ,\sigma_{x}(t)y_{1}^{\varepsilon}(t)\big\rangle\chi_{E_{\varepsilon}}(t)
\nonumber\\
& &
+\frac{1}{4}\big\langle p_{2}(t)\sigma_{xx}(t)\big(y_{1}^{\varepsilon}(t),y_{1}^{\varepsilon}(t)\big),
\delta \sigma(t)\big\rangle\chi_{E_{\varepsilon}}(t)
\nonumber\\
& &
+\frac{1}{4}\big\langle p_{2}(t) \delta \sigma(t),
\sigma_{xx}(t)\big(y_{1}^{\varepsilon}(t),y_{1}^{\varepsilon}(t)\big)
\big\rangle\chi_{E_{\varepsilon}}(t)
\nonumber\\
& &
+\frac{1}{2}\big\langle p_{2}(t)\delta \sigma(t),
\delta \sigma(t)\big\rangle\chi_{E_{\varepsilon}}(t)
+\frac{1}{2}\big\langle p_{2}(t) \delta \sigma(t),
\delta \sigma_{x}(t)\gamma(t)
\big\rangle\chi_{E_{\varepsilon}}(t)\nonumber\\
& &
+\frac{1}{2}\big\langle p_{2}(t) \delta \sigma_{x}(t)\gamma(t),
\delta \sigma(t)\big\rangle\chi_{E_{\varepsilon}}(t)
+\frac{1}{2}\big\langle p_{2}(t) \delta \sigma_{x}(t)y_{1}^{\varepsilon}(t),
\delta \sigma_{x}(t)y_{1}^{\varepsilon}(t)
\big\rangle\chi_{E_{\varepsilon}}(t)
\nonumber\\
& &
+\frac{1}{4}\big\langle p_{2}(t)\delta \sigma_{xx}(t)\big(y_{1}^{\varepsilon}(t),
y_{1}^{\varepsilon}(t)\big),
\delta \sigma(t)\big\rangle\chi_{E_{\varepsilon}}(t)
\nonumber\\
& &
+\frac{1}{4}\big\langle p_{2}(t)\delta \sigma(t),
\delta \sigma_{xx}(t)\big(y_{1}^{\varepsilon}(t),
y_{1}^{\varepsilon}(t)\big)\big\rangle\chi_{E_{\varepsilon}}(t)\nonumber\\
%
%
& &
+\frac{1}{6}p_{3}(t)\big(\delta b(t),y_{1}^{\varepsilon}(t),y_{1}^{\varepsilon}(t)\big)\chi_{E_{\varepsilon}}(t)
+\frac{1}{6}p_{3}(t)\big(y_{1}^{\varepsilon}(t),\delta b(t),y_{1}^{\varepsilon}(t)\big)\chi_{E_{\varepsilon}}(t)
\nonumber\\
& &
+\frac{1}{6}p_{3}(t)\big(y_{1}^{\varepsilon}(t),y_{1}^{\varepsilon}(t),\delta b(t)\big)\chi_{E_{\varepsilon}}(t)
+\frac{1}{6}q_{3}(t)\big(\delta \sigma(t),y_{1}^{\varepsilon}(t),y_{1}^{\varepsilon}(t)\big)\chi_{E_{\varepsilon}}(t)
\nonumber\\
& &
+\frac{1}{6}q_{3}(t)\big(y_{1}^{\varepsilon}(t), \delta \sigma(t),y_{1}^{\varepsilon}(t)\big)\chi_{E_{\varepsilon}}(t)
+\frac{1}{6}q_{3}(t)\Big(y_{1}^{\varepsilon}(t),y_{1}^{\varepsilon}(t),\delta \sigma(t)\Big)\chi_{E_{\varepsilon}}(t)
\nonumber\\
& &
+\frac{1}{6}p_{3}(t)\Big(\sigma_{x}(t)y_{1}^{\varepsilon}(t),\delta \sigma(t),y_{1}^{\varepsilon}(t)\Big)\chi_{E_{\varepsilon}}(t)
+\frac{1}{6}p_{3}(t)\Big(\sigma_{x}(t)y_{1}^{\varepsilon}(t),
y_{1}^{\varepsilon}(t),\delta \sigma(t)\Big)\chi_{E_{\varepsilon}}(t)
\nonumber\\
& &
+\frac{1}{6}p_{3}(t)\Big(y_{1}^{\varepsilon}(t),
\sigma_{x}(t)y_{1}^{\varepsilon}(t),\delta \sigma(t)\Big)
\chi_{E_{\varepsilon}}(t)
+\frac{1}{6}p_{3}(t)\Big(\delta \sigma(t),
\sigma_{x}(t)y_{1}^{\varepsilon}(t),
y_{1}^{\varepsilon}(t)\Big)\chi_{E_{\varepsilon}}(t)
\nonumber\\
& &
+\frac{1}{6}p_{3}(t)\Big(\delta\sigma(t),y_{1}^{\varepsilon}(t),
\sigma_{x}(t)y_{1}^{\varepsilon}(t)\Big)\chi_{E_{\varepsilon}}(t)
+\frac{1}{6}p_{3}(t)\Big(y_{1}^{\varepsilon}(t),\delta \sigma(t),
\sigma_{x}(t)y_{1}^{\varepsilon}(t)\Big)\chi_{E_{\varepsilon}}(t)
\nonumber\\
& &
+\frac{1}{6}p_{3}(t)\Big(\delta \sigma(t),
\delta \sigma(t),\gamma(t)\Big)\chi_{E_{\varepsilon}}(t)
+\frac{1}{6}p_{3}(t)\Big(\delta \sigma(t),\gamma(t),\delta \sigma(t)\Big)\chi_{E_{\varepsilon}}(t)
\nonumber\\
& &
+\frac{1}{6}p_{3}(t)\Big(\gamma(t),\delta \sigma(t),
\delta \sigma(t)\Big)\chi_{E_{\varepsilon}}(t)
\nonumber\\
& &
+\frac{1}{6}p_{3}(t)\Big(\delta \sigma(t),
\delta \sigma_{x}(t)y_{1}^{\varepsilon}(t),
y_{1}^{\varepsilon}(t)\Big)\chi_{E_{\varepsilon}}(t)
+\frac{1}{6}p_{3}(t)\Big(\delta \sigma(t),
y_{1}^{\varepsilon}(t),\delta \sigma_{x}(t)y_{1}^{\varepsilon}(t)\Big)\chi_{E_{\varepsilon}}(t)
\nonumber\\
& &
+\frac{1}{6}p_{3}(t)\Big(y_{1}^{\varepsilon}(t),\delta \sigma(t),
\delta \sigma_{x}(t)y_{1}^{\varepsilon}(t)\Big)\chi_{E_{\varepsilon}}(t)
+\frac{1}{6}p_{3}(t)\Big(\delta \sigma_{x}(t)y_{1}^{\varepsilon}(t),\delta \sigma(t),
y_{1}^{\varepsilon}(t)\Big)\chi_{E_{\varepsilon}}(t)
\nonumber\\
& &
+\frac{1}{6}p_{3}(t)\Big(\delta \sigma_{x}(t)y_{1}^{\varepsilon}(t),
y_{1}^{\varepsilon}(t),\delta \sigma(t)\Big)\chi_{E_{\varepsilon}}(t)
+\frac{1}{6}p_{3}(t)\Big(y_{1}^{\varepsilon}(t),\delta \sigma_{x}(t)y_{1}^{\varepsilon}(t),\delta \sigma(t)\Big)\chi_{E_{\varepsilon}}(t)
\nonumber\\
%
%
& &
+\frac{1}{24}p_{4}(t)\big(\delta\sigma(t),\delta\sigma(t),
y_{1}^{\varepsilon}(t),y_{1}^{\varepsilon}(t)\big)\chi_{E_{\varepsilon}}(t)
\nonumber\\
& &
+\frac{1}{24}p_{4}(t)\big(\delta\sigma(t),y_{1}^{\varepsilon}(t),\delta\sigma(t),
y_{1}^{\varepsilon}(t)\big)\chi_{E_{\varepsilon}}(t)
\nonumber\\
& &+\frac{1}{24}p_{4}(t)\big(\delta\sigma(t),y_{1}^{\varepsilon}(t),
y_{1}^{\varepsilon}(t),\delta\sigma(t)
\big)\chi_{E_{\varepsilon}}(t)\nonumber\\
& &
+\frac{1}{24}p_{4}(t)\big(y_{1}^{\varepsilon}(t),\delta\sigma(t),\delta\sigma(t),
y_{1}^{\varepsilon}(t)\big)\chi_{E_{\varepsilon}}(t)
\nonumber\\
& &+\frac{1}{24}p_{4}(t)\big(y_{1}^{\varepsilon}(t),\delta\sigma(t),
y_{1}^{\varepsilon}(t),\delta\sigma(t)
\big)\chi_{E_{\varepsilon}}(t)\nonumber\\
& &
+\frac{1}{24}p_{4}(t)\big(y_{1}^{\varepsilon}(t),y_{1}^{\varepsilon}(t),
\delta\sigma(t),\delta\sigma(t)\big)\chi_{E_{\varepsilon}}(t)
\Big]dt+o(\varepsilon^{2}),\qquad (\ephs\to 0^+)\\
&=&-\me\int_{0}^{T}\Big[
\mh(t,\bar{x}(t),u(t))+\inner{\mss(t,\bar{x}(t),u(t))}{\gamma(t)}\\
& &
+\frac{1}{2}\inner{\mt(t,\bar{x}(t),u(t))y_{1}^{\varepsilon}(t)}{y_{1}^{\varepsilon}(t)}
\Big]\chi_{E_{\varepsilon}}(t)dt+o(\varepsilon^{2}),\qquad (\ephs\to 0^+).
\end{eqnarray*}
This proves (\ref{shorttaylor}), and completes the proof of Proposition \ref{variational formulation for noncov}.



\begin{thebibliography}{1}

\bibitem{BellJa75}
{\sc D. J. Bell and D. H. Jacobson}, {\em Singular Optimal Control Problems},
Mathematics in Science and Engineering, vol. 117.
Academic Press, London-New York, 1975.


\bibitem{Bensoussan81}
{\sc A. Bensoussan}, {\em Lectures on stochastic control, in Nonlinear Filtering and Stochastic Control, pp. 1--62}, Lecture Notes in Math., vol. 972, Springer-Verlag, Berlin, 1981.

\bibitem{Bismut78}
{\sc J. M. Bismut}, {\em An introductory approach to duality in optimal stochastic
control}, SIAM Rev., 20 (1978), pp. 62--78.

\bibitem{fh} {\sc J. F. Bonnans and  A. Hermant}, {\em Second-order analysis for optimal control problems
with pure state constraints and mixed control-state constraints}, Ann. Inst. H. Poincar\'e Anal. Non Lin\'eaire, 26 (2009), pp. 561--598.

\bibitem{Bonnans12}
{\sc J. F. Bonnans and F. J. Silva}, {\em First and second order necessary conditions for stochastic optimal control problems}, Appl. Math. Optim., 65 (2012), pp. 403--439.

\bibitem{CA78}
{\sc D. J. Clements and B. D. O. Anderson}, {\em Singular Optimal Control: the Linear-Quadratic Problem},
Lecture Notes in Control and Information Sciences, vol. 5.
Springer-Verlag, Berlin-New York, 1978.

\bibitem{Peng97}
{\sc N. El Karoui, S. Peng, and M. C. Quenez}, {\em  Backward stochastic differential equations in fanance}, Math. Finance, 7 (1997), pp. 1--71.

\bibitem{FZZ15}
{\sc H. Frankowska, H. Zhang and X. Zhang}, {\em First and second order necessary conditions for stochastic optimal controls, the variational approach}, Preprint.

\bibitem{Frankowska13}
{\sc H. Frankowska and D. Tonon}, {\em Pointwise second-order necessary optimality conditions for the Mayer problem with control constraints}, SIAM, J. Control Optim., 51 (2013), pp. 3814--3843.

\bibitem{Gabasov72}
{\sc R. Gabasov and F. M. Kirillova},
{\em High order necessary conditions for optimality}, SIAM J. Control, 10 (1972), pp. 127--168.

\bibitem{Gabasov73} {\sc R. F. Gabasov and  F. M. Kirillova}, {\em Singular Optimal Controls}, Izdat. ``Nauka", Moscow, 1973.

\bibitem{Goh66}
{\sc B. S. Goh}, {\em Necessary conditions for singular extremals involving multiple control variables}, SIAM J. Control, 4 (1966), pp. 716--731.

\bibitem{Haussmann76}
{\sc U. G. Haussmann}, {\em General necessary conditions for optimal control of
stochastic systems}, Math. Prog. Study,  6 (1976), pp. 34--48.

\bibitem{h} {\sc D. Hoehener}, {\em Variational  approach to second-order optimality
conditions for  control  problems with pure state
constraints}, SIAM.  J. Control  Optim.,  50 (2012), pp. 1139--1173.

\bibitem{Knobloc} {\sc H.-W. Knobloch}, {\em Higher Order Necessary Conditions in Optimal Control Theory},
Springer-Verlag, Berlin-New York, 1981.

\bibitem{Kushner72}
{\sc H. J. Kushner}, {\em Necessary conditions for continuous parameter stochastic
optimization problems}, SIAM J. Control Optim., 10 (1972), pp. 550--565.

\bibitem{Lou10}
{\sc H. Lou}, {\em Second-order necessary/sufficient conditions for optimal control problems
in the absence of linear structure}, Discrete Contin. Dyn. Syst. Ser. B, 14 (2010), pp. 1445--1464.

\bibitem{MaYong99}
{\sc J. Ma and J. Yong}, {\em Forward-Backward Stochastic Differential Equations and Their Applications}, Lecture Notes in Math., vol. 1702, Springer-Verlag, Berlin, 1999.

\bibitem{Yong07}
{\sc L. Mou and J. Yong}, {\em A variational formula for stochastic controls and some
applications}, Pure Appl. Math. Q., 3 (2007), pp. 539--567.

\bibitem{Nualart06}
{\sc D. Nualart}, {\em The Malliavin Calculus and Related Topics}, Second edition,  Springer-Verlag, Berlin, 2006.

\bibitem{pz} {\sc Z. P\'ales and  V. Zeidan}, {\em Optimal control problems with set-valued control and
state constraints}, SIAM J. Optim., 14 (2003), pp. 334--358.

\bibitem{Peng90}
{\sc S. Peng}, {\em A general stochastic maximum principle for optimal control problems}, SIAM J. Control Optim., 28 (1990), pp. 966--979.

\bibitem{Pontryagin62}
{\sc L. S. Pontryagin, V. G. Boltyanskii, R. V. Gamkrelidze and E. F. Mishchenko},
{\em The Mathematical Theory of Optimal Processes}, John Wiley, New York, 1962.

\bibitem{Tang10}
{\sc S. Tang}, {\em A second-order maximum principle for singular optimal stochastic controls}, Discrete Contin. Dyn. Syst. Ser. B, 14 (2010), pp. 1581--1599.

\bibitem{w} {\sc J. Warga}, {\em A second-order condition that strengthens Pontryagin's maximum
principle}, J. Differential Equations, 28 (1978), pp. 284--307.

\bibitem{Yong99}
{\sc J. Yong and X.Y. Zhou}, {\em Stochastic Controls: Hamiltonian Systems and HJB
Equations}, Springer-Verlag, New York, Berlin, 2000.

\bibitem{ZhangSCM15}
{\sc H. Zhang and X. Zhang}, {\em Some results on pointwise second-order necessary conditions for stochastic optimal controls}, Sci. China Math., in press.

\bibitem{zhangH14a}
{\sc H. Zhang and X. Zhang}, {\em Pointwise second-order necessary conditions for stochastic optimal controls, Part I: The case of convex control constraint},
SIAM J. Control Optim., 53 (2015), pp. 2267--2296.


\end{thebibliography}
\end{document}